\documentclass[12pt]{amsart}
\usepackage{amssymb}

\newcommand{\PR}{\mathcal S}

\newcommand{\C}{{\mathcal C}}
\newcommand{\PSI}{{\mathbf\Psi}}

\newcommand{\tensor}{\mathop{\otimes}}
\newcommand{\tensorpow}{\otimes}
\newcommand{\dash}{\nobreakdash-\hspace{0pt}}
\newcommand{\Ob}{\mathrm{Ob}}
\newcommand{\YD}[1]{{{\vphantom{X}}_{#1}^{#1}\mathcal{YD}}}
\newcommand{\QYD}[1]{{ {\vphantom{X}}_{#1} \mathcal{QYD} }}
\newcommand{\coQYD}[1]{{ {\vphantom{X}}^{#1} \mathcal{QYD} }}

\newcommand{\Mod}[1]{ { {\vphantom{X}}_{#1} \mathcal{M} } }
\newcommand{\Comod}[1]{ { {\vphantom{X}}^{#1} \mathcal{M} } }
\newcommand{\D}{{\widetilde A}}
\newcommand{\RD}{\bar A}
\newcommand{\RHD}{{\mathcal H}}
\newcommand{\Dbl}{A}
\newcommand{\Uminus}{U^-}
\newcommand{\Uplus}{U^+}
\newcommand{\Iminus}{K_\Dbl^-}
\newcommand{\Iplus}{K_\Dbl^+}
\newcommand{\epsilonminus}{{\epsilon^-}}
\newcommand{\epsilonplus}{{\epsilon^+}}
\newcommand{\starV}{{\vphantom{V}^*V}}

\newcommand{\Ycr}{Y_{\PR}}
\newcommand{\RYcr}{\overline{Y}_{\PR}}
\newcommand{\lgen}{\text{$<$}}
\newcommand{\rgen}{\text{$>$}}

\newcommand{\Wor}{\mathrm{Wor}}
\newcommand{\field}{\Bbbk}
\newcommand{\Symm}{\mathbb{S}}
\newcommand{\coev}{\mathrm{coev}}
\newcommand{\lcprod}{\rtimes}
\newcommand{\rcprod}{\ltimes}
\newcommand{\act}{\mathop{\triangleright}}
\newcommand{\ract}{\mathop{\triangleleft}}
\newcommand{\qact}{{\mathop{>}}}
\newcommand{\Obstr}{\mathit{Obstr}}

\newcommand{\pr}{\mathrm{pr}_H}
\newcommand{\deriv}{\tilde\partial}
\newcommand{\Hom}{\mathrm{Hom}}
\newcommand{\hh}{\mathfrak{h}}

\DeclareMathOperator{\id}{id}
\DeclareMathOperator{\End}{End}

\theoremstyle{plain}
\newtheorem{theorem}{Theorem}[section]
\newtheorem*{theoremA}{Theorem A}
\newtheorem*{theoremB}{Theorem B}
\newtheorem*{theoremC}{Theorem C}
\newtheorem*{theoremD}{Theorem D}
\newtheorem*{theoremE}{Theorem E}
\newtheorem*{theoremF}{Theorem F}
\newtheorem*{theoremG}{Theorem G}
\newtheorem{proposition}[theorem]{Proposition}

\newtheorem*{problem*}{Problem}

\newtheorem{lemma}[theorem]{Lemma}
\newtheorem{corollary}[theorem]{Corollary}
\newtheorem*{corollary*}{Corollary}

\theoremstyle{definition}
\newtheorem{remark}[theorem]{Remark}
\newtheorem*{remark*}{Remark}
\newtheorem{definition}[theorem]{Definition}
\newtheorem*{definition*}{Definition}
\newtheorem{example}[theorem]{Example}
\newtheorem*{example*}{Example}

\advance\oddsidemargin by-0.5in
\advance\evensidemargin by-0.5in
\advance\textwidth by 1in

\begin{document}

\title{Braided doubles}

\author{Yuri Bazlov}
\address{School of Mathematical Sciences,
Queen Mary, University of London,
Mile End Road, London E1 4NS, UK}

\curraddr{Mathematics Institute,
University of Warwick,
Coventry CV4 7AL, UK}

\email{y.bazlov@warwick.ac.uk}

\author{Arkady Berenstein}
\address{Department of Mathematics, University of Oregon,
Eugene, OR 97403, USA} 
\email{arkadiy@math.uoregon.edu}
 

\pagestyle{myheadings}
\markboth{Y.~BAZLOV and A.~BERENSTEIN}{BRAIDED DOUBLES}

\begin{abstract}

We introduce and study  a large class of algebras with triangular
decomposition which we call \emph{braided doubles}. 
Braided doubles provide a unifying framework for classical
and quantum universal enveloping algebras and rational Cherednik
algebras. We classify braided doubles in terms of
\emph{quasi-Yetter-Drinfeld} (QYD) \emph{modules} 
over Hopf algebras which turn out to be a generalisation of the 
ordinary Yetter-Drinfeld modules. 
To each braiding (a solution to the braid equation) we associate a
QYD-module and the corresponding \emph{braided Heisenberg double} ---
this is a quantum deformation of the Weyl algebra where the role of
polynomial algebras is played by Nichols-Woronowicz algebras. 
Our main result is that any rational Cherednik algebra canonically
embeds into the braided Heisenberg double attached to the 
corresponding complex reflection group.

\end{abstract}

\maketitle
\setcounter{tocdepth}{1}
\begin{center}
\textsc{Introduction}
\end{center}

In the present paper we introduce and study  a large class of algebras
with triangular decomposition which we call \emph{braided doubles}. Our
approach is motivated by two recent developments in 
representation theory and quantum algebra:

\textbullet\quad The discovery by
Etingof and Ginzburg \cite{EG} of rational Cherednik algebras
$H_{t,c}(W)$  for an arbitrary complex reflection 
group $W$.  Similarly to enveloping algebras and their
quantum deformations, rational Cherednik algebras admit  a 
triangular decomposition $H_{t,c}(W)=S(\hh)\otimes \mathbb{C}
W\otimes S(\hh^*)$ (here $\hh$ is the reflection representation of
$W$). 

\textbullet\quad The emergence of the Fomin-Kirillov algebra as a
noncommutative model for the cohomology of the flag 
manifold \cite{FK}, and its interpretation by Majid \cite{Mnoncomm}
in terms of a  Nichols\dash Woronowicz algebra 
$\mathcal{B}_{S_n}$ attached to the symmetric group. 
The Fomin-Kirillov model was later 
generalised to all Coxeter groups $W$ by the first author \cite{B}, 
as a $W$\dash equivariant homomorphism $S(\hh)\to \mathcal{B}_W$, 
where $\mathcal{B}_W$ is the Nichols\dash Woronowicz algebra attached to $W$. 

Our first principal result, Theorem~\ref{th:emb0}, extends the above 
homomorphism $S(\hh)\to  \mathcal B_W$ to 
an embedding of the restricted Cherednik algebra
$\overline{H}_{0,c}(W)$ in what
we call \emph{a braided Heisenberg double} $\mathcal H_W$, 
which also has triangular 
decomposition  $\mathcal H_W=\mathcal B_W\otimes \mathbb{C} W \otimes
\mathcal B_W$. 
For nonzero $t$, such an embedding of $H_{t,c}(W)$ is obtained by
replacing $\mathcal B_W$ with its
deformation $\mathcal {B}_{W,t}$. 
We thus find a new, quantum group\dash like realisation of each
rational Cherednik algebra.

The above has prompted us to look for a framework in which
both the enveloping algebras (and quantum groups) and the rational
Cherednik algebras could be uniformly treated. 
This is precisely the framework of
\emph{braided doubles}, where the aforementioned
objects fit into a general class of algebras with triangular
decomposition $A=U^-\otimes H\otimes U^+$ over a Hopf algebra $H$,
such that the algebras $U^-$, $U^+$ are generated by dually paired $H$\dash
modules $V$, $V^*$ and the commutator of $V$ and $V^*$ in $A$ lies in~$H$.

Surprisingly, we have been able to completely classify (Theorem
~\ref{thm_qYD}) all \emph{free} braided 
doubles  in terms of \emph{quasi\dash
Yetter\dash Drinfeld modules}, 
which are a generalisation of Yetter\dash Drinfeld modules
\cite{Y,Mdoubles}.   
Our quasi\dash Yetter\dash Drinfeld modules
turn out to have a natural interpretation in terms of monoidal
categories. Using a variant of the Tannaka\dash Krein duality, we
prove in Section~\ref{sect:structural} that a set $\Pi$ of
\emph{compatible braidings} on a vector space $V$ turns $V$ into a
quasi\dash YD module over a 
certain Hopf algebra $H_\Pi$, hence yields a free braided double of the form
$T(V)\tensor H_\Pi \tensor T(V^*)$.

Braided doubles are such quotients of free braided doubles that still
admit triangular decomposition. The most interesting are the 
\emph{minimal doubles}. 
For a quasi\dash Yetter\dash
Drinfeld module $V$ over a Hopf algebra $H$, we describe
(Theorem~\ref{ker_qbfact}) the relations in the corresponding minimal
double $\RD(V)$, implicitly as 
kernels of  \emph{quasibraided factorials} on $T(V)$ and $T(V^*)$ 
given in terms of the quasi\dash YD structure.  If $V$ is a
Yetter\dash Drinfeld module, $\RD(V)$ is a braided Heisenberg double,
which factorises into $H$ and two dually paired Nichols\dash
Woronowicz algebras (Theorem \ref{th:Nichols}). 
 Prominent examples of
minimal doubles are the universal enveloping algebra $U(\mathfrak g)$ and
its quantisation $U_q(\mathfrak g)$; the relations in minimal
doubles are therefore a (vast) generalisation of the
Serre relations.

Finally, we discover that any quasi\dash YD module can be obtained as
a certain sub-quotient of a Yetter\dash Drinfeld module
(Theorem~\ref{th:perfect}). In interesting cases, this allows us to
embed a minimal double in a braided Heisenberg double. We put this
observation to use when we classify braided doubles $U^-\tensor
\field G \tensor U^+$ over group algebras, where $U^+$ and $U^-$ are
commutative. The outcome of the classification is rational Cherednik
algebras; this is how the motivating results, described in the
beginning of this Introduction, naturally re-emerge in the braided doubles
setup. 
 An immediate consequence of the theory is the 
\emph{PBW theorem} for rational
Cherednik algebras over an arbitrary field
--- a crucial property which has so far been known only in
characteristic zero.

\subsection*{Acknowledgments}
We thank Pavel Etingof, Victor Ginz\-burg, Victor Kac, Shahn Majid
and Catharina Stroppel for useful conversations. We acknowledge
partial support of the EPSRC grant GR/S10629/01 (Y.B.) and the NSF grant
DMS-0501103 (A.B.).

\tableofcontents

\section{Overview of main results}
\label{sect:overview}

In this Section we state and discuss the main
results of the paper. Details and proofs will be given in Sections
\ref{sect:structural}--\ref{sect:cherednik}. We assume that
the reader is familiar with the basics of the theory of Hopf algebras,
for which \cite{Sw} is one of the standard references.

\subsection*{Notation} Throughout the paper, $\field$ is the ground
field (of arbitrary characteristic). Vector spaces, tensor products,
(bi)algebras and Hopf 
algebras are  over $\field$. The tensor algebra of a vector space $V$
is denoted by $T(V)$; it has grading 
$T(V)=\oplus_{n\ge 0} V^{\tensorpow n}$. 
We use Sweedler\dash type notation (without the summation sign, 
\cite[1.4.2]{Mon}): 
if $H$ is a bialgebra, the 
coproduct of $h\in H$ is denoted by $h_{(1)}\tensor h_{(2)}\in
H\tensor H$; 
a left coaction $\delta \colon V\to H\tensor V$ of $H$ on a space $V$
is denoted by $v\mapsto v^{(-1)} \tensor v^{(0)}$. 
By writing $\delta(v)=v^{[-1]}\tensor v^{[0]}$, we imply that
$\delta$ is not a coaction but just a linear map $V\mapsto H\tensor V$
(referred to in the paper as \emph{quasicoaction}). 
The symbols $\act$ and $\ract$ mean left, resp.\ right, action of
a bialgebra.  The counit of a bialgebra $H$ is denoted by
$\epsilon\colon H \to \field$. 

If $U^-$ (resp.\ $U^+$) is a left
(resp.\ right) module algebra for $H$, the corresponding semidirect
product is denoted by $U^-\lcprod H$ (resp.\ $H\rcprod U^+$), see
\cite[Section 7.2]{Sw}.

Finally, if $H$ is a Hopf algebra, then $S\colon
H\to H$ denotes the antipode of $H$. 
All Hopf algebras are assumed to have
\emph{bijective} antipode.  
Theorem A below holds when $H$ is any bialgebra; in Theorems B--G, we assume
$H$ to be a Hopf algebra.  

\subsection{A problem in deformation theory}
\label{defproblem}

Let $V$ be a finite\dash dimensional space over $\field$ with left
action, $\act\colon H\tensor V \to V$, of a bialgebra $H$. 
The dual space $V^*$ is canonically a right $H$\dash module,
with right action $\ract \colon V^* \tensor H \to V^*$ defined by 
$$
   \langle f\ract h, v\rangle = \langle f, h\act v\rangle, \quad f\in V^*, \ h\in H, \
   v\in V,
$$
where $\langle\cdot,\cdot \rangle$ is the pairing between $V^*$ and $V$. 

To every linear map $\beta\colon V^*\tensor V \to H$ (a
\emph{bialgebra\dash valued pairing}) there corresponds an
associative algebra $\D_\beta$, generated by the spaces $V$, $V^*$ and
the algebra $H$ subject to the relations 
$$
    fh = h_{(1)} (f\ract h_{(2)}),\quad 
    hv = (h_{(1)} \act v) h_{(2)},\quad
    [f,v]=\beta(f,v) \in H, 
$$  
where $f\in V^*$, $h\in H$, $v\in V$. 
Here and below, $[f,v]$ denotes the commutator $fv-vf$.
It is clear from the defining relations that the map  
$$
  m_\beta \colon T(V)\tensor H \tensor T(V^*) \to \D_\beta,
$$
of vector spaces, 
given by multiplication of generators in $\D_\beta$, is surjective. We
say that $\D_\beta$ has \emph{triangular decomposition over $H$}, if
$m_\beta$ is one\dash to\dash one. We will indicate this by writing
\begin{equation*}
          \D_\beta = T(V)\lcprod H \rcprod T(V^*)\,;
\end{equation*}
observe that the subalgebras $T(V)\lcprod H$ and $H \rcprod T(V^*)$ of
$\D_\beta$ are indeed
semidirect products with respect to the action of $H$, which extends from
$V$ to $T(V)$ (resp.\ from $V^*$ to $T(V^*)$) via the coproduct in $H$.  

The algebra $\D_0$ can be shown to have triangular decomposition. 
Algebras $\D_\beta$ may be viewed as deformations of $\D_0$, with
parameter $\beta$ which takes values in $\mathrm{Hom}_\field
(V^*\tensor V, H)$. 
Triangular decomposition means that $\D_\beta$ is a flat
deformation of $\D_0$. 
Our first principal result (which appears as  Theorem \ref{thm_qYD} in
Section~\ref{sect:braideddoubles}) describes all values of $\beta$ for
which the deformation is flat: 
\begin{theoremA}
The algebra $\D_\beta$ has triangular decomposition
over the bialgebra $H$, if and only if the $H$\dash valued pairing $\beta\colon
V^*\tensor V \to H$ satisfies 
\begin{equation*}
     h_{(1)}\, \beta(f\ract h_{(2)}, v) = \beta(f,h_{(1)} \act v) \,
     h_{(2)}  
\label{eqA}
\tag{A}
\end{equation*}
for all $f\in V^*$, $v\in V$ and $h\in H$.
\end{theoremA}
\begin{remark*}
Observe that equation (\ref{eqA}) is necessary for $\D_\beta$ to have
triangular decomposition, because of obstruction in
degree $3$; the product $fhv$ of three generators $f\in V^*$, $h\in H$
and $v\in V$, can be expanded in two ways which must coincide:
\begin{equation*}
0 = (fh)v -f(hv) 
= h_{(1)} \beta(f\ract h_{(2)}, v) - \beta(f,h_{(1)}\act v)h_{(2)} 
\ \in H \hookrightarrow  \D_\beta.
\end{equation*}
\end{remark*}

\begin{remark*}
When $H=\field G $ is a group algebra of a group $G$, equation (A)
means that $\beta\colon V^*\tensor V \to \field G$ is a $G$\dash
equivariant map, where the action of $g\in G$ on $V^*\tensor V $ is
given by
$g(f\tensor v):=f\ract g^{-1} \tensor g\act v$, and the $G$\dash
action on $\field G$ is the adjoint one. In other words, equation (A)
is precisely what allows
us to extend the action of $G$ from each of the factors $T(V)$,
$\field G$, $T(V^*)$ in the triangular decomposition to a covariant
$G$\dash action on the whole
algebra $\D_\beta$. 

For a Hopf algebra $H$, one shows that (under mild
technical assumptions) the $H$\dash action
extends in this way to a covariant $H$\dash action on the algebra
$\D_\beta$ if and only if $H$ is cocommutative.
We would thus like to warn the reader that in general, algebras
$\D_\beta$ have no natural covariant action of $H$ and \emph{cannot} be
viewed as algebras in the category of $H$\dash modules. 
\end{remark*}

We will now make the above deformation problem harder by assuming 
additional relations, not necessarily quadratic, between the 
elements of $V$ (resp.\ $V^*$). Let $I^-\subset T^{>0}(V)$, 
$I^+\subset T^{>0}(V^*)$ be two\dash sided ideals.
The algebra $\D_\beta/\lgen I^-,I^+ \rgen$ is said to have triangular
decomposition over $H$,  if the natural linear map 
$$
T(V)/I^- \tensor H \tensor T(V^*)/I^+ \ \to\ 
\D_\beta/\lgen I^-,I^+\rgen
$$ 
is bijective. (Angular brackets denote a two\dash sided ideal
with given generators.)

Once again, algebras $\D_\beta/\lgen I^-, I^+\rgen$ with
triangular decomposition are flat deformations of $\D_0 /\lgen I^-,
I^+\rgen$. But now, instead of looking for the values of 
$\beta\in \mathrm{Hom}_\field(V^*\tensor V,$ $ H)$ 
which guarantee flatness, we pose an inverse problem: 
\begin{problem*}
For a given bialgebra\dash valued pairing
$\beta\colon V^*\tensor V \to H$, 
describe all possible ideals $I^-\subset T^{>0}(V)$, $I^+\subset
T^{>0}(V^*)$ of relations such that 
the algebra $\D_\beta /\lgen I^-,I^+\rgen$
has triangular decomposition over $H$.
\end{problem*}

To attack this deformation problem, we introduce and study 
quasi\dash Yetter\dash Drinfeld  modules. 

\subsection{Quasi-Yetter-Drinfeld modules}
\label{q-y-d mod}

The following observation is crucial for the theory of braided doubles
developed in the present paper: equation (\ref{eqA}) appears in the
definition of a \emph{Yetter\dash Drinfeld module}
over a bialgebra $H$. There is, however, an extra ingredient in that
definition, which we do not have in our picture.  

We will now define finite\dash dimensional Yetter\dash Drinfeld
modules in a way 
different from (but equivalent to) what is usually seen in the quantum
groups literature, and will introduce their generalisation called
\emph{quasi\dash Yetter\dash Drinfeld modules}. 
Note that the space $V^*\tensor V$ has a standard structure of a
coalgebra, dual to the algebra $\End(V)\cong V\tensor V^*$.
\begin{definition*}
A \emph{quasi\dash Yetter\dash Drinfeld module}
over a bialgebra $H$ is a finite\dash
dimensional space $V$ with the following structure:
\begin{itemize}
\item[-- ] left $H$\dash action $\act$;
\item[-- ] linear map $\beta\colon V^*\tensor V \to H$, which satisfies
(\ref{eqA}).
\end{itemize}
\end{definition*}
\begin{definition*}
A \emph{Yetter\dash Drinfeld module} over $H$ is a quasi\dash Yetter\dash
Drinfeld module where the map $\beta$ is a morphism of coalgebras. 
\end{definition*}
Yetter\dash Drinfeld modules over a bialgebra $H$ were introduced by
Yetter in \cite{Y} as ``crossed bimodules'', 
and were shown by Majid \cite{Mdoubles} to be the
same as modules over the Drinfeld quantum double $D(H)$ when $H$ is a
finite\dash dimensional Hopf algebra. 
It can be said that Yetter\dash Drinfeld modules' raison
d'\^etre is their relationship with \emph{braidings}. 
A Yetter\dash Drinfeld module structure on the space $V$ gives rise to a map
$$
   \Psi\colon V\tensor V \to V\tensor V, \quad \Psi(v\tensor w) =
   \beta(f^a,v)\act w \tensor v_a,
$$
which is a braiding, i.e., a solution to the 
quantum Yang\dash Baxter equation
$\Psi_{12}\Psi_{23}\Psi_{12}$ $=$ $\Psi_{23}\Psi_{12}\Psi_{23}$. 
(Here $\{f^a\}$, $\{v_a\}$ denote a pair of dual bases of $V^*$, $V$;
summation over the index $a$ is implied.) Moreover, Yetter\dash
Drinfeld modules over a Hopf algebra form a \emph{braided monoidal
  category} (see a survey in \cite[4.3]{CGW}).

 Traditionally, in the definition of Yetter\dash
Drinfeld module over a bialgebra $H$ 
the $H$\dash valued pairing  between $V^*$ and $V$ is encoded by a
linear map $V$ to $H\tensor V$:
\begin{equation*}
    \beta\colon V^* \tensor V \to H \qquad \leadsto \qquad
    \delta=\delta_\beta\colon V \to H \tensor V, 
\quad
     \delta(v) = \beta(f^a,v)\tensor v_a.
\end{equation*}
The Yetter\dash Drinfeld condition translates in terms of $\delta$ 
into a formula with two levels of Sweedler notation, 
see Definition~\ref{def:qydmod}. 
Moreover, $\beta$ is a coalgebra morphism if and only if 
$\delta$ is a coaction of $H$. 
Dropping the coaction condition leads to the class of  quasi\dash
Yetter\dash Drinfeld modules.
We will think of quasi\dash YD modules for a bialgebra $H$ as pairs
$(V,\delta)$, where $V$ is an $H$\dash module and
$\delta\in\Hom_\field(V,H\tensor V)$ is a \emph{Yetter\dash Drinfeld
quasicoaction}.

The original motivation for the quasi\dash YD modules was the deformation
problem given above. However, Section~\ref{sect:structural} of the present
paper  treats them from a categorical
viewpoint, drawing a parallel with Yetter\dash Drinfeld modules. In
particular, quasi\dash YD modules over a Hopf algebra form  what we
call a \emph{semibraided monoidal category}. A converse is also true:
a given semibraided category can be realised as quasi\dash YD modules
over some Hopf algebra, reconstructed from the category. We present
the reconstruction process as a form of the Tannaka\dash Krein duality. 

Unlike for braidings, there is no canonical notion of a semibraiding
on a vector space, which would lead to a realisation of such space as
a quasi\dash YD module. Nevertheless, we show in~\ref{subsect:acons}
that if $V$ is equipped with a finite set $\Pi$ of braidings which are
pairwise \emph{compatible}, then $V$ is canonically an object in a semibraided
category, hence a quasi\dash Yetter\dash Drinfeld module for a certain
Hopf algebra $H_\Pi$. 

A basic example of a set of compatible braidings and a quasi\dash YD
module is as
follows. Let $(V, \beta\colon V^*\tensor V \to H)$ be a Yetter\dash
Drinfeld module over a cocommutative Hopf algebra $H$, with induced
braiding $\Psi$. Then $\Pi=\{\Psi,\tau\}$ is a set of compatible
braidings, where $\tau(v\tensor w) = w\tensor v$ is the trivial
braiding on $V$. The space $V$ can be made a 
quasi\dash Yetter\dash Drinfeld  module over $H$ via the new $H$\dash
valued pairing $\beta_{\Psi,\lambda\tau}\colon V^* \tensor V \to H$, defined by
$\beta_{\Psi,\lambda\tau}(f,v)=\beta(f,v)+\lambda\langle f,v\rangle$ for any
scalar $\lambda$.

\subsection{Braided doubles}
\label{bdintro}
We are now ready to give the 
\begin{definition*}
In the notation as above, an algebra $\D_\beta/\lgen I^-,I^+\rgen$
with triangular decomposition $T(V)/I^- \lcprod H \rcprod T(V^*)/I^+$ 
over the bialgebra $H$ is called a \emph{braided double}.
\end{definition*}
Thus, by definition, braided doubles are the same as solutions to the
deformation theory problem posed in \ref{defproblem}.

The algebras $\D_\beta$ with triangular decomposition 
will now be referred to as \emph{free} braided doubles. 
Theorem~A means that free braided doubles are parametrised by 
quasi\dash Yetter\dash Drinfeld modules. 
 Instead of $\D_\beta$ we write $\D(V,\delta)$ for a free braided
double associated to the quasi\dash YD module $(V,\delta)$.

The following Example demonstrates how one\dash dimensional
quasi\dash Yetter\dash Drinfeld modules lead to
interesting algebraic objects already at the level of 
free braided doubles.
\begin{example*}
We show in~\ref{ex:trivial} that all 
one\dash dimensional quasi\dash YD modules over a cocommutative Hopf
algebra $H$ are 
 of the form $V_{\alpha,p}$, where $\alpha\colon H\to
\field$ is an algebra homomorphism 
and $p$ is any central element of $H$; one has 
$h\act v=\alpha(h)v$ and $\delta(v)=p\tensor v$ for $h\in H$, $v\in
V_{\alpha,p}$. 
Let $H=S(\mathfrak h)$ be the algebra of
polynomials over a vector space $\mathfrak h$, with Hopf structure
given by coproduct $\Delta h = h\tensor 1 + 1\tensor h$ for $h\in
\mathfrak h$. Consider any quasi\dash YD module $(V,\delta)$ 
which is a direct sum of one\dash dimensional quasi\dash YD modules:
$V = V_{\alpha_1,p_1} \oplus \dots \oplus V_{\alpha_m, p_m}$ where 
$\alpha_i\in {\mathfrak h}^*$, and 
$p_i$ are arbitrary polynomials in $S(\hh)$. 
Let $\{f_i\}$, $\{e_i\}$ be dual bases of $V^*$, $V$ such that $e_i\in
V_{\alpha_i,p_i}$. 
The free braided double $\D(V,\delta)$ is given
by generators and relations
$$
   [h,e_i] = \alpha_i(h) e_i, \ 
   [h,f_i] = -\alpha_i(h) f_i, \ h\in \mathfrak h; \quad 
   [f_i, e_i] = p_i. 
$$
If the $p_i$ are chosen to be in $\hh$, and $\alpha_i$ are a
basis of ${\mathfrak h}^*$ related to $p_i$ via a
generalised Cartan matrix, the algebras with triangular decomposition
thus obtained are 
$\widetilde U(\mathfrak g)$, the Kac\dash Moody universal enveloping
algebras before quotienting by the Serre
relations. The Serre relations arise in the context of minimal
doubles (see below). 

In the simplest case $m=\dim \mathfrak h = 1$, we obtain a Smith algebra (a 
``polynomial deformation of $\mathit{sl}(2)$''), considered in 
\cite{Smith} (and earlier in a different form in \cite{Joseph}). 
These algebras have a notion of 
highest weight modules and an analogue of the BGG category $\mathcal
O$; those are, of course, consequences of the triangular
decomposition. Polynomial deformations of $\mathit{sl}(2)$  
have physical applications in quantum mechanics, conformal field
theory, Yang-Mills-type gauge theories, inverse scattering, quantum
optics \cite{BBD}.   
\end{example*}

After this Example, let us move on to the case of nonzero ideals $I^\pm$.

\subsection{Minimal doubles}

A \emph{minimal double} $\RD(V,\delta)$ is a quotient of the free
double $\D(V,\delta)$ by largest
ideals $I(V,\delta)\subset T^{>0}(V)$ and $I^*(V,\delta)\subset
T^{>0}(V^*)$ such 
that the quotient still has the triangular decomposition property.  

Minimal doubles the most interesting braided doubles; 
they have the largest set of relations. Results of Section~4 of the
present paper imply
\begin{theoremB}
1. Any braided double has a unique minimal double as a quotient double.  

2. The ideals $I(V,\delta)\subset T(V)$ and $I(V^*,\delta)\subset
   T(V^*)$ are graded, and are given by
$$
I(V,\delta) = \oplus_{n\ge 1} \ker\widetilde{[n]}!_\delta, 
\qquad
I(V^*,\delta) = \oplus_{n\ge 1} \ker\widetilde{[n]}!_{\delta_r}
$$
where $\widetilde{[n]}!_\delta \colon V^{\tensorpow n}\to
   (H\tensor V)^{\tensorpow n}$ and 
 $\widetilde{[n]}!_{\delta_r} \colon V^{*\tensorpow n}\to
   (V^*\tensor H)^{\tensorpow n}$ are \emph{quasibraided factorials}, 
which arise from the quasi\dash Yetter\dash Drinfeld structure
$\delta$ on $V$. 
\end{theoremB}
To each quasi\dash Yetter\dash Drinfeld module $(V,\delta)$ over a
Hopf algebra $H$, Theorem~B associates two graded algebras, generated
in degree one: 
\begin{equation*}
   U(V,\delta)=T(V)/I(V,\delta), \quad
   U(V^*,\delta)=T(V^*)/I(V^*,\delta),
\end{equation*}
such that the minimal double has triangular decomposition 
\begin{equation*}
\RD(V,\delta) = U(V,\delta) \lcprod H \rcprod U(V^*,\delta).
\end{equation*}
Theorem~B is formally an answer to the deformation problem posed
in~\ref{defproblem}, however, the relations in the algebras
$U(V,\delta)$, $U(V^*,\delta)$ are given only implicitly by the kernels of 
quasibraided factorials (introduced in Definition
\ref{def_qbfact}). The latter might not be well suited for
computational purposes: these operators may have values in
infinite\dash dimensional spaces.

In Section 4, we also point out a sufficient condition for
minimality of a braided double. A braided double with 
triangular decomposition of the form
$T(V)/I^-\lcprod H \rcprod T(V^*)/I^+$, gives rise to an $H$\dash
valued \emph{Harish\dash Chandra pairing} 
between the algebras $T(V^*)/I^+$
and $T(V)/I^-$: the pairing $(b,\phi)_H$ 
is the product $b\phi$ of $b\in T(V^*)/I^+$ and
$\phi\in T(V)/I^-$ in the braided double, projected onto $H$. One has
\begin{theoremC}
A braided double is minimal if its Harish\dash Chandra pairing is
nondegenerate. 
\end{theoremC}
For example, the universal enveloping algebra $U(\mathfrak g)$ is a minimal
double, so that the kernels of the corresponding quasibraided
factorials come out as the Serre relations. 
The converse of Theorem~C is not true (Example \ref{ex:nilpotent}),
and is disproved using a counterexample to third Kaplansky's
conjecture on Hopf algebras.

Non\dash degeneracy of the Harish\dash Chandra pairing is a
property which strongly influences the algebra structure of a braided
double. As an example of this, a simple argument in
Proposition~\ref{good_ideals} shows that if 
the scalar\dash valued pairing $\epsilon((\cdot,\cdot)_H)$ in a
braided double $\RD(V)$ is non\dash degenerate, then any two\dash
sided ideal in $\RD(V)$ has a nontrivial projection which is an ideal
in $H$. 

An important class of braided doubles with such a nondegeneracy property are
\emph{braided Heisenberg doubles}.
These are precisely the minimal doubles $\RD(V)$ which correspond to
Yetter\dash Drinfeld modules $V$.

\subsection{Braided Heisenberg doubles and Nichols-Woronowicz algebras}

``Honest'' Yetter\dash Drinfeld modules are obviously a distinguished
 class of quasi\dash Yetter\dash Drinfeld modules. 
Section~\ref{sect:bhd} of the paper describes minimal doubles
associated to this class. They are called
braided Heisenberg doubles. The defining ideals in a braided
 Heisenberg double are expressed in terms of the braiding on the
 Yetter\dash Drinfeld module:
\begin{theoremD}
Let $V$ be a Yetter\dash Drinfeld module over a Hopf algebra
$H$. Denote by $\Psi$ the induced braiding on $V$. The minimal double 
associated to $V$ has triangular decomposition of the form 
$$
    \mathcal H_V \cong \mathcal B(V, \Psi)\lcprod H \rcprod \mathcal
    B(V^*,\Psi^*), 
$$
where $\mathcal B(V,\Psi)$ is the \emph{Nichols\dash Woronowicz algebra}
of a braided space $(V,\Psi)$.
\end{theoremD}
Nichols\dash Woronowicz algebras 
are a remarkable class of \emph{braided Hopf algebras}, and  are
quantum analogues of symmetric and exterior algebras.
One has $\mathcal B(V, \Psi) = T(V)/\ker\Wor(\Psi)$ where $\Wor(\Psi)$
is the Woronowicz symmetriser associated with the braiding $\Psi$ on
$V$. The Woronowicz symmetriser, introduced in \cite{Wo}, appears in
our setting as a specialisation of a more general quasibraided
factorial to the case of a Yetter\dash Drinfeld module.

The symmetric (resp.\ exterior) algebra of $V$ is $\mathcal B(V,
\tau)$ (resp.\ $\mathcal B(V, -\tau)$) where $\tau(v\tensor
w)=w\tensor v$ is the trivial braiding on $V$. Note that if $V$ is a
trivial Yetter\dash Drinfeld module over $H=\field$, $\mathcal
H_V$ is the Heisenberg\dash Weyl algebra $S(V)\tensor S(V^*)$. 

Algebras $\mathcal B(V, \Psi)$ were formally introduced by
Andruskiewitsch and Schneider in \cite{AS1} (as `Nichols
algebras' honouring an earlier work of Nichols \cite{N}) and 
coincide with quantum exterior algebras of Woronowicz
\cite{Wo}. 
In the present form of two dually paired algebras, they
appeared in the work of Majid \cite{Mcalculus}. Nichols\dash
Woronowicz algebras are the same as ``quantum shuffle algebras'' of
Rosso \cite{R}.
These algebras have already been linked to a number of different
areas, such as pointed Hopf algebras \cite{AS2} and
noncommutative differential geometry \cite{Wo,Mnoncomm,KMexterior}, 
and have led to a useful generalisation of root systems due to
Heckenberger \cite{H}. (In Section~\ref{sect:bhd}, we use Nichols\dash Woronowicz algebras 
to give a new and simple counterexample to the aforemetioned third
conjecture of Kaplansky.)
Our approach thus leads to a surprising appearance of  Nichols\dash
Woronowicz algebras in
deformation theory; the braided coproduct on $\mathcal B(V,\Psi)$ is
now recast as the product in the graded\dash dual algebra $\mathcal
B(V^*,\Psi^*)$, and the braided Hopf algebra property is encoded in
the commutation relation between 
$\mathcal B(V,\Psi)$ and $\mathcal B(V^*,\Psi^*)$ and the associativity of
multiplication in the minimal double.

In Section~\ref{sect:bhd} we also consider
quasi\dash Yetter\dash Drinfeld modules $V$ with structure given by
compatible braidings (see~\ref{q-y-d mod}). 
In the corresponding minimal double, 
the formula for the defining ideals 
is more involved and leads to a generalisation of
Nichols\dash Woronowicz algebras associated to a set of compatible
braidings (instead of just one braiding); 
but the degree $2$ part of the formula is still quite manageable:  
\begin{theoremE}
Let $\delta_k\colon V \to H\tensor V$, $k=1,2,\dots,N$, be
Yetter\dash Drinfeld 
coactions on an $H$\dash module $V$, which induce braidings
$\Psi_k$ on $V$. Let $t_k$ be generic coefficients (e.g., formal
parameters). Define the quasi\dash YD module structure on $V$ by
putting $\delta = \sum_k t_k \delta_k$. Then the defining ideals in the
corresponding minimal quadratic double  are
$$
I_{\mathit{quad}}(V) = \lgen \bigcap_{k=1}^N \ker(\id+\Psi_k)
\rgen, 
\quad
I^*_{\mathit{quad}}(V) = \lgen \bigcap_{k=1}^N \ker(\id+\Psi^*_k) \rgen.
$$
\end{theoremE}

\subsection{Perfect subquotients}

In Section~\ref{sect:perfect}, we justify our earlier claim
that Yetter\dash Drinfeld modules are a ``basic family'' of solutions
of the deformation problem set out in~\ref{defproblem}. 

First of all,
we define a special class of morphisms (called \emph{subquotients}) 
between quasi\dash Yetter\dash
Drinfeld modules $V$, $W$ over a Hopf algebra $H$. 
Subquotients are diagrams $V \to W \to V$ where the arrows are
$H$\dash module homomorphisms, and satisfy a certain condition of
compatibility with quasi\dash Yetter\dash Drinfeld structures,
$\delta_V$ on $V$ and $\delta_W$ on $W$. We show
that subquotients $V\to W \to V$ are the same as \emph{triangular
  morphisms} between free braided doubles $\D(V,\delta_V)$ and $\D(W,\delta_W)$, which
are the precisely the morphisms in the category $\mathcal D_H$  of braided
doubles over $H$.

One can observe that, if $(W,\delta_W)$ is a quasi\dash YD module for
$H$ and $V$ is an $H$\dash module, then each pair $V\xrightarrow{\mu}
W \xrightarrow{\nu} W$ of $H$\dash module maps defines a unique quasi\dash Yetter\dash Drinfeld
structure $\delta_V$ on $V$. In this situation we say that $V$ is a
subquotient of $W$ via the maps $\mu$, $\nu$. Recall the left\dash
side  defining
ideal $I(V,\delta_V)\subset T(V)$ of the minimal double associated to the quasi\dash
Yetter\dash Drinfeld module $V$. We show
(Proposition~\ref{prop:preimage}) that 
$$
I(V,\delta_V) \supseteq \mu^{-1} ( I(W,\delta_W) ).
$$
If this inclusion is in fact an equality, leading to an embedding
$U(V,\delta_V) \hookrightarrow U(W,\delta_W)$ of graded algebras,
we say that the quasi\dash
YD module $(V,\delta_V)$ is a \emph{perfect subquotient} of
$(W,\delta_W)$. 
We prove

\begin{theoremF}
Every quasi\dash YD module can be obtained as a perfect subquotient of a
Yetter\dash Drinfeld module.
\end{theoremF}

See Theorem~\ref{th:perfect}. Note that, given a finite\dash
dimensional quasi\dash YD module $V$, one needs additional
assumptions to guarantee  that $V$ can be realised as a subquotient of
a \emph{finite\dash dimensional} Yetter\dash Drinfeld module $Y$.

However, this is not yet the main problem. From the point of view of
braided doubles, one would hope to find a perfect subquotient
$V\xrightarrow{\mu} Y \xrightarrow{\nu} V$ such that  $V^*
\xrightarrow{\nu^*} Y^* \xrightarrow{\mu^*}V^*$ is also a perfect
subquotient. This would yield an embedding 
$$
     \RD(V,\delta_V) = U(V,\delta_V) \lcprod H \rcprod
     U(V^*,\delta_V) \ \hookrightarrow \ \mathcal{H}_Y
$$
of a given minimal double into a braided Heisenberg double. 
But in general, we do not know what conditions $(V,\delta_V)$ should satisfy so
that such two simultaneous perfect subquotients exist. 

Nevertheless, we find and study a particular type of quasi\dash Yetter\dash
Drinfeld modules $(V,\delta_V)$ such that $\RD(V,\delta_V)$ embeds
into a braided Heisenberg double. This happens for 
$\RD(V,\delta_V)$ which are \emph{rational Cherednik algebras}.

\subsection{Rational Cherednik algebras}

 Let $G\le \mathit{GL}(V)$ be a finite linear group over $\field$. 
A rational Cherednik algebra of $G$ is a flat deformation
of the semidirect product algebra $S(V\oplus V^*)\lcprod \field G$,
obtained by replacing the right\dash hand side of the relation
$[f,v]=0$ ($f\in V^*$, $v\in V$) with some $\field G$\dash valued
pairing between $V^*$ and $V$. Clearly, rational Cherednik algebras are braided doubles over
$\field G$; as such, they become the subject of our inquiry 
in the last section of the paper. 

The problem we pose in Section~\ref{sect:cherednik} is to classify
braided doubles $A=\Uminus \lcprod \field G \rcprod \Uplus$, associated to the
$G$\dash module $V$,  such that $U^\pm$ are commutative
algebras. By Theorem~B, we need to find 
quasi\dash Yetter\dash Drinfeld structures $\delta$ on $V$ such that
the algebras $U(V,\delta)$ and $U(V^*,\delta)$ are commutative. 
Such $\delta$ parametrise rational Cherednik algebras over the group
$G$. Analysing quasibraided factorials, we write down all such
quasi\dash YD structures $\delta$ for an irreducible linear group $G$ 
in terms of \emph{complex reflections} (elements $s$ such that
$\mathrm{rank}(s-1)=1$, otherwise called pseudoreflections) in $G$.
A rational Cherednik algebra $H_{t,c}(G)$ has parameters $t\in \field$
and $c\in \field(\PR)^G$, where $\PR$ is the set of complex reflections.   

Our method is independent of the characteristic of the ground field
$\field$. Rational 
Cherednik algebras over $\field=\mathbb C$ are already known by the
Etingof\dash Ginzburg classification \cite{EG}, a new proof of which we
obtain; $H_{t,c}(G)$ in positive characteristic are a relatively recent
object of study (see \cite{BFG,La}). In general, the \emph{Poincar\'e-Birkhoff-Witt
theorem} for $H_{t,c}(G)$ over a pseudoreflection
group does not follow from $\field=\mathbb C$ case, and the the Koszulity
argument \cite{EG} is not directly applicalble. Irreducible finite
pseudoreflection groups $G$ were classified by Kantor, 
Wagner, Zalesski\v\i{} and Sere\v zkin, see an exposition in
\cite{KeM}; the group algebra $\field G$ is, in general, not
semisimple, and $H_{t,c}(G)$ may not have a
$\mathbb Z$\dash form. The present paper gives a proof of the PBW
theorem for rational Cherednik algebras in arbitrary characteristic. 

We remark in passing that  the representation theory of $H_{t,c}(G)$
in positive 
characteristic is clearly expected to differ from characteristic $0$
in a number of ways, even in the non\dash modular case when
$\field G$ is a semisimple algebra. 
For example, a family of $H_{t,c}(G)$\dash modules which should 
be viewed as standard modules, may be finite\dash dimensional; 
in this case, there is no question
of existence of finite\dash dimensional representations, but one
is still interested in the values of parameters $t,c$ for which the
standard modules are reducible.

Going further, we apply the results on perfect subquotients obtained in
Section~\ref{sect:perfect} to see that all quasi\dash Yetter\dash
Drinfeld structures $\delta$ on $V$ for which
the algebras $U(V,\delta)$ and $U(V^*,\delta)$ are commutative, 
come from perfect embeddings of $V$ in a certain module $\Ycr(G)$ over the
quantum group  $D(G)$ (in our terminology, a Yetter\dash Drinfeld
module over $G$). The ``quantisation'' $\Ycr(G)$ of $V$ turns out to
be trivial, $\Ycr(G)=V$, if $G$ has no complex reflections; in general,
$\dim \Ycr(G)= r + |\PR|$.
Using  techniques developed in Section~\ref{sect:perfect}, in
characteristic zero we obtain
\begin{theoremG}
Let $G\le \mathit{GL}(V)$ be a finite linear group over $\field$. For
each value of the parameters $t,c$, the 
rational Cherednik algebra $H_{t,c}(G)\cong S(V)\lcprod \field G
\rcprod S(V^*)$ embeds as a subdouble in the
braided Heisenberg double $\RHD_{\Ycr(G)} \cong \mathcal
B(\Ycr(G))\lcprod \field G \rcprod \mathcal B(\Ycr(G)^*)$.
\end{theoremG}
We study the embedding given by the Theorem for an irreducible complex
 reflection group $G$ over $\mathbb C$. This leads to: 

(1)
An embedding of a restricted Cherednik algebra $\overline{H}_{0,c}(G)$
in a braided Heisenberg double $\RHD_{Y_G}$ attached to a Yetter\dash Drinfeld
module $Y_G$ of dimension $|\PR|$. 
This double decomposes as $\mathcal{B}(Y_G) \lcprod \mathbb C G
\rcprod \mathcal{B}(Y_G^*)$. In particular, the coinvariant algebra
$S_G$ of $G$ embeds in a very interesting Nichols\dash Woronowicz
 algebra $\mathcal{B}(Y_G)$. We thus recover, using the new method of braided
 doubles, the result of the first
 author \cite{B} for Coxeter groups and an extension of this result 
 to all complex reflection groups due to Kirillov and Maeno \cite{KM};

(2)
 An action of  $\overline{H}_{0,c}(G)$ on $\mathcal{B}(Y_G)$. 
 In Coxeter type $A$, the
 algebra $\mathcal{B}(Y_{\Symm_n})$, or its quadratic cover, coincides
 with the Fomin\dash Kirillov algebra $\mathcal{E}_n$ from \cite{FK}; 

(3) An action of a rational Cherednik algebra $H_{t,c}(G)$ for $t\ne
 0$  on the deformed Nichols\dash Woronowicz
 algebra $\mathcal{B}_\tau(Y_G)$; in type $A$, the algebra
 $\mathcal{B}_\tau(Y_G)$ turns out to be the universal
 enveloping algebra of  a ``\emph{triangular Lie algebra}'' introduced in \cite{BEER}.

We finish the paper with an Appendix which contains proofs to a number
of auxiliary results on algebras with triangular decomposition.


\section{Quasi-Yetter-Drinfeld modules}
\label{sect:structural}

In this Section we introduce 
quasi\dash Yetter\dash Drinfeld modules.
Although the original motivation for these objects came from
a flat deformation problem given in Section~\ref{sect:overview},
we show that quasi\dash Yetter\dash Drinfeld modules arise naturally
in the framework of monoidal categories. 
We discuss properties of 
quasi\dash Yetter\dash Drinfeld modules and ways to construct such
modules. 

\subsection{Quasi-Yetter-Drinfeld modules and comodules}
\label{qydmodcomod}
Recall that by a left quasicoaction of (any vector space) $H$ on a vector
space $V$ we mean an arbitrary linear map $V\to H\tensor V$. 
We denote a quasicoaction by $v \mapsto v^{[-1]} \tensor v^{[0]}$. 

\begin{definition}
\label{def:qydmod}
Let $H$ be a bialgebra over $\field$. A \emph{quasi\dash Yetter\dash
  Drinfeld module} over $H$ is a vector space $V$ with
\begin{itemize}
\item[(1)] 
left $H$\dash action $\act\colon H \tensor V \to V$, 
$h\tensor v\mapsto h\act v$;
\item[(2)]
left $H$\dash quasicoaction $V\mapsto H \tensor V$, 
$v \mapsto v^{[-1]} \tensor v^{[0]}$,
\end{itemize}
which satisfy the \emph{Yetter\dash Drinfeld compatibility condition}:
\begin{equation*}
      ( h_{(1)}\act v)^{[-1]} \, h_{(2)} \tensor 
      ( h_{(1)}\act v)^{[0]} = h_{(1)}v^{[-1]} \tensor h_{(2)} \act v^{[0]}.
\end{equation*}
We will often abbreviate quasi\dash Yetter\dash Drinfeld to quasi\dash YD.
\end{definition}
\begin{remark}
A dual notion is that of a \emph{quasi\dash
  Yetter\dash Drinfeld comodule}  over $H$. It is a space $V$ with
an $H$\dash quasiaction (linear map $H\tensor V \to V$) and 
an $H$\dash coaction, which satisfy the same Yetter\dash Drinfeld
  compatibility condition   as in Definition~\ref{def:qydmod}.  
One can check that this compatibility condition is self\dash dual, and
  a quasi\dash YD module for $H$ is a quasi\dash
  YD comodule for $H^*$ when $\dim H<\infty$.

\emph{Yetter\dash Drinfeld modules}, a notion widely used in
modern quantum groups literature, are quasi\dash YD modules where
the quasicoaction is in fact a coaction. Yetter\dash Drinfeld modules
were formally introduced by Yetter \cite{Y} under the name of
``crossed bimodules'' (a linearisation of crossed sets in algebraic
topology, see e.g.\ Whitehead \cite{W}), and were shown by Majid
\cite{Mdoubles} to be the same as modules over the Drinfeld quantum
double $D(H)$ of $H$  if $H$ is a
finite\dash dimensional Hopf algebra.
\end{remark}

\subsection{The map $\Psi_{V,W}$}
\label{subsect:psi}

Let $H$ be a bialgebra, $V$ be an $H$\dash module with $H$\dash quasicoaction
$v\mapsto v^{[-1]}\tensor v^{[0]}$, and $W$ be an $H$\dash module. 
Consider the map
\begin{equation*}
    \Psi_{V,W} \colon V \tensor W \to W \tensor V, 
\qquad
     \Psi_{V,W}(v\tensor w) = (v^{[-1]}\act w )\tensor  v^{[0]}.
\end{equation*}
The Yetter\dash Drinfeld compatibility condition for $V$ can be recast
in terms of the maps $\Psi_{V,W}$ where $W$ runs over the category of
$H$\dash modules:
\begin{lemma}
\label{lem:Hmodmap}
An $H$\dash action and $H$\dash quasicoaction make a space $V$  a quasi\dash Yetter\dash Drinfeld
module, if and only if $\Psi_{V,W}$ is an $H$\dash module map for any
$H$\dash module $W$.
\end{lemma}
\begin{proof}
The condition that $\Psi_{V,W}$ is an $H$\dash equivariant map is
\begin{equation*}
         (h_{(1)}\act v)^{[-1]} \act (h_{(2)} \act w) \tensor
  (h_{(1)}\act v)^{[0]}  
= h_{(1)} \act (v^{[-1]}\act w )\tensor  h_{(2)} \act v^{[0]}.
\end{equation*}
This is just the Yetter\dash Drinfeld compatibility condition where
the first tensor component (which is in $H$) is ``evaluated'' (via the action)
on an arbitrary $H$\dash module $W$. 
\end{proof}

A categorical interpretation of the map $\Psi$ will be given after a
brief and informal reminder on 
    
\subsection{$\field$-linear monoidal categories}

Let  $(\C, \tensor, \mathbb{I})$ be a monoidal, or tensor, category. 
The monoidal product $\tensor$ is associative, meaning that there are
isomorphisms $\Phi_{X,Y,Z}$: $X\tensor (Y\tensor Z) \cong (X\tensor Y) \tensor
Z$ for all $X,Y,Z\in \C$, which are natural in $X$, $Y$, $Z$ and  satisfy
 MacLane's pentagon condition \cite[Ch.\ VII]{MacLane}. 

That $\C$ is a $\field$\dash linear monoidal category ideologically means
that all objects in $\C$ are $\field$\dash vector
spaces  with additional structure. 
Formally, $\C$ is a monoidal
category equipped with a (forgetful) tensor functor  $\C \to
\mathrm{Vect}_{\field}$ to the category of vector
spaces over $\field$. For a more formal treatment, the reader is
referred to \cite[\S8]{JStannaka}.
We may (and will) suppress the associativity isomorphisms $\Phi_{X,Y,Z}$ in all
formulas, thus in fact assuming the category to be strict; see \cite{Sch2} for justification.

Functors between $\field$\dash linear monoidal categories are tensor
functors which preserve the forgetful functor to  $\mathrm{Vect}_{\field}$. In
other words, a functor does not change the underlying vector space of
an object. 

A \emph{rigid} monoidal category is a category where any object $V$ has
a left dual $V^*$ and a right dual $\starV$.
A left dual $V^*$ comes with two maps, the evaluation  
$\langle\cdot,\cdot\rangle=\langle\cdot,\cdot\rangle_V\colon V^*\tensor V\to \field$, and the
coevaluation, $\coev=\coev_V\colon \field \to V \tensor V^*$
satisfying the axioms of the dual 
(as e.g.\ in \cite[Definition 9.3.1]{Mbook}). 
Right duals are defined similarly.
One may identify  $(\starV)^*$ and $^*(V^*)$ with $V$ (but not
$V^{**}$ with $V$; these objects may be non\dash isomorphic!). 
In a $\field$\dash linear rigid monoidal category, objects are finite\dash
dimensional vector spaces, and coevaluation is necessarily
given by $\coev_V(1)=v_a \tensor f^a$ where $\{f^a\}$, $\{v_a\}$ 
are dual bases of $V^*$, $V$ with respect to the evaluation
$\langle\cdot,\cdot\rangle$.
(Summation over repeated indices is implied.)

\begin{example}
\label{ex:mod}
Let $H$ be a bialgebra over $\field$. The category
$\Mod{H}$ of left modules over $H$ is a $\field$\dash linear  monoidal
category. The left $H$\dash action on the 
tensor product $X\tensor Y$ of two modules $X,Y$ is given by
$h\act(x\tensor y) = h_{(1)}\act x \tensor h_{(2)}\act y$. 
The trivial module $\field$, where $H$ acts via $h\act 1 = \epsilon(h)$, is the unit object. 

If $H$ is a Hopf algebra (with bijective antipode), the
category $\Mod{H}_\mathrm{f.d.}$ of finite\dash dimensional
$H$\dash modules is rigid. The module structure on $X^*$
is given by the equation 
$\langle h\act f, x\rangle = \langle f, Sh\act x\rangle$, 
where $f\in X^*$, $x\in X$ and $S$ is the antipode in $H$. 

A similar example is the category $\Comod{H}$ of left $H$\dash comodules.  
\end{example}
\begin{example}
Let $H$ be a bialgebra over $\field$. Define the category $\QYD{H}$ as follows:
\begin{itemize}
\item objects of $\QYD{H}$ $=$ quasi\dash Yetter\dash Drinfeld modules over
  $H$;
\item morphisms between $X$ and $Y$ $=$ $H$\dash module maps $X\to
  Y$ ( \emph{compatibility with the quasicoaction is not required});
\item monoidal product $=$ the structure of a quasi\dash YD module on
  $X\tensor Y$, given by
\begin{itemize}
\item $H$\dash action 
      $h\act(x\tensor y) = h_{(1)}\act x \tensor h_{(2)}\act y$, 
\item $H$\dash quasicoaction 
      $x\tensor y \mapsto x^{[-1]}y^{[-1]}\tensor
      x^{[0]}\tensor y^{[0]}$.
\end{itemize}
\end{itemize}
We will refer to the latter two formulas by saying that the action and
quasicoaction ``respect the tensor product''.
\end{example}
\begin{lemma}
\label{lem:monoidal}
$\QYD{H}$ is a $\field$\dash linear monoidal category.
\end{lemma}
\begin{proof}
Let $X$, $Y$ be quasi\dash YD modules. 
We have to check that the above action and quasicoaction on $X\tensor Y$
are Yetter\dash Drinfeld compatible. By Lemma \ref{lem:Hmodmap}, it is
enough to check that $\Psi_{X\tensor Y ,Z} \colon X \tensor Y \tensor
Z \to Z \tensor X \tensor Y$ is an $H$\dash module map, for an
arbitrary $H$\dash module $Z$. Indeed,
\begin{align*}
\Psi_{X\tensor Y ,Z} (x\tensor y \tensor z) & = x^{[-1]}y^{[-1]} \act z
   \tensor x^{[0]} \tensor y^{[0]} 
\\
& = (\Psi_{X,Z}\tensor \id_Y)(\id_X\tensor \Psi_{Y,Z})(x\tensor y
   \tensor z)
\end{align*}
for $x\in X$, $y\in Y$ and $z\in Z$. Since $\Psi_{X,Z}$ and
$\Psi_{Y,Z}$ are $H$\dash module morphisms, so is $\Psi_{X\tensor Y ,Z}$.  
\end{proof}

\begin{remark}
\label{rem:equiv}

It is convenient to think of an object in
$\QYD{H}$ as a pair $(V,\delta)$, where $V$ is an $H$\dash module and
$\delta\colon V \to H \tensor V$ is a Yetter\dash Drinfeld
quasicoaction on $V$. There is an obvious forgetful functor 
\begin{equation*}
    F\colon \QYD{H}\to \Mod{H}, \qquad F(V,\delta) = V,
\end{equation*}
which forgets the quasicoaction. This functor is an equivalence of monoidal
categories. Indeed, 
\begin{equation*}
   G \colon \Mod{H}\to\QYD{H}, \qquad G(V) = (V,0)
\end{equation*}
is a monoidal functor (the zero quasicoaction, $\delta(v)=0$, is always Yetter\dash
Drinfeld); $FG(V)=V$ and $GF(V,\delta)=(V,0)$ which is naturally isomorphic
to $(V,\delta)$. 

Hence the category $\QYD{H}$ may be viewed as a ``decorated module
category''. The fibre, $F^{-1}(V)$, of the forgetful functor over an
$H$\dash module $V$ has the structure of a $\field$\dash vector space:
if $(V,\delta_1)$ and $(V, \delta_2)$ are quasi\dash YD modules, then
$(V, \delta_1+\lambda \delta_2)$ is a quasi\dash YD module for any
$\lambda \in \field$. This is because the Yetter\dash Drinfeld
compatibility condition is linear in the quasicoaction. 

However, we do not know the dimension of $F^{-1}(V)$ for a given
$H$\dash module $V$, nor a sufficient condition that $\dim
F^{-1}(V)>0$.

Observe also that Yetter\dash Drinfeld $H$\dash\emph{coactions} on $V$ are a
subset (not a subspace) in $F^{-1}(V)$; this subset does not
necessarily span $F^{-1}(V)$. 
\end{remark}

Although the categories $\QYD{H}$ and $\Mod{H}$ are equivalent as
monoidal categories, there is an important structure, intrinsic in
$\QYD{H}$, which is ``forgotten'' by the forgetful functor to
$\Mod{H}$. This structure is the semibraiding. 

\subsection{Semibraided monoidal categories}

\begin{definition}
Let $\C$ be a monoidal category. A \emph{right semibraiding} on $\C$
is a family $\PSI=\{ \Psi_{X,Y}\colon X\tensor Y \to Y \tensor
X \mid X,Y\in\Ob\ \C\}$ of morphisms, such that 
\begin{itemize}
\item[1.] (naturality in the right\dash hand argument) 
     $\Psi_{X,Y}$ is natural in $Y$;
\item[2.] (right hexagon condition) 
$\Psi_{X\tensor Y, Z}=(\Psi_{X,Z}\tensor \id_Y)(\id_X\tensor \Psi_{Y,Z})$.
\end{itemize}
Similarly, a \emph{left semibraiding} is a family of
morphisms $\Psi_{X,Y}$, natural in $X$ and satisfying the ``mirror'' hexagon
condition for $\Psi_{X,Y\tensor Z}$. 
\end{definition}
\begin{remark}
Naturality of $\Psi_{X,Y}$ in $Y$ means that 
\begin{equation*}
   (\phi\tensor \id_X)\Psi_{X,Y} = \Psi_{X,Y'}(\id_X\tensor \phi)
\quad\text{for any morphism}\  \phi
\colon Y \to Y'. 
\end{equation*}
The (left and right) hexagon conditions, originally due to MacLane,
are so named because if the associativity isomorphisms like 
$(X\tensor Y)\tensor Z \cong X \tensor (Y\tensor Z)$ are explicitly
shown, these conditions are given by hexagonal commutative diagrams.  
See \cite[VII.7]{MacLane}.
\end{remark}
\begin{remark}[Braidings]
\label{rem:braiding}
A \emph{braiding} on a monoidal category is a collection of invertible
morphisms $\Psi_{X,Y}\colon X \tensor Y \to Y \tensor X$ which is both right and left semibraiding. Braidings
are a principal object in the theory of quantum groups, whereas
semibraidings are a new notion introduced in the present paper.
\end{remark}
We will use the term \emph{right} (resp.\ \emph{left})
\emph{semibraided category} for a pair 
$(\C, \PSI)$, where $\C$ is a monoidal category and $\PSI$ is a
right (resp.\ left) semibraiding on $\C$.
Let us give an example of a semibraided category, which will turn out to
be the canonical one.

For two quasi\dash Yetter\dash Drinfeld modules $X,Y$ over a
bialgebra $H$, define the map $\Psi_{X,Y}\colon X\tensor Y \to Y \tensor X$ as in \ref{subsect:psi}.
Denote by $\PSI$ the collection of $\Psi_{X,Y}$ for all pairs $X,Y\in
\Ob\ \QYD{H}$. 
\begin{lemma}
\label{lem:semibraided}
$(\QYD{H},\PSI)$ is a right semibraided monoidal category. 
\end{lemma}
\begin{proof}
By Lemma \ref{lem:Hmodmap}, $\Psi_{X,Y}$ are morphisms in $\QYD{H}$. 
Let us check the naturality in $Y$: for a morphism $\phi\colon Y \to
Y'$, which is an $H$\dash module map, 
\begin{equation*}
    (\phi\tensor \id_X)\Psi_{X,Y}(x\tensor y) = \phi(x^{[-1]}\act y)
    \tensor x^{[0]} = x^{[-1]}\act \phi(y) \tensor x^{[0]} =
    \Psi_{X,Y'}(x\tensor \phi(y))
\end{equation*}
as required. Finally, the right hexagon condition for $\Psi$ was explicitly
checked in the proof of Lemma~\ref{lem:monoidal}.
\end{proof}
\begin{remark}
The same can be done for the category of quasi\dash Yetter\dash
Drinfeld comodules. Let $X,Y\in \Ob\ \coQYD{H}$. Denote the
quasiaction by $\qact \colon H\tensor X \to X$. Let
$\PSI=\{\Psi_{X,Y}\}$ where
$\Psi_{X,Y}\colon X\tensor Y \to Y \tensor X$ is defined by the formula
$\Psi_{X,Y}(x\tensor y)= x^{(-1)} \qact y \tensor x^{(0)}$. Then
$(\coQYD{H},\PSI)$ is a left semibraided category.  
\end{remark}

\subsection{Reconstruction theorems for semibraided monoidal categories}

Our next goal is to prove a ``converse'' of
Lemma~\ref{lem:semibraided}. That is, a $\field$\dash linear right
semibraided category $(\C,\PSI)$ should be realised as a semibraided subcategory of
$\QYD{H}$ for some bialgebra $H=H(\C,\PSI)$. The process of
obtaining $H(\C,\PSI)$ from the category $(\C,\PSI)$ is called 
reconstruction. Ideally, starting with the category $\QYD{H}$, the
reconstruction should yield the original bialgebra $H$.  

Likewise, left semibraided monoidal categories are
expected to be realised as subcategories of $\coQYD{H}$.

Known reconstruction theorems include the realisation of a
$\field$\dash linear monoidal category $\C$, under
certain finiteness assumptions, in terms of either modules
or comodules over a bialgebra $H$ (a Hopf algebra if $\C$ is
rigid). If $\C$ is a braided category, $H$ will be a
(co)quasitriangular bialgebra. See the survey \cite{JStannaka}; original
sources include \cite{U,Y,Mbraidedgps} etc.

We will only state and prove a reconstruction theorem for a left
semibraided category and quasi\dash YD
comodules.
``Finiteness assumptions'' are easier to state for a comodule
realisation. Besides that, comodules behave in a more algebraic way
compared to modules.  We will bear in mind that a module version also
holds (interested reader can recover it, using \cite[9.4.1]{Mbook}
as a guide). The following finiteness assumptions are sufficient for the comodule
reconstruction, and are satisfied in all our applications of
the reconstruction theorem:
\begin{itemize}
\item monoidal categories are strict and small;
\item objects in $\field$\dash linear monoidal categories are
  finite\dash dimensional linear spaces over $\field$.
\end{itemize}
In what follows, the symbol $\qact$ will denote quasiaction.
\begin{definition}
We say that a $\field$\dash linear left semibraided category $(\C,\PSI)$ is
\emph{realised} over a bialgebra $H$, if there is a monoidal functor
$\C \to \coQYD{H}$ which preserves the left semibraiding. 

In other words, 
$H$ coacts and Yetter\dash Drinfeld compatibly quasiacts on (the underlying vector space of)
each object of $\C$, so that morphisms in $\C$ commute with the
coaction, the coaction and the quasiaction respect the tensor product, and 
\begin{equation*}
\Psi_{X,Y}(x\tensor y) =
  (x^{(-1)}\qact y) \tensor x^{(0)}
\end{equation*}
for $X,Y\in\Ob\ \C$.
\end{definition}
The following Lemma is easy; note that the product on $H_\delta$
arises from tensor multiplication in $\C$:
\begin{lemma}
\label{lem:subq}
Suppose that a $\field$\dash linear left semibraided category
$(\C,\PSI)$ is realised over a bialgebra $H$.
 For $X\in \Ob\ \C$, let $\delta$ be the coaction of $H$ on $X$.  

(a) Denote by $H_\delta$ the minimal subspace of $H$ such that
$\delta(X)\subset H_\delta\tensor X$ for all $X\in\Ob\ \C$. Then
$H_\delta$ is a subbialgebra of $H$.

(b) Let $I_\qact \subset H_\delta$ be the largest  biideal of  $H_\delta$ which quasiacts by
   zero on all objects in $\C$. Then $(\C, \PSI)$ is realised over the quotient
   bialgebra $H_\delta/I_\qact$. 
\qed
\end{lemma}
We will call the bialgebra $H_\delta/I_\qact$ the \emph{minimal
  subquotient} of $H$ \emph{realising} $(\C,\PSI)$. 
Among the bialgebras which realise $(\C,\PSI)$, we will be
reconstructing the bialgebra which is the smallest possible. 
We now state our 
\begin{theorem}[Reconstruction theorem for semibraidings]
\label{th:reconstr}
Let  $(\C,\PSI)$ be a $\field$\dash linear left semibraided monoidal
category, satisfying finiteness assumptions.
There exists a bialgebra $H(\C,\PSI)$ such that:

1.  $(\C,\PSI)$ is realised over $H(\C,\PSI)$;

2. (minimality) If $(\C,\PSI)$ is realised over another bialgebra $H'$,
   then the minimal subquotient of $H'$, realising $(\C,\PSI)$, is
   isomorphic to $H(\C,\PSI)$. 
\end{theorem}
\begin{proof}
The first step is to apply to the category $\C$ (ignoring the
semibraiding) the standard comodule
reconstruction, see \cite{JStannaka, U,Y,Mbraidedgps}.
According to this procedure, there exists a universal bialgebra $H_\C$, which 
coacts on all objects of $\C$, such that the coaction respects the
tensor product, and morphisms in $\C$ are
$H_\C$\dash comodule morphisms. Universality means that for any other
bialgebra $H'$ with these properties, 
there is a unique map $p\colon H_\C \to H'$ such that the coaction of
$H'$ on all $X\in\Ob\ \C$ factors through the coaction of $H_\C$:
$X\to H_\C\tensor X \xrightarrow{p\tensor\id} H'\tensor X$. 

We will use the following description of $H_\C$, which can be found in
the sources cited above. 

The bialgebra $H_\C$ is spanned by comatrix
elements $h_{x,\xi}=x^{(-1)}\langle x^{(0)},\xi\rangle$, $x\in X\in
\Ob\ \C$, $\xi\in X^\vee$. 
Here $X^\vee$ denotes the right dual vector space to $X$ (note that
$X^\vee$ is not an object in $\C$), and $\langle
x,\xi\rangle\in\field$ is the pairing of $x\in X$ and $\xi\in X^\vee$.
The product in $H_\C$ is given by $h_{x,\xi} h_{y,\eta} =
h_{x\tensor y, \eta\tensor \xi}$. Observe that the dual to the space
$X\tensor Y$ is $Y^\vee \tensor X^\vee$. The unit in $H_\C$ is
$h_{1,1}$ where $1\in \field = $the trivial object of $\C$. 
As $h_{x,\xi}$ are comatrix elements, the coproduct of $h_{x,\xi}$
    is $h_{x,\xi_a}\tensor h_{x^a,\xi}$.  Here $\{x^a\}$,
$\{\xi_a\}$ is any pair of dual bases of $X$, $X^\vee$; summation over
the repeated index is implied. The counit is
    $\epsilon(h_{x,\xi})=\langle x, \xi \rangle$. 

The full set of relations between the comatrix elements in $H_\C$ is spanned
by the obstructions for morphisms in $\C$ to become comodule morphisms:
\begin{equation*}
 \Obstr(\C)=\mathrm{span} \{h_{\phi(x), \eta} -
   h_{x,\phi^\vee(\eta)} \mid \phi\in\mathrm{Mor}(\C),\ 
   x\in \mathrm{source}(\phi), \
   \eta\in\mathrm{target}(\phi)^\vee \}.
\end{equation*}
Here $\phi^\vee$ is the linear map which is adjoint to $\phi$.

This description of  the bialgebra $H_\C$ is explicit enough to enable us to
introduce a quasiaction of $H_\C$ on objects in 
$\C$, which realises the semibraiding $\PSI$. In fact, it is
clear that there is
only one choice for such quasiaction:
for $x\in X$, $\xi \in X^\vee$, $y\in Y$ where $X,Y$ are objects in
$\C$, put 
\begin{equation*}
       h_{x,\xi} \qact y =(\id_Y \tensor \langle \cdot,
       \xi\rangle)\Psi_{X,Y}(x\tensor y).  
\end{equation*}
Let us check that the quasiaction is well\dash defined. We need to
make sure that elements of the space $\Obstr(\C)$ quasiact by zero on
objects in $\C$. 
Indeed, let $\phi\colon X \to Y$ be a morphism in $\C$, $x\in X$,
$\eta\in Y^\vee$.
The condition that $h_{\phi(x), \eta} - h_{x,\phi^\vee(\eta)}$
quasiacts on $z\in Z$ by zero is precisely equivalent to  
\begin{equation*}
    \Psi_{Y,Z}(\phi(x)\tensor z) - (\id_Z\tensor
    \phi)\Psi_{X,Z}(x\tensor z) = 0,
\end{equation*}
 which is the functoriality of $\Psi$ in the left\dash hand argument.
One easily checks that the quasiaction $\qact$ indeed realises
$\PSI$: i.e., $h_{x,\xi_a}\qact y \tensor x^a = \Psi_{X,Y}(x\tensor
y)$. It follows from the left hexagon axiom for $\PSI$ that the
quasiaction respects the tensor product in $\C$. 

Let us prove that the quasiaction of $H_\C$ is Yetter\dash Drinfeld
compatible with the coaction. 
The condition that $\Psi_{X,Y}(x\tensor y)=x^{(-1)}\qact y\tensor x^{(0)}$ is a morphism of comodules between
$X\tensor Y$ and $Y\tensor X$ reads
\begin{equation*}
    (x^{(-2)}\qact y)^{(-1)} x^{(-1)} \tensor (x^{(-2)}\qact y)^{(0)}\tensor x^{(0)}
= x^{(-2)}y^{(-1)}\tensor x^{(-1)}\qact y^{(0)} \tensor x^{(0)}.
\end{equation*}
Evaluating the rightmost tensor factor on both sides on $\xi\in
X^\vee$ and putting $h=h_{x,\xi}$, we obtain
\begin{equation*}
    (h_{(1)}\qact y)^{(-1)} h_{(2)} \tensor (h_{(1)}\qact y)^{(0)}
= h_{(1)}y^{(-1)}\tensor h_{(2)}\qact y^{(0)},
\end{equation*}
which is the required Yetter\dash Drinfeld compatibility condition; we
reiterate that comatrix elements $h_{x,\xi}$ span $H_\C$. 

The algebra $H_\C$ realises the semibraiding, but it may not be
minimal. 
Let $I_\qact(H_\C)$ be the largest among ($=$the sum of all) biideals in
$H_\C$ which quasiact on the whole category $\C$ by zero. 
Define 
\begin{equation*}
     H(\C,\PSI) = H_\C / I_\qact(H_\C).
\end{equation*}
Let us show that $H(\C,\PSI)$ satisfies the minimality property required
in the theorem. 
Suppose that $H$ is another bialgebra which realises the semibraided
category $(\C,\PSI)$. By the universality of $H_\C$, there is a 
map $p\colon H_\C \to H$ of bialgebras given by $p(h_{x,\xi})=x^{(-1)} \langle
x^{(0)},\xi\rangle$. Clearly, the image of the map $p$ is the subbialgebra
of $H$ denoted by $H_\delta$ in Lemma~\ref{lem:subq}.  
Note that, since the quasiaction of $H$ realises $\PSI$, 
the map $p$ must commute with quasiaction:
\begin{equation*}
    p(h_{x,\xi})\qact y = x^{[-1]}\qact y \langle x^{[0]}, \xi\rangle 
    = (\id_Y\tensor \langle \cdot,\xi\rangle )\Psi_{X,Y}(x\tensor y)
    = h_{x,\xi}\qact y.
\end{equation*}
Therefore, if $I_\qact$ is the largest biideal in $H_\delta$ which quasiacts
on $\C$ by zero, $p^{-1}(I_\qact)=I_\qact(H_\C)$ and the map $p$
induces isomorphism $H(\C,\PSI) = H_\C / I_\qact(H_\C) \xrightarrow{\sim}
H_\delta/I_\qact$ of bialgebras. 
\end{proof}

\begin{remark}
Strictly speaking, the finiteness conditions we specified do not allow
us to apply reconstruction to the category $\coQYD{H}$. It is not
clear whether we can always reconstruct $H$ from the category  
$\C=\coQYD{H}_{\mathrm{f.d.}}$ of finite\dash dimensional quasi\dash YD
comodules. No doubt that $H=H_\C$; the problem is whether there is a
biideal in $H$ which quasiacts by zero on all finite\dash
dimensional quasi\dash YD comodules. 

However, if $\dim H<\infty$, then $H(\coQYD{H}_{\mathrm{f.d.}},\PSI)=H$. 
\end{remark}

\subsection{Rigid semibraided categories and reconstruction of Hopf
  algebras}

In the present paper, we mainly use quasi\dash Yetter\dash Drinfeld
modules over Hopf algebras (not general bialgebras). In terms of 
reconstruction, Hopf algebras correspond to rigid monoidal categories.  
The following Lemma is straightforward:
\begin{lemma}
When $H$ is a Hopf algebra with invertible antipode $S$, 
the category $\QYD{H}_{\mathrm{f.d.}}$ of
finite\dash dimensional quasi\dash YD modules is a rigid right\dash
semibraided category, with  
the action and quasicoaction on $V^*$ given by 
\begin{equation*}
     \langle h\act f, v\rangle = \langle f, Sh\act v \rangle, 
\qquad
     f^{[-1]} \langle f^{[0]}, v\rangle = S^{-1}v^{[-1]} \langle f, v^{[0]}\rangle. 
\qed
\end{equation*}
\end{lemma}
\subsection*{Claim}
There is a Hopf algebra version of Lemma~\ref{lem:subq} and
Theorem~\ref{th:reconstr}, where
\begin{itemize}
\item[-- ] ``monoidal category'' is replaced with ``rigid monoidal
  category'';
\item[-- ] ``bialgebra'' is replaced with ``Hopf algebra'',
  ``subbialgebra'' with ``sub\dash Hopf algebra'' and ``biideal''
with ``Hopf ideal''. 
\end{itemize}
This follows from the fact that when $\C$ is
a rigid category, $H_\C$ has well\dash defined invertible antipode $S$
given on comatrix elements by
\begin{equation*}
        Sh_{f,x}=h_{x,f^\vee}.
\end{equation*}
Here the linear isomorphism $\cdot ^\vee\colon X^* \to X^\vee$, where
$X$ is an object in $\C$, is given by $\langle f,x\rangle = \langle
x,f^\vee\rangle$. It is not a morphism in the category $\C$. One
checks that $S$ preserves the biideal of definition of $H(\C,\PSI)$. 

\begin{remark}
If a monoidal category $\C$ is a subcategory of a rigid monoidal
category $\overline\C$, such that $\overline\C$ is a ``rigid
envelope'' of $\C$ in a proper sense, then the bialgebra $H_\C$ embeds
injectively in the Hopf algebra $H_{\overline\C}$; moreover,
$H_{\overline\C}$ will be the ``Hopf envelope'' of $H_\C$. Under some
assumptions, one can construct a rigid envelope of a monoidal
category. This is a categorical version of Manin's
construction of a Hopf envelope of a quadratic bialgebra in
\cite[Chapter 7]{Man}.
\end{remark}

\subsection{Compatible braidings on a vector space}

We are interested to have a supply of quasi\dash Yetter\dash Drinfeld
modules over bialgebras, or, better, Hopf algebras. 
The idea of the reconstruction theory is that such modules naturally
occur as objects of (rigid) right semibraided categories. We would
thus like to have a way of constructing semibraided
categories.

It is straightforward that a \emph{braided
category} is a particular case of a semibraided category.
Braided category is a monoidal category with a braiding (recall
Remark \ref{rem:braiding}). 
Braided categories were formally introduced in \cite{JS}; see also the
exposition in \cite[Ch.\ 9]{Mbook}. 
 However, braided categories are a source only of Yetter\dash Drinfeld
modules; this class of modules is not rich enough for our purposes.
We therefore need a more sophisticated example
of a semibraided category. Recall the following
\begin{definition}
A \emph{braiding on a vector space} $V$ is a linear 
map $\Psi\colon V\tensor V \to V \tensor V$,
satisfying the \emph{braid equation}
\begin{equation*}
           (\id \tensor \Psi)(\Psi \tensor \id)(\id \tensor \Psi) 
           = (\Psi \tensor \id)(\id \tensor \Psi)(\Psi \tensor \id). 
\end{equation*}
Here is a new notion:
\end{definition}
\begin{definition}
\label{def:compat}
A finite set $\Pi$ of braidings on a vector space $V$ is
\emph{compatible}, if for all $\Psi$, $\Psi'\in \Pi$
\begin{equation*}
      (\Psi'\tensor \id)(\id\tensor \Psi)(\Psi\tensor \id) = 
(\id\tensor \Psi)(\Psi\tensor \id)(\id\tensor \Psi').       
\end{equation*}
This is a right\dash handed version of compatibility; 
there is a left\dash handed version where the order of factors on both
sides is opposite. 
\end{definition}

We will now show how to construct a right semibraided category from a
set of compatible braidings on a vector space $V$. This will give a
structure of a quasi\dash Yetter\dash Drinfeld module on $V$.  

\subsection{A construction of a semibraided category from a set of
compatible braidings}
\label{subsect:acons}

Let $\Pi$ be a finite set of compatible braidings on a
space $V$. Let $\C_\Pi$ be a monoidal category,
whose objects are $V^{\tensorpow n}$, $n\ge 0$. 
The space $\mathrm{Mor}(V^{\tensorpow m},V^{\tensorpow n})$ of
morphisms will be:
\begin{itemize}
\item[-- ] $\field$, if $m=n=1$;
\item[-- ] the subalgebra of $\End(V^{\tensorpow n})$ generated by
  $\Psi_{i,i+1}$ (in leg notation), for all $\Psi\in\Pi$ and $1\le
  i\le n-1$, if $m=n>1$;
\item[-- ] $0$, if $m\ne n$. 
\end{itemize} 
For any $n\ge 1$ and $\Psi\in\Pi$, define
\begin{equation*}
   \Psi^{1,n}\colon V \tensor V^{\tensorpow n} \to  V^{\tensorpow
   n}\tensor V,
\qquad
   \Psi^{1,n} = \Psi_{n,n+1}\Psi_{n-1,n}\dots \Psi_{23}\Psi_{12}
\end{equation*}
(we introduce $\Psi^{1,n}$ as an endomorphism of $V^{\tensorpow n+1}$
and use the leg notation). In fact, $\Psi^{1,n}$ is obtained from
$\Psi$ using the ``left hexagon'' rule. Let 
\begin{equation*}
   \Psi_\Pi^{1,n}\colon  V \tensor V^{\tensorpow n} \to  V^{\tensorpow
   n}\tensor V,
\qquad
   \Psi_\Pi^{1,n} = \sum\nolimits_{\Psi\in\Pi} \Psi^{1,n}.
\end{equation*}
Now extend $\Psi_\Pi$ to a map exchanging any two objects in the
category:
\begin{equation*}
     \Psi_\Pi^{m,n}\colon V^{\tensorpow m} \tensor V^{\tensorpow n}
   \to  V^{\tensorpow  n}\tensor V^{\tensorpow m},
\quad
   \Psi_\Pi^{m,n} = (\Psi_\Pi^{1,n})_{1\dots n+1}(\Psi_\Pi^{1,n})_{2\dots
   n+2} \dots (\Psi_\Pi^{1,n})_{m\dots m+n},
\end{equation*}
that is, using the ``right hexagon'' rule. 
\begin{lemma}
The maps $\Psi_\Pi^{m,n}$ are a right semibraiding on the category
$\C_\Pi$. 
\end{lemma}
\begin{proof}
By construction, $\Psi_\Pi^{m,n}$ is a morphism in $\C_\Pi$ and
satisfies the right hexagon rule. We need to check that
$\Psi_\Pi^{m,n}$ is natural in its second argument. 
Because of the right hexagon rule, it is enough to check the
naturality of $\Psi_\Pi^{1,n}$ in its second argument. 
This will follow from the naturality of  $\Psi^{1,n}$ in the second
argument, for any $\Psi\in \Pi$.

Let $\phi\colon V^{\tensorpow n} \to V^{\tensorpow n}$ be a
morphism. We have to check that $\Psi^{1,n}(\id_V\tensor
\phi)=(\phi\tensor \id_V)\Psi^{1,n}$. We may assume that
$\phi=\Psi'_{i,i+1}$ for some $\Psi'\in\Pi$ and $i$ between $1$ and $n-1$. 
Then the naturality equation is the same as the compatibility
condition for $\Psi$ and $\Psi'$. 
\end{proof}
\begin{remark}
Let $\Pi$ be a finite compatible set of braidings on a vector space
$V$. Denote by $H_\Pi$ the bialgebra reconstructed, using the module
version of Theorem~\ref{th:reconstr}, from the category $\C_\Pi$. 
Then the space $V$ becomes a quasi\dash Yetter\dash Drinfeld module
for $H_\Pi$.

One can show that $\C_\Pi$ may be embedded in a rigid category, if and
only if all braidings $\Psi\in\Pi$ are rigid (or
\emph{biinvertible}, see \cite[4.2]{Mbook}). 
Such a rigid category will be generated, as a monoidal category, by 
objects 
\begin{equation*}
 \dots, V^{[-2]}={\vphantom{V}}^{**}V, \ 
        V^{[-1]}={\vphantom{V}}^{*}V, \
        V^{[0]} = V, \ 
        V^{[1]} = V^*, \ 
        V^{[2]} = V^{**},
\dots
\end{equation*} 
such that $(V^{[m]})^* = V^{[m+1]}$ and
${\vphantom{V}}^*(V^{[m]})=(V^{[m-1]})$. 
Standard formulas show
how to compute the braiding between, say, $V$ and $V^*$, and this
extends recursively to braidings between $V^{[m]}$ and $V^{[n]}$.  

In this case, we obtain a
Hopf algebra $H_\Pi$.  
\end{remark}

\begin{remark}
\label{rem:twocoact}
Suppose that  a Hopf algebra $H$ acts on a space $V$, and there are coactions, $v$ $\mapsto$ $v^{(-1)_i}\tensor v^{(0)_i}$, $i=1,\dots,N$, of $H$ on
$V$, each satisfying the Yetter\dash Drinfeld condition. Then the braidings
$\Psi_i(v\tensor w)=v^{(-1)_i} \act w \tensor  v^{(0)_i}$ are
compatible. This can be checked directly. 
The Hopf algebra $H_{\{\Psi_1,\dots,\Psi_N\}}$ will be the
minimal among subquotients of $H$ which still act and coact (in $N$ ways) on $V$. 
The quasicoaction of $H_{\{\Psi_1,\dots,\Psi_N\}}$ will
be given by $v$ $\mapsto$ $\sum_i v^{(-1)_i}\tensor v^{(0)_i}$. 
\end{remark}

\subsection{Minimal Yetter-Drinfeld realisation of a braided space}
\label{HPsi}

A particular case of a set of compatible braidings on $V$ is a one\dash
element set $\Pi=\{\Psi\}$. Assume that $\Psi$ is biinvertible. 
The above procedure yields a minimal
Hopf algebra (denote it by $H_\Psi$) over which the braided space
$(V,\Psi)$ is realised as a Yetter\dash
Drinfeld module. 

That a braided space $(V,\Psi)$  can be realised as a module over a coquasitriangular bialgebra (hence a Yetter\dash
Drinfeld module), follows from the Faddeev -
Reshetikhin - Takhtajan construction \cite{FRT}; the latter admits a Hopf
algebra version, e.g.\ \cite{Sch,Tak}.   
The Hopf algebra $H_\Psi$ which we propose to reconstruct, is not the
one given by the FRT construction but rather its quotient by the left kernel of the
coquasitriangular structure. Because $H_\Psi$ has more relations than
the FRT Hopf algebra, it looks more interesting algebraically.

Let us list some properties of the Hopf algebras $H_\Psi$. 
``Braided space'' will mean a finite\dash
dimensional space over $\field$ with biinvertible braiding. 

\subsubsection{}
The Hopf algebra $H_\Psi$ is trivial ($H_\Psi\cong\field$) if and
only if the braiding $\Psi$ is trivial ($\Psi(x\tensor y)=y\tensor
x$ for all $x$, $y$).

\subsubsection{}
 Any finite\dash dimensional Hopf algebra $H$ is isomorphic to  a
   structural Hopf algebra of some braided space $(V, \Psi)$. 
(For example, take $V$ to be the Drinfeld double $D(H)$ of $H$, with
   a standard braiding.
Braided spaces of dimension smaller than $\dim D(H)=(\dim H)^2$ may
give rise to the same Hopf algebra $H$.)

\subsubsection{}
 Suppose $\Psi$ is a braiding on a space $V$. 
   Then $\Psi^*$ is a braiding on $V^*$, and 
   $H_{\Psi^*}$ is isomorphic to $H_\Psi$. 
(This is because the rigid
   braided categories, generated by $(V,\Psi)$ and by $(V^*,\Psi^*)$,
   are the same.)

\subsubsection{}
\label{qtau}
 If $\Psi(x\tensor y)=qy\tensor x$ for a constant $0\ne q\in
   \field$, the Hopf algebra $H_\Psi$ is isomorphic to the group
   algebra of $\mathbb{Z}/n\mathbb{Z}$, if $q$ is a root of unity of
   order $n$ in $\field$, or of $\mathbb{Z}$ if $q$ is not a root of
   unity. (It is easy to realise $\Psi$ over this Hopf algebra and to
   show that $\Psi$ cannot be realised over its proper subquotient.)

\subsubsection{}
 $H_\Psi$ and $H_{\Psi^{-1}}$ are
   nondegenerately dually paired Hopf algebras. If $H_\Psi$ is
   finite\dash dimensional, so is $H_{\Psi^{-1}}$, and
   $H_{\Psi^{-1}}=(H_{\Psi})^*$. 
(This follows by analysing the
   coquasitriangular structure on the FRT Hopf
   algebra of $(V,\Psi)$.)
This duality pairing yields an elegant proof of the following fact:
\begin{lemma}
Let the field $\field$ be algebraically closed. The group algebra
$\field G$ of a finitely generated Abelian group $G$ is a self\dash
dual Hopf algebra. 
\end{lemma}
\begin{proof}
It is enough to prove this for a group $G=\mathbb{Z}/n \mathbb{Z}$ where $n$ is either
zero or a positive integer. Let $q$ be a root of unity of order $n$
(or not a root of unity, if $n=0$). Let $V=\field x$ and 
$\Psi(x\tensor x)=q  x\tensor x$ be the braiding on $V$. Then 
by \ref{qtau} both
$H_\Psi$ and $H_\Psi^{-1}$ are isomorphic to the group algebra $\field G$. 
\end{proof}

\subsubsection{}
\label{cocomm}
 $H_\Psi$ is cocommutative, if and only if the braiding $\Psi$ is
   compatible with the trivial braiding $\tau$. Dualising, $H_\Psi$ is
   commutative if and only if $\Psi^{-1}$ is compatible with $\tau$.

 It seems to be a challenging problem to extract other 
  properties of the Hopf algebra $H_\Psi$ 
  from the properties of the operator $\Psi$. For example, when
   $H_\Psi$ is finite\dash dimensional? Semisimple? Is a group
   algebra? Here is a converse problem, which may also be of interest:
   given a finite\dash dimensional Hopf algebra $H$, find a braided
   space $(V,\Psi)$ of smallest dimension, such that
   $H_\Psi\cong H$.

\subsection{Quasi-Yetter-Drinfeld modules over cocommutative  or
  quasitriangular $H$}

We conclude this Section with a simple but useful observation
which allows us to obtain  compatible braidings and to
 construct quasi\dash Yetter\dash Drinfeld modules from given Yetter\dash
 Drinfeld modules.  

\begin{lemma}
\label{lem:cocomm}
Let $V$ be a Yetter\dash Drinfeld module over a cocommutative Hopf
algebra $H$, with coaction $\delta(v)= v^{(-1)} \tensor v^{(0)}$. 
Then the induced braiding $\Psi$ on $V$ 
is compatible with the
trivial braiding $\tau(v\tensor w)=w\tensor v$. 
For any $\lambda\in \field$
\begin{equation*}
       \delta_{\Psi,\lambda \tau}(v) = v^{(-1)} \tensor v^{(0)} + \lambda\cdot 1\tensor v
\end{equation*}
defines a Yetter\dash Drinfeld quasicoaction on $V$. 
\end{lemma}
\begin{proof}
Any module $V$ over a cocommutative Hopf algebra
can be turned into a Yetter\dash Drinfeld module with the trivial
coaction $v\mapsto 1\tensor v$. Indeed, let us check that the trivial
coaction is Yetter\dash Drinfeld compatible with any action:
\begin{equation*}
       1\cdot h_{(2)} \tensor h_{(1)}\act v = h_{(1)}\cdot 1 \tensor
       h_{(2)} \act v,
\end{equation*} 
which is true by cocommutativity. 

Thus, the braidings $\Psi$ and $\tau$ are realised on $V$ via the same
action and two Yetter\dash Drinfeld coactions of a Hopf algebra,
therefore by Remark~\ref{rem:twocoact} they are compatible. The
quasicoaction $\delta_{\Psi,\lambda\tau}$ is Yetter\dash Drinfeld
as a linear combination of Yetter\dash Drinfeld coactions. 
\end{proof}
\begin{remark}
More generally, let $(V,\Psi)$ be a Yetter\dash Drinfeld module over a
quasitriangular Hopf algebra $H$. The braiding on $V$, induced
by the quasitriangular structure $R=R^1\tensor R^2 \in H\tensor H$, is compatible with $\Psi$. 
There is a Yetter\dash Drinfeld quasicoaction on $V$, given by 
\begin{equation*}
     \delta_{\Psi,\lambda R}(v) = v^{(-1)} \tensor v^{(0)} + \lambda
     R^2 \tensor R^1\act x
\end{equation*} 
for any $\lambda\in\field$.
\end{remark}


\section{Free braided doubles}
\label{sect:braideddoubles}

We will now study algebras with triangular decomposition of the form
$T(V)\lcprod H \rcprod T(V^*)$, where the commutator of $V^*$ and $V$
lies in the bialgebra $H$.  Such algebras are called free braided
doubles. The purpose of this Section is to show that free braided
doubles are ``the same'' as quasi\dash Yetter\dash Drinfeld modules
over $H$. 

\subsection{Algebras $\D_\beta$}
\label{nhqa}

Let $H$ be a bialgebra, and let $V$ be a finite\dash dimensional 
space with left $H$\dash action $\act$. 
As usual, $V^*$ denotes the linear dual of $V$,
$\langle\cdot,\cdot\rangle\colon V^*\tensor V\to \field$ is the canonical pairing, 
and $V^*$ is viewed as a right $H$\dash module via 
$\langle f\ract h, v\rangle = \langle f,  h\act v\rangle$. 

\begin{definition}
To any 
linear map $\beta\colon V^* \tensor V \to H$ 
there corresponds an associative algebra $\D_\beta$, generated by
all $v\in V$, $h\in H$ and $f\in V^*$, subject to:  
\begin{itemize}
\item[(i)] semidirect product relations 
    $h\cdot v = (h_{(1)} \act v)\, h_{(2)}$, 
    $f \cdot h = h_{(1)} \, (f \ract h_{(2)})$;
\item[(ii)] commutator relation $f\cdot v - v\cdot f =\beta(f,v)$.
\end{itemize}
In this definition, we assume that all relations between $h\in H$ hold in
$\D_\beta$, and also that the unity in $H$ is the unity in $\D_\beta$:
$1_{\D_\beta} = 1_H$.  
\end{definition}

The map
\begin{equation*}
        m_\beta \colon  T(V)\tensor H\tensor T(V^*) \to \D_\beta,
\end{equation*}
of vector spaces, which is induced by the multiplication in
$\D_\beta$, is surjective. 
This is because any monomial in generators of $\D_\beta$ may be
rewritten, using the relations, as a linear combination of monomials
of the form  $v_1 v_2 \dots v_m \cdot h\cdot f_1 f_2 \dots f_n$, where
$v_i\in V$, $h\in H$ and $f_j\in V^*$.

If $\beta=0$, so that  the generators $f\in V^*$
commute with the generators $v\in V$, the
map $m_0$ is an isomorphism of vector spaces:
\begin{equation*}
    m_0 \colon T(V) \tensor H \tensor T(V^*) \cong \D_0. 
\end{equation*}
Indeed, it is easy to check that multiplication on 
$T(V) \tensor H \tensor T(V^*)$ defined by $(\eta\tensor g \tensor
a)(\theta\tensor h \tensor b)=\eta(g_{(1)}\act \theta)
\tensor g_{(2)} h_{(1)} \tensor (a\ract h_{(2)})b$ is associative, and
it clearly obeys the semidirect product and commutator relations in $\D_0$. 
Note that the subalgebra of $\D_0$, generated by $V$ and $H$, is
isomorphic to the semidirect product $T(V)\lcprod H$ by the left
action of $H$. 
Similarly, $H$ and $V^*$ generate a subalgebra isomorphic to the semidirect
product $H\rcprod T(V^*)$. The algebra $\D_0$ is obtained by
``gluing'' these two semidirect products together along~$H$. 

\begin{definition}
We say that the algebra $\D_\beta$ has \emph{triangular decomposition} over the
bialgebra $H$, if the map $m_\beta$ is a vector space isomorphism.  
\end{definition}

Our key question in this Section is, which $\beta\colon V^*\tensor V
\to H$ have this property. A complete answer to this question is given in 
\begin{theorem}
\label{thm_qYD}
The algebra $\D_\beta$ has triangular decomposition 
\begin{equation*}
 \D_\beta \cong T(V) \tensor H \tensor T(V^*)
\end{equation*}
over the bialgebra $H$, 
if and only if the $H$\dash valued pairing $\beta\colon V^*\tensor V
\to H$ satisfies the Yetter\dash Drinfeld condition:
\begin{equation*} 
     h_{(1)} \, \beta(f \ract h_{(2)}, v)
= \beta(f, h_{(1)}\act v) \, h_{(2)} 
\end{equation*}
 for all $v\in V$, $h\in H$, $f\in V^*$.
\end{theorem}
We postpone the proof of Theorem \ref{thm_qYD} 
until the end of this Section, and will now discuss the result itself. 

The next Lemma (which follows by easy linear algebra) clarifies why
the equation for $\beta$ in the Theorem is termed the Yetter\dash Drinfeld condition:

\begin{lemma}
Let $V$ be a finite\dash dimensional module over a bialgebra $H$. 

1. Linear maps $\beta\colon V^* \tensor V\to H$ are in one\dash
   to\dash one correspondence with quasicoactions (linear maps) $\delta\colon V \to
   H\tensor V$, via the formula 
\begin{equation*}
   \delta_\beta(v) = \beta(f^a,v)\tensor v_a,
\end{equation*}
where $\{f^a\}$, $\{v_a\}$ are dual bases of $V^*$, $V$. 

2. A map $\beta\colon V^* \tensor V\to H$ satisfies the equation in
   Theorem~\ref{thm_qYD}, if and only if the quasicoaction
   $\delta_\beta$ is Yetter\dash Drinfeld compatible with the $H$\dash
   action on $V$.
\qed
\end{lemma}

\begin{definition}
An algebra  $\D_\beta$, satisfying the conditions of Theorem
\ref{thm_qYD}, is called a \emph{free {braided} double}.  
\end{definition}
We may now restate Theorem \ref{thm_qYD} in the following way:
\begin{corollary} 
\label{cor_parametrisation}
Free {braided} doubles over a bialgebra $H$ are parametrised by
finite\dash dimensional quasi\dash Yetter\dash Drinfeld modules over $H$. 
\end{corollary}
The parametrisation is as follows. Let $(V,\delta)$ be a finite\dash dimensional
quasi-YD module  over $H$. According to Definition~\ref{def:qydmod},
this means that $V$ is an $H$\dash module
and $\delta(v)=v^{[-1]}\tensor v^{[0]}$ is an $H$\dash quasicoaction
on $V$ satisfying the Yetter\dash Drinfeld compatibility condition.
To $(V,\delta)$ is associated the free braided double 
\begin{equation*}
        \D(V,\delta) := T(V) \lcprod H \rcprod T(V^*)
\quad\text{with defining relation }[f,v]=v^{[-1]}\langle f,v^{[0]}\rangle.
\end{equation*}
Square brackets mean a commutator $fv-vf$.
 
Vice versa, a free {braided} double of the form
$T(V)\lcprod H\rcprod T(V^*)$ where $V$ is a finite\dash dimensional $H$\dash module,
gives rise to a Yetter\dash Drinfeld quasicoaction on 
$V$ given by $v\mapsto [f^a,v]\tensor v_a$. 
Here $\{f^a\}$, $\{v_a\}$ are dual bases of $V^*$,~$V$.

\subsection{Classification of one-dimensional quasi-Yetter-Drinfeld modules}
\label{ex:trivial}

We know from Section~\ref{sect:structural} that the universal source
of quasi\dash Yetter\dash Drinfeld modules are right semibraided monoidal categories.
This means that in general, quasi\dash YD modules are at least as
complicated as solutions to the quantum Yang\dash Baxter equation. 

However, one\dash dimensional quasi\dash YD modules over a Hopf
algebra can be fully classified. We will do this here. 
A practical way to obtain some non\dash trivial quasi\dash YD modules
is to take direct sums of one\dash dimensional modules.

One\dash dimensional representations of $H$ are the same as 
algebra maps $H\to \field$. Under the convolution product of algebra maps
(=tensor product of representations), these form a group $G(H^\circ)$
of grouplike elements  in the finite dual $H^\circ$ of $H$ \cite[9.1.4]{Mon}.  
Quasicoactions on a $1$\dash dimensional space $V$ are given by  
$v\mapsto p\tensor v$, where $p\in H$.

The group $G(H^\circ)$ acts on $H$ by algebra automorphisms 
$t_\alpha\colon H\to H$, defined as  
$t_\alpha(h) = \alpha(Sh_{(1)}) h_{(2)} \alpha(h_{(3)})$ for $\alpha
\in G(H^\circ)$.
Let $[g,h]_{t_\alpha} = gh-t_\alpha(h) g$ be the $t_\alpha$\dash 
commutator in $H$. One can check that the quasicoaction $v\mapsto
p\tensor v$ on the representation $\alpha$ 
is Yetter\dash Drinfeld, if and only if
\begin{equation*}
        [p,h]_{t_\alpha}=0\quad \text{for all }h\in H.
\end{equation*}
In particular, if $H$ is cocommutative, all $t_\alpha$ are the
identity on $H$. Isomorphism classes of $1$\dash
dimensional quasi\dash YD modules over $H$ then correspond to pairs
\begin{equation*}
     (\alpha, p) \in G(H^\circ) \times Z(H),
\end{equation*}
where $Z(H)$ is the centre of $H$. Under the tensor product, these
isomorphism classes form a commutative monoid isomorphic to  
$G(H^\circ) \times Z(H)$. One\dash dimensional Yetter\dash Drinfeld
modules correspond to the subgroup 
$G(H^\circ) \times (G(H)\cap Z(H))$. 
 
This classification, incidentally, shows that the space
of Yetter\dash Drinfeld quasicoactions on a given $H$\dash module $V$
need not coincide with the linear span of  Yetter\dash Drinfeld coactions.  
It is also a key ingredient in the
example of braided doubles given in \ref{bdintro}.

\subsection{Proof of Theorem \ref{thm_qYD}}

The rest of this Section will be devoted to the proof of Theorem
\ref{thm_qYD}. 

It is easy to show that the Yetter\dash Drinfeld condition is
necessary for  $\D_\beta$ to have triangular decomposition over $H$. 
Indeed, let $f\in V^*$, $h\in H$ and $v\in V$. 
Denote $L=h_{(1)}\beta(f\ract h_{(2)},v)$ and $R=\beta(f,h_{(1)}\act
v)h_{(2)}$. Compute the product $fhv$ in
$\D_\beta$ in two ways: first, $fhv=h_{(1)}(f\ract h_{(2)})v$, which
by the commutator relation equals 
$L+h_{(1)} v(f\ract h_{(2)})=L+(h_{(1)}\act v)h_{(2)}(f\ract
h_{(3)})$.
Second, $fhv=f(h_{(1)}\act v)h_{(2)}$, which by the commutator
relation is $R+(h_{(1)}\act v)fh_{(2)}=R+(h_{(1)}\act v)h_{(2)}(f\ract
h_{(3)})$. Thus, $L=R$ in $\D_\beta$. But $H$ embeds in $\D_\beta$
injectively because of the triangular decomposition. Therefore, $L=R$ in $H$
as required. 

To show that the Yetter\dash Drinfeld condition is sufficient, it is
enough to introduce on $T(V)\tensor H \tensor T(V^*)$ 
associative multiplication which satisfies the defining relations of
$\D_\beta$.

In order to construct such multiplication on $T(V)\tensor H \tensor
T(V^*)$, we would like to use a general fact about \emph{algebra
  factorisations}. Let $X$ and $Y$ be
associative algebras. Denote by 
$m^X\colon X\tensor X \to X$, resp.\ $m^Y\colon Y\tensor Y \to Y$,  
the multiplication map for $X$, resp.\ $Y$. 
An associative product on $X\tensor Y$, which simultaneously extends $m^X$ and
$m^Y$ (in other words, an algebra factorisation into $X$, $Y$), 
is defined via 
\begin{equation*}
   (x\tensor y)(x'\tensor y') = (m^X\tensor m^Y)(x\tensor c(y\tensor
   x') \tensor y'),
\end{equation*}
where $c\colon Y\tensor X \to X\tensor Y$ is a 
\emph{twist map} between $Y$ and $X$ (also called rule of exchange of
tensorands). Associativity of this product is equivalent to two
equations on $c$; see \cite[Proposition 21.4]{Mcompanion}:
\begin{proposition} 
\label{prop_hexagons}
The above product on $X\tensor Y$ is associative, if and only if 

$(\ref{prop_hexagons}a)$\quad $c\circ (\id^Y \tensor m^X) = (m^X\tensor \id^Y)(\id^X\tensor
   c)(c\tensor \id^X)$,

$(\ref{prop_hexagons}b)$\quad $c \circ (m^Y\tensor \id^X) = (\id^X\tensor m^Y)(c\tensor
   \id^Y)(\id^Y\tensor c)$. 

Both sides of $(\ref{prop_hexagons}a)$ are maps from $Y\tensor
X\tensor X$ to $X\tensor Y$, whereas in $(\ref{prop_hexagons}b)$ the
maps are from $Y\tensor Y\tensor X$ to $X\tensor Y$. 
\qed
\end{proposition}

We put $X=T(V)\lcprod H$, the semidirect product algebra arising from
the left action of $H$ on $V$, and $Y=T(V^*)$. 
In the construction of the twist map $c$ between $Y$ and $X$, one uses
the fact that $Y$ is a free
tensor algebra. We use the notation $T^{\le 1}(V^*)=\field \oplus
V^*$. 
\begin{lemma}
Let $Y=T(V^*)$ and let  $c'\colon V^*\tensor X \to X
\tensor T^{\le 1}(V^*)$ be a ``partial twist map'' 
satisfying $(\ref{prop_hexagons}a)$.
Then there exists a unique twist map $c\colon T(V^*)\tensor  X \to X \tensor
T(V^*)$, which extends $c'$ and satisfies $(\ref{prop_hexagons}a)$, 
$(\ref{prop_hexagons}b)$ and $c(1\tensor x)=x\tensor 1$.
\end{lemma}  
\begin{proof}[Proof of the Lemma]
The map $c\colon Y\tensor X \to X\tensor Y$ is defined (in tensor leg
notation) by $c(f_1\tensor \dots \tensor f_n, x) = 
c'_{12}c_{23}\dots c'_{n,n+1}(f_1\tensor \dots \tensor f_n\tensor x)$, 
where $f_i$ are elements of some basis of $V^*$ and $n\ge 1$, and 
the condition $c(1\tensor x)=x\tensor 1$.
By construction, $c$ satisfies $(\ref{prop_hexagons}b)$; this property
also guarantees uniqueness of $c$.

Let us now check $(\ref{prop_hexagons}a)$, i.e., that 
$cm^X_{23}(\xi\tensor x \tensor x')=m^X_{12} c_{23} c_{12}(\xi\tensor
x \tensor x')$ for all $\xi\in T(V^*)$, $x,x'\in X$. 
We use induction in the tensor degree $n$ of $\xi\in V^{*\tensorpow
  n}$. When $n=1$, the property holds because $c$ coincides with
$c'$. Assume $n>1$ and that $(\ref{prop_hexagons}a)$ holds for tensors in
$T(V^*)$ of degree $<n$. Write $\xi=\eta\tensor \theta$ where $\eta$,
$\theta$ are tensors in $T(V^*)$ of degree strictly less than $n$. 
We have
\begin{align*}
    cm^X_{23}(\xi \tensor x \tensor x') 
   & = c_{12} c_{23} m^X_{34}(\eta\tensor \theta \tensor x \tensor x')
=c_{12}m^X_{23}c_{34}c_{23}(\eta\tensor \theta \tensor x \tensor x')
\\
&=m^X_{12}c_{23}c_{12}c_{34}c_{23}(\eta\tensor \theta \tensor x \tensor x')
=m^X_{12}c_{23}c_{34}c_{12}c_{23}(\eta\tensor \theta \tensor x \tensor x'),
\end{align*}
where the $1$st step is by property $(\ref{prop_hexagons}b)$ of $c$,
the $2$nd and the $3$rd steps are by induction hypothesis, and the
last step is trivial. But by property $(\ref{prop_hexagons}b)$, this
expression is precisely $m^X_{12}c_{23}c_{12}(\xi \tensor x \tensor
x')$. The Lemma is proved.
\end{proof}

We will now construct a certain partial twist map $c'\colon V^*\tensor X \to X
\tensor T^{\le1}(V^*)$, which will satisfy $(\ref{prop_hexagons}a)$. 
First of all, we define operators
\begin{equation*}
      \deriv_f \colon T(V) \to T(V)\lcprod H,
\quad \deriv_f(V^{\tensorpow n})\subset V^{\tensorpow n-1}\tensor H, 
\end{equation*}
by the formula
\begin{equation*}
\deriv_f(v_1\tensor \dots \tensor v_n) = 
\sum_{i=1}^n (v_1\tensor \dots \tensor v_{i-1})\cdot \beta(f,v_i)
\cdot (v_{i+1}\tensor \dots \tensor v_{n}),
\end{equation*}
where $\cdot$ is the multiplication in the algebra $X=T(V)\lcprod H$. 
We put $\deriv_f1=0$. 
\begin{lemma}
\label{Leibn}
1. The operators $\deriv_f$ obey the
Leibniz rule in the following form: 
\begin{equation*}
   \deriv_f(pq)=
   (\deriv_f p)\cdot q + p\cdot (\deriv_f q)
   \quad\text{for }p,q\in T(V),
\end{equation*}
where $\cdot$ is the product in $T(V)\lcprod H$.

2. For any $b\in T(V)$, 
\begin{equation*} 
       h_{(1)} \cdot \deriv_{f\ract h_{(2)}}b 
       = \deriv_f(h_{(1)}\act b) \cdot h_{(2)}. 
\end{equation*}
\end{lemma}
\begin{proof}[Proof of the Lemma]
1.~The Leibniz rule is obvious from the definition of $\deriv_f$. 

2.~When $b=v\in V$, this equality is the Yetter\dash Drinfeld
condition $h_{(1)} \beta(f\ract h_{(2)}, v)$ $=$ $\beta(f,h_{(1)}\act
v)h_{(2)}$. (This is the only place  in the proof of Theorem
\ref{thm_qYD} where the Yetter\dash
Drinfeld condition is invoked.)
Furthermore, it is easy to see that if the equality holds for $b$ and
for $b'$ (where $b$, $b'$ are tensors in $T(V)$), it holds for their
product $bb'$. Indeed, 
$h_{(1)} \cdot \deriv_{f\ract h_{(2)}}(bb')$ is equal, by the
Leibniz rule,  
to $h_{(1)} \cdot \deriv_{f\ract h_{(2)}}b\cdot b'+
(h_{(1)} \act b) \cdot h_{(2)} \cdot \deriv_{f\ract h_{(3)}}b'$. 
Replace this with $\deriv_f(h_{(1)}\act b) \cdot h_{(2)}\cdot b'
+(h_{(1)} \act b)\cdot \deriv_f(h_{(2)}\act b) \cdot h_{(3)}$
which is, again by the Leibniz rule,  equal to 
$\deriv_f(h_{(1)}\act (bb'))\cdot h_{(2)}$. 
The equality thus holds for any $b\in T(V)$, and 
the Lemma is proved. 
\end{proof}

Now for $f\in V^*$ and $ah\in X$, where $a\in T(V)$ and $h\in H$, we put
\begin{equation*}
   c'(f,ah) = (\deriv_f a)\cdot h\tensor 1 + ah_{(1)} \tensor f\ract
   h_{(2)}. 
\end{equation*}
\begin{lemma}
The above map $c'\colon V^*\tensor X \to X \tensor T^{\le1}(V^*)$
satisfies $(\ref{prop_hexagons}a)$.
\end{lemma} 
\begin{proof}[Proof of the Lemma]
Take $ah$, $bk\in
X$, where $a,b\in T(V)$ and $h,k\in H$; $(\ref{prop_hexagons}a)$ is
equivalent to 
\begin{align*}
\tag{$*$}
        c'(f,ah\cdot bk) = (\deriv_f a)\cdot h\cdot bk \tensor 1 + 
      ah_{(1)} &\cdot (\deriv_{f\ract h_{(2)}} b)\cdot k \tensor 1 
    \\ &+ ah_{(1)}\cdot bk_{(1)} \tensor f\ract h_{(2)}k_{(2)}.
\end{align*}
Let us expand the left\dash hand side of $(*)$. 
In the semidirect product algebra $T(V)\lcprod H$ the product  
$ah\cdot bk$ is equal to $a(h_{(1)}\act b) h_{(2)}k$, hence we have 
$\deriv_f(a(h_{(1)}\act b))\cdot h_{(2)}k \tensor 1
+ a(h_{(1)}\act b) h_{(2)}k_{(1)} \tensor f\ract h_{(3)}k_{(2)}$ on
the left in $(*)$. By the Leibniz rule for $\deriv_f$, the
left\dash hand side of $(*)$ is 
\begin{equation*}
(\deriv_f a)\cdot (h_{(1)}\act b)
h_{(2)}k \tensor 1 + a \cdot \deriv_f(h_{(1)}\act b)\cdot h_{(2)}k
\tensor 1  
+ a(h_{(1)}\act b) h_{(2)}k_{(1)} \tensor f\ract h_{(3)}k_{(2)}.
\end{equation*}
It is obvious that the first and the third term of this expression 
coincide with the respective terms on the right\dash hand side of
$(*)$. To see that the second terms also coincide, apply Lemma \ref{Leibn}.
\end{proof}

We have just constructed an algebra factorisation of the form 
$(T(V)\lcprod H)\tensor T(V^*)$. To show that it coincides with the
algebra $\D_\beta$, we have to check that the defining relations of 
$\D_\beta$ hold in this algebra factorisation. 
We do not need to check the relation $hv=(h_{(1)}\act v)h_{(2)}$
because it is automatically fulfilled in the semidirect product
algebra $T(V)\lcprod H$. Let us now compute the product $fh$ in 
$(T(V)\lcprod H)\tensor T(V^*)$. We have $fh = c'(f,1\cdot h) = 
1\cdot h_{(1)}\tensor f\ract h_{(2)}=h_{(1)} (f\ract h_{(2)})$, i.e.,
the second defining relation of $\D_\beta$ also holds. 
Finally, $fv = c'(f,v\cdot 1)=\deriv_fv\tensor 1 + v \tensor f$
where $\deriv_f v=\beta(f,v)$. Thus, the commutator relation holds
as well. Theorem \ref{thm_qYD} is proved.
 
\begin{remark}
\label{rem_comm}
It is clear from the proof of the Theorem that 
the operator $\deriv_f b\in T(V)\tensor H$ is the commutator 
$[f,b]$ in $\D_\beta$, for $b\in T(V)$.  
\end{remark}


\section{{Braided} doubles}
\label{sect:brdoubles}

\subsection{Definition of a braided double}

Recall (Corollary \ref{cor_parametrisation}) that any free {braided}
double over a bialgebra $H$ has triangular decomposition of the form 
\begin{equation*}
      \D(V,\delta) = T(V)\tensor H \tensor T(V^*),
\end{equation*}
where $(V,\delta)$ is a quasi\dash Yetter\dash Drinfeld module over $H$. 

We will now be dealing with braided doubles which are no longer
``free''; that is, they have relations within
$T(V)$ and within $T(V^*)$, but still have triangular decomposition
over $H$.  This is formalised as follows.
Denote by $T^{>0}(V)$ the ideal $\oplus_{n>0}V^{\tensorpow n}$ of $T(V)$.

\begin{definition}
A \emph{triangular ideal} in  $\D(V,\delta) = T(V)\lcprod H \rcprod T(V^*)$ is
a two\dash sided ideal of the form 
\begin{equation*}
    I^-\tensor H \tensor T(V^*) + T(V)\tensor H \tensor I^+,
\end{equation*} 
where $I^-$, $I^+$ are subspaces (and automatically two\dash sided
$H$\dash invariant ideals) in $T^{>0}(V)$ and $T^{>0}(V^*)$, respectively.

\end{definition}
\begin{definition}
A \emph{{braided} double} is a quotient of a free
{braided} double modulo a triangular ideal.
\end{definition}

Where $(V,\delta)$ is a quasi\dash YD module over a bialgebra $H$, we 
will refer to a quotient of $\D(V,\delta)$ modulo a triangular ideal 
as a $(V,\delta)$\dash   {braided} double.

\subsection{Hierarchy of braided doubles}
\label{hierarchy}

In what follows, we will use the facts about
triangular decomposition over a bialgebra and triangular ideals,
collected and proved in the Appendix. 

Denote by $\mathcal{D}(V,\delta)$ the set of 
$(V,\delta)$\dash braided doubles. This set is partially ordered 
by the reverse inclusion of triangular ideals $I\subset \D(V,\delta)$. 
Note that if $I_1 \subseteq I_2$ are triangular ideals, 
the double $\D(V,\delta)/I_2$ is a \emph{triangular quotient} of
$\D(V,\delta)/I_1$. This notion is defined in \ref{triang_simple}.

All  $(V,\delta)$\dash braided doubles are triangular quotients of the free
double $\D(V,\delta)$, the greatest element of $\mathcal{D}(V,\delta)$. 
By Corollary \ref{cor_sum}, a sum of triangular ideals
in $\D(V,\delta)$ is a triangular ideal; therefore, all $(V,\delta)$\dash
braided doubles have a common triangular quotient:
\begin{definition}
Let $(V,\delta)$ be a quasi\dash Yetter\dash Drinfeld module over a bialgebra
$H$. Denote by 
\begin{equation*}
    I(V,\delta)\subset T(V), \quad I(V^*,\delta)\subset T(V^*)
\end{equation*}
the pair of $H$\dash invariant two\dash sided ideals such that 
$I_{\D(V,\delta)} = I(V,\delta)\tensor H \tensor T(V^*)+T(V)\tensor  H \tensor I(V^*,\delta)$ is the
largest among (=the sum of all) triangular ideals 
in $\D(V,\delta)$. 
 Define two algebras:
\begin{equation*}
    U(V,\delta)=T(V)/I(V,\delta), \quad U(V^*,\delta) = T(V^*)/I(V^*,\delta).
\end{equation*}
The braided double
\begin{equation*}
    \RD(V,\delta) = \D(V,\delta)/I_{\D(V,\delta)}=
    U(V,\delta)\lcprod H \rcprod U(V^*,\delta)
\end{equation*}
is called the \emph{minimal double} associated to $(V,\delta)$. 
\end{definition} 

There are other distinguished $(V,\delta)$\dash braided doubles which lie
between $\D(V,\delta)$ and $\RD(V,\delta)$ in the above partial order. 
\emph{Quadratic doubles} are braided doubles of the form 
$T(V)/I^- \tensor H \tensor T(V^*)/I^+$ where 
$I^-$ and $I^+$ are quadratic ideals in $T(V)$, $T(V^*)$ 
(i.e., are generated by subsets of
    $V^{\tensorpow 2}$ and $V^{*\tensorpow 2}$, respectively).
The lowest element in this class is the \emph{minimal quadratic double} 
\begin{equation*}
     \RD_{\mathit{quad}}(V,\delta) \cong T(V)/I_{\mathit{quad}}(V,\delta) \lcprod H
\rcprod T(V^*)/I^*_{\mathit{quad}}(V,\delta). 
\end{equation*}
One can deduce from Theorem
\ref{ker_qbfact} below that 
\begin{equation*}
I_{\mathit{quad}}(V,\delta) =  \lgen I(V,\delta)\cap V^{\tensorpow 2}\rgen, \quad 
I^*_{\mathit{quad}}(V,\delta) =  \lgen I(V^*,\delta)\cap V^{*\tensorpow 2}\rgen,
\end{equation*}
where $\lgen\dots\rgen$ denotes the two\dash sided ideal with given generators.
There is a canonical surjection
$\RD(V,\delta)_{\mathit{quad}}\twoheadrightarrow 
\RD(V,\delta)$,  and we may regard the double $\RD_{\mathit{quad}}(V,\delta)$ as a `first
approximation' to $\RD(V,\delta)$.

\subsection{Relations in the minimal double}

Our goal in this Section is to describe the largest
triangular ideal in the 
the free braided double $\D(V,\delta)$ --- or, the same, the relations
in the algebras $U(V,\delta)$ and $U(V^*,\delta)$ --- in terms of the
quasi\dash Yetter\dash Drinfeld structure on $V$. The first step is
the following 

\begin{lemma}
\label{triang_ideals_bdouble}
Triangular ideals in $\D(V,\delta)$ are subspaces $J\subset \D(V,\delta)$
of the form
$J = J^-H T(V^*) + T(V) H J^+$, where 
$J^- \subset T^{>0}(V)$ and $J^+\subset T^{>0}(V^*)$ are 
$H$\dash invariant two\dash sided ideals, such that
\begin{equation*}
       [f,J^-] \subset J^-\tensor H, \qquad
       [J^+, v] \subset H \tensor J^+
\end{equation*}
for all $f\in V^*$, $v\in V$.
\end{lemma}
\begin{proof}
Triangular ideals are described by Proposition \ref{triang_ideals},
and we only have to adapt that description to the case of {braided}
doubles. Denote $U^-=T(V)$, $U^+=T(V^*)$, and let $\epsilon^\pm\colon
U^\pm\to\field$ be the projections to degree zero component. 
By Proposition \ref{triang_ideals}, triangular ideals are of
the form $J = J^-\tensor H\tensor \Uplus + \Uminus \tensor H \tensor J^+$,
where $J^\pm \subset \ker \epsilon^\pm$ are $H$\dash invariant
two\dash sided ideals in the algebras $U^\pm$, such that 
$\Uplus \cdot J^-$, $J^+\cdot \Uminus$ lie in $J$. 
Since $V$ (resp.\ $V^*$) generates $\Uminus$ (resp.\ $\Uplus$) 
as an algebra, this is equivalent to 
\begin{equation*}
f \cdot J^-, \ J^+\cdot v \ \subset \  J^-H\Uplus + \Uminus
H J^+
\end{equation*}
for all $v\in V$ and $f\in V^*$. 
Now, since $J^- \cdot f$ obviously lies in $J^-H\Uplus$ 
and $v\cdot J^+$ lies in $\Uminus H J^+$, we may replace products 
by commutators and rewrite this condition as
\begin{equation*}
[f, J^-], \quad [J^+,v] 
\quad \subset \quad J^-H\Uplus + \Uminus H J^+. 
\end{equation*}
 Finally, we observe that by Remark \ref{rem_comm},
$[f,J^-]$ lies in $\Uminus H$, and, similarly, $[J^+,v]$ lies in
$H \Uplus$. Therefore, the condition splits into two separate inclusions,
$[f, J^-] \subset J^-H$ and $[J^+,v]\subset H J^+$.
\end{proof}
\begin{remark}
\label{rem:sum2}
Note that the Lemma implies that any triangular ideal in 
$\D(V,\delta)$ is a sum of two triangular ideals of special form:
one $J^-H T(V^*)$  and the other $T(V) H J^+$.
\end{remark}

Our next step is to show that the defining ideals $I(V,\delta)$, $I(V^*,\delta)$ of
the minimal double are graded ideals in $T(V)$, $T(V^*)$, respectively.
We call a triangular ideal $I^-\tensor H \tensor T(V^*)+T(V)\tensor H \tensor
I^+$ \emph{graded}, if $I^-$ (resp.\ $I^+$) is a graded
ideal in $T^{>0}(V)$ (resp.\ $T^{>0}(V^*)$). 
A \emph{graded braided double} is a quotient of a free braided double
by a graded triangular ideal. 

\begin{lemma}
\label{lem:graded}
Any triangular ideal in a free {braided} double is contained in a graded
triangular ideal.
\end{lemma}
\begin{proof}
Let $\D(V,\delta)\cong T(V)\lcprod H \rcprod T(V^*)$ be a free {braided}
double, and $J$ be a triangular ideal in $\D(V,\delta)$. Denote $\Uminus=T(V)$
and $\Uplus=T(V^*)$.  
By Remark \ref{rem:sum2}, it is enough to consider the cases $J=J^- H
\Uplus$ and $J=\Uminus H J^+$. We will assume $J=J^- H \Uplus$, the
other case being analogous. By Lemma
\ref{triang_ideals_bdouble}, $J^-$ is a two\dash sided ideal in 
$\Uminus_{>0}$ is an $[f,J^-]\subset J^-\tensor H$ 
for any $f\in V^*$.
 
Denote by $p_n$ the projection map from $\Uminus$ onto its $n$th homogeneous
component $\Uminus_n$. 
Put $J^-_n = p_n(J^-)$ and let $J^-_{\mathrm{gr}} = \oplus_{n>0}
J^-_n$. Let us check that the space $J^-_{\mathrm{gr}}$ satisfies 
the conditions of Lemma  \ref{triang_ideals_bdouble}. Indeed, 
$J^-_{\mathrm{gr}}$ is a two\dash sided ideal in $\Uminus$, because 
$v J^-_n = p_{n+1}(v J^-) \subset p_{n+1}(J^-)=J^-_{n+1}$ for any
$v\in V$, and similarly $J^-_n v\subset J^-_{n+1}$. 
Subspaces $\Uminus_n$ are $H$\dash submodules of $\Uminus$, and $p_n$
are $H$\dash equivariant maps; 
thus, $J^-_{\mathrm{gr}}$ is $H$\dash invariant. 
By construction, $J^-_{\mathrm{gr}}$ lies in
$\Uminus_{>0}=\ker \epsilonminus$ and contains $J^-$. 
Finally, it is clear (e.g. from the definition of the operator
$\deriv_f=[f,\cdot]$ in the proof of Theorem \ref{thm_qYD})
that the commutator $[f,\cdot]$ is lowering the degree
in $\Uminus \tensor H$ by one:
\begin{equation*}
            [f,\Uminus_n] \subset \Uminus_{n-1}\tensor H,
\end{equation*}
therefore $[f,J^-_n] \subset J^-_{n-1}\tensor H$. 
Thus, $J^- H \Uplus$ is contained in a graded triangular ideal 
$J^-_{\mathrm{gr}} H \Uplus$ of $\D(V,\delta)$. 
\end{proof}

\begin{corollary}
\label{triang_quot}
$I(V,\delta)$, $I(V^*,\delta)$ are graded ideals in $T(V)$, $T(V^*)$, respectively. 
\end{corollary}
\begin{proof}
Observe that by the Lemma, the largest triangular ideal in a free {braided}
double is graded.  
\end{proof}

\subsection{Computation of the ideals $I(V,\delta)$, $I(V^*,\delta)$}

To proceed with the computation of the maximal triangular ideal
$I(V,\delta)HT(V^*)+T(V)HI(V^*,\delta)$ of $\D(V,\delta)$, we assume that $H$ is a Hopf
algebra. 
To make the exposition concise, 
let us focus on the ideal $I(V,\delta)$; 
we will state the final result for $I(V^*,\delta)$ later in Remark
\ref{rem:righthanded}. 
We say that a subspace $W\subset T^{>0}(V)$ is
``\emph{preserved by commutators}'', if $[f,W]\subset
W\tensor H$ for any $f\in V^*$. 

\begin{lemma}
\label{lem:computation}
$I(V,\delta)$ is the  maximal subspace in $T^{>0}(V)$,
preserved by commutators.
\end{lemma}
\begin{proof}
Let $W$ be a subspace of $T^{>0}(V)$, preserved by commutators.
Then $W'=H\act W$ is a subspace of $T^{>0}(V)$.
Let us show that $W'$ is also preserved by commutators.
By Lemma \ref{Leibn}, 
\begin{equation*}
           [f,h\act b]=\deriv_f(h\act b) = h_{(1)}\cdot (\deriv_{f\ract
           h_{(2)}}b)\cdot Sh_{(3)} \quad\text{for }b\in T(V); 
\end{equation*}
applying this to $b\in W$ shows that $[f,W']$ lies in $H\cdot W \cdot
H\subset W'H$. 

It follows that the maximal subspace preserved by commutators is $H$\dash
invariant. Assume now that $W$ is an $H$\dash invariant subspace of
$T^{>0}(V)$, preserved by commutators. Let us show that $W$
is contained in an $H$\dash invariant two\dash sided ideal in
$T^{>0}(V)$, preserved by commutators. Indeed, apply $\deriv_{f}$ to the  
ideal $T(V) \cdot W \cdot T(V)$.
By the Leibniz rule, the result lies in 
$(T(V) H)\cdot W\cdot T(V) + T(V)\cdot (WH)\cdot T(V) + T(V) \cdot W
\cdot (T(V) H)$.
Since $W$ and $T(V)$ are  $H$\dash invariant subspaces of
$T(V)$, this coincides with  $(T(V) \cdot W \cdot T(V))\tensor H$. 

Thus, the maximal subspace of $T^{>0}(V)$, preserved by commutators,
is an $H$\dash invariant two\dash sided ideal with 
this property. But the maximal among such ideals is $I(V,\delta)$.
\end{proof}

\subsection{Quasibraided integers and quasibraided factorials}

We are ready to describe the graded components of the ideal
$I(V,\delta)\subset T(V)$ as kernels of \emph{quasibraided factorials},
which we now introduce.

\begin{definition}
\label{def_qbfact}
Let $(V,\delta)$ be a quasi\dash Yetter\dash Drinfeld module over a bialgebra
$H$, with action $\act$ and quasicoaction $\delta(v)=v^{[-1]}\tensor
v^{[0]}$. The \emph{quasibraided integers} are  maps
\begin{align*}
   &\widetilde{[n]}_\delta\colon V^{\tensorpow n} \to V^{\tensorpow
     n-1}\tensor H \tensor V, 
\\
   &\widetilde{[n]}_\delta (v_1\tensor \dots \tensor v_n) = 
   \sum_{i=1}^n v_1\tensor \dots \tensor v_{i-1}\tensor
    {v_i^{[-1]}}_{(1)}\act (v_{i+1}\tensor \dots \tensor v_n) \tensor 
    {v_i^{[-1]}}_{(2)} \tensor {v_i^{[0]}}.
\end{align*}
The \emph{quasibraided factorials} are maps
\begin{equation*}
 \widetilde{[n]}!_\delta\colon V^{\tensorpow n} \to (H\tensor
 V)^{\tensorpow n},
\qquad 
\widetilde{[n]}!_\delta = 
(\widetilde{[1]}_\delta \tensor \id_{H\tensor V}^{\tensorpow n-1})
\circ
(\widetilde{[2]}_\delta \tensor \id_{H\tensor V}^{\tensorpow n-2})
\circ \dots \circ \widetilde{[n]}_\delta.
\end{equation*}
We also put $\widetilde{[0]}!_\delta=1$.
\end{definition}
\begin{lemma}
\label{comm_formula}
The commutator of $f\in V^*$ and $b\in V^{\tensorpow n}$ in
the free braided double $\D(V,\delta)$ is given by 
\begin{equation*}
        [f,b] = (\id_V^{\tensorpow n-1}\tensor
                 \id_H \tensor \langle f,-\rangle)\widetilde{[n]}_\delta b.
\end{equation*}
\end{lemma}
\begin{proof} 
By Remark \ref{rem_comm}, $[f,b]=\deriv_f b$ where the operator
  $\deriv_f$ was introduced in the proof of Theorem \ref{thm_qYD}. 
Recall that $\beta(f,v)=v^{[-1]} \langle f,v^{[0]} \rangle $, and  
rewrite the formula for $\deriv_f$ in terms of the quasicoaction:
\begin{equation*}
  \deriv_f(v_1\tensor \dots \tensor v_n) = 
  \sum_{i=1}^n (v_1\tensor \dots \tensor v_{i-1})\cdot
  v_i^{[-1]} \langle f,v_i^{[0]}\rangle 
  \cdot (v_{i+1}\tensor \dots \tensor v_{n}).
  \end{equation*}
The Lemma now follows from the relations in the semidirect product
$T(V)\lcprod H$. 
\end{proof}
\begin{theorem}
\label{ker_qbfact}
Let $(V,\delta)$ be a quasi\dash YD module over a Hopf algebra $H$. 
The ideal $I(V,\delta)$ in $T(V)$ is given by
\begin{equation*}
     I(V,\delta) = \mathop{\oplus}\limits_{n=1}^\infty \ker \widetilde{[n]}!_\delta.
\end{equation*}
\end{theorem}
\begin{proof}
The ideal $I(V,\delta)\subset T^{>0}(V)$ is graded by Corollary
\ref{triang_quot}. Write  $I(V,\delta)=I_0\oplus I_1{\oplus \dots}$, where
$I_n=I(V,\delta)\cap V^{\tensorpow n}$. By Lemma
\ref{lem:computation}, $I(V,\delta)$ is the maximal subspace of $T^{>0}(V)$ 
preserved by commutators, which in terms of the graded components
rewrites as 
\begin{equation*}
         I_n = \{ b\in V^{\tensorpow n} : [f,b]\in I_{n-1}\tensor H
                  \text{ for all }f\in V^*\}, 
\qquad n\ge 1.
\end{equation*}
Let us show that $I_n=\ker\widetilde{[n]}!_\delta$. This is
true for $n=0$ (because $I_0=0$); assume this to be true for $n-1$. 
Substitute the commutator $[f,b]$ with its expression via
the quasibraided integer from Lemma \ref{comm_formula}:
\begin{equation*}
         I_n = \{ b\in V^{\tensorpow n} : 
                 (\id_V^{\tensorpow n-1}\tensor
                 \id_H \tensor \langle f,- \rangle)\widetilde{[n]}_\delta b  
                  \in (\ker \widetilde{[n-1]}!_\delta)\tensor H\}.
\end{equation*}
Since this holds for arbitrary $f\in V^*$, 
the space $I_n$ consists of $b\in V^{\tensorpow n}$ such that 
$\widetilde{[n]}_\delta b$ is in $(\ker
\widetilde{[n-1]}!_\delta)\tensor H\tensor V$. That is, $I_n=\ker
\widetilde{[n]}!_\delta$. The Theorem follows by induction. 
\end{proof}

\begin{remark}
\label{rem:righthanded}
The ideal $I(V^*,\delta)$ of $T(V^*)$ has a description of the same nature. 
Consider the right quasicoaction $\delta_r\colon V^*\to V^*\tensor H$,
$f\mapsto f^{[0]}\tensor f^{[1]}$, which is given by
$\beta(f,v)=\langle f^{[0]},v \rangle f^{[1]}$. This, together  with the right
action of $H$ on $V^*$, gives rise to right\dash handed
quasibraided integers $\widetilde{[n]}_{\delta_r}\colon
V^{*\tensorpow n}\to V^*\tensor H\tensor V^{*\tensorpow n-1}$
and to right\dash handed
quasibraided factorials $\widetilde{[n]}!_{\delta_r}\colon
V^{*\tensorpow n}\to (V^*\tensor H)^{\tensorpow n}$. One has
$I(V^*,\delta)=\oplus_n \ker  \widetilde{[n]}!_{\delta_r}$. 
\end{remark}

 The following Corollary gives a useful criterion of minimality of
a graded braided double. 
\begin{corollary}(Minimality criterion)
\label{cor:criterion}
Let $\Dbl = \Uminus \lcprod H \rcprod \Uplus$ be a $(V,\delta)$\dash
braided double which is graded: $\Uminus = \oplus_{n=0}^\infty \Uminus_n$, 
 $\Uplus = \oplus_{n=0}^\infty \Uplus_n$. Then $\Dbl$ is a minimal
 double, when and only when

(a) if $b\in \Uminus$, $[f,b]=0$ for all $f\in V^*$, then $b\in
\Uminus_0$;

(b) if $\phi\in \Uplus$, $[\phi,v]=0$ for all $v\in V$, then $\phi\in
\Uplus_0$.
\end{corollary}
\begin{proof}
$\Uminus\ne U(V,\delta)$, if and only if $\Uminus_{>0}$ contains a
  graded subspace 
preserved by commutators $[f,\cdot]$ and not contained in degree $0$
of grading. Such a subspace exists if and only if there is a
homogeneous element $b$ of positive degree (in the lowest degree component of
the subspace) which commutes with all $f\in V^*$. Similarly, 
$\Uplus\ne U(V^*,\delta)$, if and only if there is
  $\phi\in\Uplus_{>0}$ which commutes with all $v\in V$.
\end{proof}

The next Corollary will be useful in constructing graded braided doubles
which are not minimal.
\begin{corollary}
\label{cor:nonminimal}
Let $(V,\delta)$ be a quasi\dash YD module over a Hopf algebra $H$.
Let $I^-\subset T^{>0}(V)$ be a graded two\dash sided ideal. 
Then $I^-\tensor H \tensor T(V^*)$ is a triangular ideal in 
$\D(V,\delta)$, if and only if $I^-$ has an $H$\dash invariant 
generating space $R=\oplus_{n>0} R_n$, $R_n\subset V^{\tensorpow n}$,
such that
\begin{equation*}
      \widetilde{[n]}_\delta R_n \subseteq R_{n-1}\tensor H \tensor V
\end{equation*} 
(assuming $R_0=0$). 
A similar statement holds for ideals $I^+\subset  T^{>0}(V)$ and 
the right\dash handed quasibraided integers $\widetilde{[n]}_{\delta_r}$. 
\end{corollary}
\begin{proof}
By Lemma \ref{comm_formula}, 
the equation on $R$ is equivalent to saying that $[f,R_n]\subseteq
R_{n-1}\tensor H$ for all $f\in V^*$. 
This implies that the ideal $I^-=\lgen R \rgen$ is an $H$\dash
equivariant ideal in $T^{>0}(V)$, such that 
$[f,I^-]\subseteq I^-\tensor H$. By Lemma \ref{triang_ideals_bdouble}, 
$I^-HT(V^*)$ is a triangular ideal in $\D(V,\delta)$. This is the `if'
part; the `only if' part follows by putting $R=I^-$.
\end{proof}

\subsection{Standard modules for braided doubles}
\label{standmod}

Clearly, minimality of a braided double should influence its
representation theory. Let us mention a construction which yields a
family of standard modules for an algebra with triangular
decomposition (known as Verma modules in Lie theory and used also for
rational Cherednik algebras). 

Let $\Dbl = \Uminus \lcprod H \rcprod \Uplus$ be a graded
$(V,\delta)$\dash braided double over $H$. To any irreducible 
left representation $\rho\colon H \to \End(L_\rho)$ of $H$ is 
associated a left $\Dbl$\dash module:
\begin{equation*}
       M_\rho = \mathrm{Ind\,}_{H\tensor \Uplus}^{\Dbl} \ (\rho \tensor
       \epsilonplus).
\end{equation*}
As a vector space, $M_\rho$ is the tensor product $\Uminus \tensor L_\rho$. 
The standard modules $M_\rho$ are crucial in the
Bernstein-Gelfand-Gelfand theory of category $\mathcal O$ for
$U(\mathfrak g)$ \cite{BGG} and its more
recent version for rational Cherednik algebras as in \cite{GGOR}.

Observe, however, that all $M_\rho$ are reducible $\Dbl$\dash modules
unless $\Uminus = U(V,\delta)$. Indeed, let $\overline{M}_\rho\cong
U(V,\delta)\tensor L_\rho$ be the induced module for the minimal
double $\RD(V,\delta)$. Then $\overline{M}_\rho$, which is an $\Dbl$\dash
module via the quotient map $\Dbl \twoheadrightarrow \RD(V,\delta)$,
is a quotient of $M_\rho$. We therefore suggest $\{\overline{M}_\rho
\mid \rho \in \mathrm{Irr}(H)\}$ as a family of standard modules for
any $(V,\delta)$\dash braided double.

\subsection{The Harish-Chandra pairing and minimality}

We will now suggest a useful method for  proving minimality of a
given braided double. 
Let us introduce the 
Harish\dash Chandra pairing in braided doubles; in fact, this can be
done for any algebra with triangular decomposition over a bialgebra,
see Appendix, \ref{HCpairing}.

\begin{definition}
Let $A=\Uminus \lcprod H \rcprod \Uplus$ be a graded braided double over a
bialgebra $H$. Let $\epsilon^\pm \colon U^\pm \to \field$ be
projections onto degree $0$ components in $U^\pm$. 
Denote 
\begin{equation*}
   p_H = \epsilonminus \tensor \id_H \tensor \epsilonplus  
\colon A \twoheadrightarrow H.
\end{equation*}
The \emph{Harish\dash Chandra pairing} in $A$ is 
\begin{equation*}
   (\cdot,\cdot)_H \colon \Uplus \times \Uminus \to H, 
   \qquad
   (\phi,b)_H = p_H(\phi b). 
\end{equation*}
\end{definition}
The terminology is inspired by the example of the universal enveloping
algebra $U(\mathfrak g)$ of a semisimple Lie algebra $\mathfrak g$. 
The next Theorem follows from a general result on algebras with
triangular decomposition (Proposition \ref{HCkernels}):
\begin{theorem}
\label{th:HCmin}
If the Harish\dash Chandra pairing in a braided double $A$ is
non\dash degenerate, then $A$ is a minimal double. 
\qed
\end{theorem}
We will now give a formula for the Harish\dash Chandra pairing in a
free braided double $\D(V,\delta)$ in terms of quasibraided factorials.  
(It works as well for any graded
$(V,\delta)$\dash braided double.) We will use the notation
\begin{align*}
  &m_H\colon (H\tensor V)^{\tensorpow n} \to H \tensor V^{\tensorpow n},
  \\
  &m_H(h_1 \tensor v_1 \tensor h_2\tensor v_2 \tensor \dots 
      \tensor h_n \tensor v_n) = 
    h_1 h_2 \dots h_n \tensor v_1 \tensor v_2 \tensor \dots \tensor
  v_n. 
\end{align*}
\begin{proposition}
\label{prop:HCformula}
Let $(V,\delta)$ be a quasi\dash Yetter\dash Drinfeld double over a bialgebra
$H$. The Harish\dash Chandra pairing in
$\D(V,\delta)$ is given by
\begin{equation*}
    \phi \in V^{*\tensorpow n}, \ b\in V^{\tensorpow n} 
 \quad \mapsto \quad
(\phi,b)_H = (\id_H \tensor \langle  \phi, - \rangle_{V^{\tensorpow n}})m_H
 \widetilde{[n]}!_\delta b
\end{equation*}
and $(V^{*\tensorpow n},V^{\tensorpow m})_H=0$ if $n\ne m$. 
\end{proposition}
\begin{proof}
Recall the operator $\deriv_f\colon T(V)\to T(V)\lcprod H$, $\deriv_f
b = [f,b]$ (commutator in the double $\D(V,\delta)$), and extend it to
$T(V)\lcprod H$ by 
\begin{equation*}
   \deriv_f(b\tensor h) = (\deriv_f b)\cdot h, \quad b\in T(V), \ h\in
   H, \ f\in V^*.
\end{equation*}
Consider the subspace $A^+=T(V)\tensor H \tensor T^{>0}(V^*)$ of
$\D(V,\delta)$. It has the property that $V^* A^+ \subset A^+$ and
$A^+ H \subset A^+$. Therefore, for any $b\in T(V)$, $h\in H$
and  $f\in V^*$ we have 
\begin{equation*}
     fbh \simeq \deriv_f (bh) \quad \text{modulo }A^+.
\end{equation*}
Let $\phi=f_1 \tensor f_2 \tensor \dots \tensor f_n\in
V^{*\tensorpow n}$ and $b\in V^{\tensorpow m}$.
The subspace $A^+$ lies in the kernel of
the map $p_H$, therefore 
\begin{equation*}
    (\phi,b)_H = p_H(\phi b) = p_H(\deriv_{f_1}\deriv_{f_2} \dots
    \deriv_{f_n} b). 
\end{equation*}
If $m=n$, it is easy to deduce from Lemma \ref{comm_formula} that the
right\dash hand side equals 
$(\id_H\tensor \langle f_1,- \rangle \tensor \dots \tensor \langle f_n,-\rangle) m_H
\widetilde{[n]}!_\delta$. 
If $m>n$, then  $\deriv_{f_1} \dots
    \deriv_{f_n} b$ lies in the space $V^{m-n}\tensor H\subset \ker p_H$. 
Finally, if $m<n$, then $\deriv_{f_1} \dots
    \deriv_{f_n} b = 0$.
\end{proof}

\subsection{Non-degeneracy of the Harish-Chandra pairing and ideals}

We would like to make a simple observation concerning ideals in
braided doubles, which will not be used in the sequel. 
It is here to highlight a possible direction of further research. 

The study of ideals in universal enveloping algebras $U(\mathfrak g)$
was a significant topic in representation theory in the second half of
the 20th century. 
It has been observed, however, that some important results on
ideals may be deduced from the fact that the algebra has a triangular
structure of a certain kind, cf.\ \cite{Ginzb_prim}. This allows one to
extend such results to objects of more recent vintage
such as rational Cherednik algebras. 
 
Let us extend the Harish\dash Chandra pairing in a braided double
$A=\Uminus \lcprod H \rcprod \Uplus$ to obtain pairings
\begin{equation*}
 (\cdot,\cdot)_H \colon \Uplus\times  \Uminus H  \to H,
\quad
 (\cdot,\cdot)_H \colon H\Uplus \times \Uminus \to H,
\end{equation*}
both defined by the same formula $(y,x)_H=p_H(yx)$ and denoted by the
same symbol. We say that the Harish\dash Chandra pairing in $A$ is
\emph{strongly non\dash degenerate}, if these two extensions are
non\dash degenerate pairings. 

\begin{remark}
One can show that if a scalar pairing $\lambda((\cdot,\cdot)_H)$ is
non\dash degenerate for an algebra homomorphism $\lambda\colon H \to
\field$ (e.g.\ for the counit $\lambda=\epsilon$), then the
Harish\dash Chandra pairing is strongly non\dash degenerate. 
\end{remark}

\begin{proposition}
\label{good_ideals}
Let $\Dbl = \Uminus \lcprod H \rcprod \Uplus$ be a braided double with 
strongly non\dash degenerate Harish\dash Chandra pairing. 
Then $p_H(I) \ne 0$ for any  non\dash zero two\dash sided ideal $I$ of $A$. 
\end{proposition}
Note that the Proposition links
ideals in the algebra $\Dbl$ (in general with no Hopf algebra
structure) and ideals in the Hopf algebra $H$.
In particular, if $H$ is commutative or super\dash commutative, this
may allow one to define associated (super)varieties in
$\mathit{Spec}(H)$ for two\dash sided ideals in $\Dbl$ (see, e.g.,
\cite{Ginzb_prim}).  

\begin{proof}[Proof of Proposition \ref{good_ideals}] 
Denote the subalgebra $H\Uplus$ of $\Dbl$ by $B^+$, and denote by
$N^\pm$ the respective kernels of $\epsilon^\pm\colon U^\pm \to \field$.
Let $p_{B^+} = \epsilonminus \tensor \id_{B^+}$ be the
projection onto $B^+$. Let us show that if
$p_{B^+}(I)\ne 0$ for an ideal $I$ of $\Dbl$, then
$p_H(I)\ne 0$. Indeed, $p_{B^+}(I)\ne 0$ means that $I$
contains an element $\phi\in \phi'+N^-H\Uplus=\phi'+N^- \Dbl$ for some non\dash
zero $\phi'\in B^+$. 
By strong non\dash degeneracy, there is $ b \in
\Uminus$ such that  
$p_H(\phi' b )=( \phi', b )_H \ne 0$. Since $\phi' b $ differs from
$\phi b$
by an element from $N^-\Dbl$, which is in the kernel of the
Harish\dash Chandra projection $p_H$, one has $p_H(\phi b)\ne 0$; it
remains to note that $\phi b\in I$. 

We now have to check that if $I$ is a non\dash zero two\dash sided
ideal in $\Dbl$, then $\mathrm{pr}_{B^+}(I)\ne 0$.
By Theorem~\ref{th:HCmin}, $\Dbl$ is a minimal double, hence $\Uplus$
is graded by Corollary~\ref{triang_quot}. 
Choose a graded basis $\mathcal B$ of $\Uplus$.
Take a non\dash zero element in $I$ and write it in the form
$a_1\tensor u_1 + \dots + a_n \tensor u_n$, where $a_i$ are nonzero
elements of $\Uminus H$ and $u_i$ are in $\mathcal B$.
Without loss of generality, assume that $u_1, \dots, u_k$ ($1\le k \le
n$) are of lowest degree, say $m$, among all $u_i$. 
Using non\dash degeneracy, find an element $v\in \Uplus$ such that    
$(v,a_1)_H=h_1\ne 0$.  This means that $v a_1$ lies in 
$1\tensor h_1 \tensor 1 + N^-\tensor H \tensor \Uplus + \Uminus
\tensor H \tensor N^+$, therefore 
$v a_1 u_1$ is in $1\tensor h_1\tensor u_1 + N^-\Dbl + \Dbl
(N^+)^{m+1}$.
If $(v,a_i)_H$ is denoted by $h_i$ ($h_i$ may be zero for $i>1$), 
then 
\begin{equation*}
    \sum\nolimits_{1\le i \le n} v a_i u_i 
    \equiv
    \sum\nolimits_{1\le i \le k} 1\tensor h_i \tensor u_i 
    \quad
    (\mathrm{mod}\ N^-\Dbl + \Dbl (N^+)^{m+1}).
\end{equation*}
Projecting the element $\sum_i v a_i u_i$ of $I$ 
onto $B^+$, then projecting further onto the quotient $B^+/B^+ (N^+)^{m+1}
= H \tensor (\Uplus / (N^+)^{m+1})$ gives
$\sum_{1\le i \le k} h_i\tensor (u_i\ \mathrm{mod}\ (N^+)^{m+1})$. This
is not zero, since $h_1\ne 0$ and $u_1,\dots,u_k$ are linearly
independent modulo $(N^+)^{m+1}$. Thus, $\mathrm{pr}_{B^+}(I)\ne 0$ as
required. 
\end{proof} 

\subsection{Two examples of braided doubles}

We would like to finish this Section with two (counter)examples.
The first example shows that a minimal double may
have degenerate Harish\dash Chandra pairing.

\begin{example}
\label{ex:nilpotent}
Let $\field x$ be a one\dash dimensional module
over a Hopf algebra $H$, with trivial action $h\act
x=\epsilon(h)x$. By \ref{ex:trivial}, any quasi\dash Yetter\dash
Drinfeld structure on $\field x$ is given by
$\delta(x)=a\tensor x$, where 
\begin{itemize}
\item[(1)] $a$ is a central element in $H$. 
\end{itemize}
We will write $a$ instead of $\delta$. 
The free braided double $\D(\field x, a)$ has triangular decomposition 
$\field[x] \tensor H \tensor \field[y]$ where $y$ is the spanning
vector of the module dual to $\field x$. 
The quasibraided factorial is given by 
\begin{equation*}
       \widetilde{[n]}!_a (x^{\tensorpow n}) = n!\, (a\tensor
       x)^{\tensorpow n},
\end{equation*}
It follows that $\D(\field x, a)$ is a minimal double, if and only if 
\begin{itemize}
\item[(2)] $a\ne 0$ and $\field$ is of characteristic zero. 
\end{itemize}
Assume $\mathit{char}\ \field=0$. 
By Proposition \ref{prop:HCformula}, the Harish\dash 
Chandra pairing in $\D(\field x, a)$ is given by 
\begin{equation*}
       (y^{\tensorpow n},x^{\tensorpow n})_H = n!(y,x)^n a^n.
\end{equation*}
The Harish\dash Chandra pairing is degenerate, if and only if 
\begin{itemize}
\item[(3)] $a$ is a nilpotent element in $H$.
\end{itemize}
Thus, any  Hopf algebra in characteristic zero with a nonzero central
nilpotent element gives rise to a braided double which is minimal but
has degenerate Harish\dash Chandra pairing. 
\end{example}
\begin{remark}[Kaplansky's third conjecture]
\label{rem:3rd_conj}
A conjecture that a Hopf algebra with the above properties 
\emph{does not exist}, was number 3 in a list of ten conjectures on Hopf algebras
published by I.~Kaplansky in 1975. For some time, this third
conjecture has been known to be false. The survey \cite{So} contains
 historical remarks and a comprehensive account of progress made in
 relation to Kaplansky's conjectures, including counterexamples
 to the third conjecture. 

In order to complete Example~\ref{ex:nilpotent}, we give 
a new explicit counterexample to the third Kaplansky's conjecture
(a Hopf algebra of dimension $8$) below in Example~\ref{ex:3rd_conj}.   
\end{remark}

The second construction of a braided double shows
that the Hilbert series of the two  graded ``halves'', $U(V,\delta)$ and
$U(V^*,\delta)$, of a minimal double may be different (even in degree
$1$). In particular, $V$ which is a ``space of generators'' 
for the algebra $U(V,\delta)$, may not embed injectively in $U(V,\delta)$. 

\begin{example}
\label{ex:pathological}

Let $\mathbb{Z}_2 = \{1,s\}$ be the two\dash element group.  
We take $V$ to be a two\dash dimensional $\field
\mathbb{Z}_2$\dash module where
$s$ acts as a multiplication by $-1$. Let $v_1,v_2$ be a basis of $V$
and $f_1, f_2$ be the dual basis of $V^*$. 
Consider a $\field\mathbb{Z}_2$\dash valued pairing between $V^*$ and $V$, defined on
the bases as follows:
\begin{equation*}
\beta(f_1,v_1)=1, \quad \beta(f_1, v_2)=s, 
\quad \beta(f_2,v_1)=\beta(f_2, v_2)=0.
\end{equation*}
It is easy to see that the pairing $\beta$ satisfies the Yetter\dash
Drinfeld condition as in Theorem~\ref{thm_qYD}. Hence $V$ becomes a
quasi\dash Yetter\dash Drinfeld module over $\field\mathbb{Z}_2$. The quasicoactions
$\delta\colon V \to \field\mathbb{Z}_2 \tensor V$ and $\delta_r\colon V ^*\to
V^*\tensor \field\mathbb{Z}_2$ are given by 
\begin{equation*}
 \delta(v_1) = 1\tensor v_1, \quad \delta(v_2) = s\tensor v_1;
\qquad 
\delta_r(f_1)=f_1\tensor 1 + f_2\tensor s, \quad
\delta_r(f_2)=0.
\end{equation*}
It follows that
\begin{equation*}
  I(V,\delta)\cap V^{\tensorpow 1} = \ker \delta = 0, \qquad
  I(V^*,\delta)\cap V^{*\tensorpow 1} = \ker \delta_r = \field f_2.
\end{equation*}
The degree $1$ component in the graded algebra $U(V,\delta)$ has
dimension $2$, whereas the  degree $1$ component in $U(V^*,\delta)$ is
one\dash dimensional. It is not difficult to check that 
$f_1^{\tensorpow n}$ does not vanish under the 
quasibraided factorial $\widetilde{[n]}!_{\delta_r}$ in characteristic
$0$; therefore, $U(V^*,\delta)$ is isomorphic to the polynomial
algebra $\field[f_1]$. 

The algebra $U(V,\delta)$ is, however, not commutative. 
One may compute
\begin{align*}
   &\widetilde{[2]}!_\delta(v_1\tensor v_2) = 1\tensor v_1 \tensor s
   \tensor v_1 + s\tensor v_1 \tensor 1 \tensor v_1, 
\\
   &\widetilde{[2]}!_\delta(v_2\tensor v_1) = s\tensor v_1 \tensor 1
   \tensor v_1  - 1\tensor v_1 \tensor s\tensor v_1,
\end{align*}
hence (recall Theorem \ref{ker_qbfact})
the commutator $v_1\tensor v_2 - v_2 \tensor v_1$ is not in $I(V,\delta)$.
\end{example}

\section{Braided Heisenberg doubles}
\label{sect:bhd}

In the previous Section, we associated to every 
quasi\dash Yetter\dash Drinfeld module $(V,\delta)$, $\dim V<\infty$, over a Hopf
algebra $H$ a pair of two\dash sided graded $H$\dash invariant ideals
\begin{equation*}
I(V,\delta)\subset T(V), \quad I(V^*,\delta) \subset T(V^*)
\end{equation*}
which are the defining ideals in the minimal double 
$\RD(V,\delta)\cong U(V,\delta)\tensor H$ $\tensor U(V^*,\delta)$. 
However, the description of $I(V,\delta)$, $I(V^*,\delta)$ as kernels of quasibraided
factorials may be far from satisfactory: the factorials are
operators from $V^{\tensorpow n}$ to $(H\tensor V)^{\tensorpow n}$,
which possibly is an infinite\dash dimensional space. 
We have also seen that the algebras $U(V,\delta)$ and $U(V^*,\delta)$ 
may not look similar at all (Example \ref{ex:pathological}). 

The goal of this section is to analyse $I(V,\delta)$, $I(V^*,\delta)$ and the
minimal double $\RD(V,\delta)$ in the case when $(V,\delta)$ is a Yetter\dash
Drinfeld module over a Hopf algebra $H$ (without the ``quasi''
prefix). 

Write the $H$\dash coaction on $V$ as $\delta\colon v \mapsto v^{(-1)}\tensor v^{(0)}$. 
Let $\Psi\colon V\tensor V \to V \tensor V$ be the braiding
on the space $V$ induced by the Yetter\dash Drinfeld structure:
\begin{equation*}
   \Psi(v\tensor w) = v^{(-1)}\act w \tensor v^{(0)}.
\end{equation*}
We will show that $I(V,\delta)$ is the kernel of the \emph{Woronowicz
  symmetriser} $\Wor(\Psi)\in \End T(V)$. 
This is because our
factorial $\widetilde{[n]}!_\delta$ specialises, in the case of
Yetter\dash Drinfeld module, to the braided factorial of Majid, which
is an endomorphism of $V^{\tensorpow n}$. Braided doubles associated
  to Yetter\dash Drinfeld modules will never have pathological
  properties such as those demonstrated in Examples \ref{ex:nilpotent}
  and \ref{ex:pathological}.

\subsection{Free doubles $\D(V,\delta)$ and $\D(V,\Psi)$}

Let $V$ be a finite\dash dimensional Yetter\dash Drinfeld module
over $H$.
We denote by $\act$ the $H$\dash action on $V$, and by $v\mapsto
v^{(-1)}\tensor v^{(0)}\in H\tensor V$ the $H$\dash coaction on $V$.   
Recall that the free {braided} double associated to $V$ has
triangular decomposition  
$T(V^*)\lcprod H \rcprod  T(V)$, and the multiplication is defined by
the relations 
\begin{equation*}
    h\cdot v = (h_{(1)}\act v)\cdot h_{(2)}, \quad
    f\cdot h = h_{(1)} \cdot (f\ract h_{(2)}), \quad
        [f,v] = v^{(-1)} \langle f,v^{(0)} \rangle
\end{equation*}
for $v\in V$, $f\in V^*$, $h\in H$.

Note that these free {braided} doubles can be associated to any
braided space with biinvertible braiding; that is, the Hopf algebra $H$
does not have to be a part of the input. If $\Psi$ is a biinvertible
braiding on a finite\dash dimensional space $V$, then by~\ref{HPsi} $(V, \Psi)$  
is a Yetter\dash Drinfeld module over the  Hopf algebra
$H_\Psi$ which is a canonical minimal realisation of the braiding
$\Psi$. This yields a free braided double 
$\D(V,\Psi)\cong T(V)\tensor H_\Psi \tensor T(V^*)$ 
canonically associated to a braided space $(V,\Psi)$.

\subsection{Braided integers and braided derivatives}

The coaction $\delta\colon v \mapsto v^{(-1)}\tensor
v^{(0)}$ on $V$ has the property that ${v^{(-1)}}_{(1)}\tensor
{v^{(-1)}}_{(2)}\tensor v^{(0)} = {v^{(-1)}} \tensor
{v^{(0)}}^{(-1)}\tensor {v^{(0)}}^{(0)}$. 
Hence the formula for the quasibraided integer
$\widetilde{[n]}_\delta\colon V^{\tensorpow n} \to (H\tensor
V)^{\tensorpow n}$ from Definition \ref{def_qbfact} can be 
rewritten as
\begin{equation*}
    \widetilde{[n]}_\delta (v_1\tensor \dots \tensor v_n) = 
   \sum_{i=1}^n v_1\tensor \dots \tensor v_{i-1}\tensor
    v_i^{(-1)}\act (v_{i+1}\tensor \dots \tensor v_n) \tensor 
    {v_i^{(0)}}^{(-1)} \tensor {v_i^{(0)}}^{(0)}.
\end{equation*}
These operators can be expressed in terms of the
braiding $\Psi$ on $V$. We need the following 
\begin{definition}
Let $(V,\Psi)$ be a braided space. 
\emph{Braided integers} are operators
\begin{equation*}
   [n]_\Psi = \id_{V^{\tensorpow n}} + \Psi_{n-1,n} +
       \Psi_{n-1,n}\Psi_{n-2,n-1} + \dots + \Psi_{n-1,n}
       \Psi_{n-2,n-1} \dots  \Psi_{1,2} \quad \in \End(V^{\tensorpow n}).
\end{equation*}
We are using the leg notation, thus $\Psi_{i,i+1}$ stands for the
operator $\Psi$ applied at positions $i$, $i+1$ in the tensor
product. In particular, $[1]_\Psi=\id_V$ and
$[2]_\Psi=\id_{V^{\tensorpow 2}}+\Psi$. 

\emph{Braided factorials} are operators 
\begin{equation*}
 [n]!_\Psi = ([1]_\Psi \tensor \id^{\tensorpow n-1})
\circ ([2]_\Psi \tensor \id^{\tensorpow n-2}) \circ \dots
\circ [n]_\Psi \quad \in \End(V^{\tensorpow n}).
\end{equation*}
\end{definition}
Here is a new formula for
quasibraided integers and factorials on a Yetter\dash Drinfeld module:
\begin{lemma}
\label{lem:qb_braided}
Let $V$ be a Yetter\dash Drinfeld module over $H$, with coaction
$\delta$ inducing a braiding $\Psi$. Then
\begin{equation*}
         \widetilde{[n]}_\delta = (\id_V^{\tensorpow n-1}\tensor
         \delta)\circ [n]_\Psi, 
\qquad
        \widetilde{[n]}!_\delta = \delta^{\tensorpow n} \circ [n]!_\Psi.
\end{equation*}
Let $\pi\colon H\tensor V \to V$ be the projection $\pi=\epsilon
\tensor \id_V$, where $\epsilon$ is the counit of $H$. Then
\begin{equation*}
     [n]_\Psi = (\id_V^{\tensorpow n-1}\tensor
     \pi)\circ \widetilde{[n]}_\delta, 
\qquad
     [n]!_\Psi = \pi^{\tensorpow n}\circ \widetilde{[n]}!_\delta. 
\end{equation*}
\end{lemma}
\begin{proof}
The formula for $\widetilde{[n]}_\delta$ follows from  $\Psi(v\tensor
w)=v^{(-1)}\act w \tensor v^{(0)}$. 
For example, $v_1^{(-1)}\act (v_2\tensor \dots \tensor v_n) \tensor 
    {v_1^{(0)}}^{(-1)} \tensor {v_1^{(0)}}^{(0)}$ rewrites as 
$\delta_n\Psi_{n-1,n}\dots \Psi_{23}\Psi_{12}(v_1\tensor \dots \tensor
    v_n)$. 
The formula $\widetilde{[n]}!_\delta$ is then immediate from the
definition of the quasibraided factorial. 
The rest is an immediate consequence of
counitality of the coaction, $\pi\circ\delta = \id_V$. 
\end{proof}

Braided integers and braided factorials, which we have just obtained
as a particular case of their quasibraided analogues,
were introduced by Majid 
\cite{Mcalculus} (a book reference is \cite[10.4]{Mbook}).
When the braided space is $1$\dash dimensional, 
the braiding $\Psi$ is multiplication by constant $q\in \field$, 
and braided integers are the well\dash known $q$\dash integers 
$[n]_q=\frac{1-q^n}{1-q^{\phantom n}}$. 
Another important form of the braided factorial is

\subsection{The Woronowicz symmetriser}

For a permutation $\sigma$ in the symmetric group $\Symm_n$,  let
$\sigma =(i_1\, i_1+1)\dots
(i_l\, i_l+1)$ be a 
reduced (i.e., shortest) decomposition of $\sigma$ into elementary 
transpositions. For a braiding $\Psi$ on $V$, define $\Psi_\sigma$ to be equal to 
the operator $\Psi_{i_1, i_1+1}\dots \Psi_{i_l, i_l+1}$ on
$V^{\tensorpow n}$ (this does not depend
on the choice of a reduced decomposition of $\sigma$ because $\Psi$
satisfies the braid equation). 
The above expression for the braided factorial expands into 
\begin{equation*}
         [n]!_\Psi = \sum_{\sigma\in \Symm_n} \Psi_\sigma. 
\end{equation*}
This endomorphism of $V^{\tensorpow n}$, associated to a braiding
$\Psi$, is called the \emph{Woronowicz symmetriser of degree $n$}. It
was introduced by Woronowicz in \cite{Wo} (as an ``antisymmetriser''
with $-\Psi$ instead of $\Psi$), and its factorial expression was
given by Majid \cite[10.4]{Mbook}. 

We will consider an endomorphism of the whole tensor algebra
$T(V)$, given on tensor powers by the braided factorials:
\begin{equation*}
         \Wor(\Psi) \colon T(V) \to T(V), \quad 
         \Wor(\Psi)|_{V^{\tensorpow n}} := [n]!_\Psi.
\end{equation*}
Note that $[0]!_\Psi=1$ and $[1]!_\Psi=\id_V$. 
We refer to $\Wor(\Psi)$ simply as the \emph{Woronowicz symmetriser}. 

In the following definition, we identify $V^*\tensor V^*$ with the
dual space to $V\tensor V$ in a standard way via $\langle f\tensor
g,v\tensor w\rangle =\langle f,v\rangle \langle g,w\rangle$. This
allows us to view $\Psi^*$ as a braiding on $V^*$. 
\begin{definition}
\label{def:NW}
Let $(V,\Psi)$ be a braided space.  
The graded algebras 
\begin{equation*}
\mathcal B(V,\Psi)=T(V)/\ker\Wor(\Psi), \quad 
\mathcal B(V^*,\Psi^*)=T(V^*)/\ker\Wor(\Psi^*),
\end{equation*} 
are called the \emph{Nichols\dash Woronowicz algebras} associated to $(V,\Psi)$. 
\end{definition}

We are now ready to give a description of minimal doubles
specific to the case Yetter\dash Drinfeld modules. 

\begin{theorem}[Doubles of Nichols-Woronowicz algebras]
\label{th:Nichols}
Let $(V,\delta)$ be a Yetter\dash Drinfeld module for a Hopf algebra $H$,
and let $\Psi$ be the induced braiding on $V$. 
Then
\begin{equation*}
               U(V,\delta) = \mathcal B(V,\Psi), \quad
               U(V^*,\delta) = \mathcal B(V^*,\Psi^*)
\end{equation*}
are Nichols\dash Woronowicz algebras. 
The minimal double $\RHD_V:=\RD(V,\delta)$ has triangular decomposition
\begin{equation*}
               \RHD_V = \mathcal B(V,\Psi) \lcprod H \rcprod
               \mathcal B(V^*,\Psi^*).
\end{equation*}
\end{theorem}
\begin{proof}
By Theorem \ref{ker_qbfact}, $I(V,\delta)=\oplus_n \ker
\widetilde{[n]}!_\delta$ where $\delta$ is the coaction on $V$. But by
Lemma \ref{lem:qb_braided}, $\ker\widetilde{[n]}!_\delta = \ker
[n]_\Psi$. Hence $I(V,\delta)=\ker\Wor(\Psi)$ as required. The formula
$I(V^*,\delta)=\ker\Wor(\Psi^*)$ can be obtained in a similar way, using
Remark \ref{rem:righthanded}.
\end{proof}
We will call $\RHD_V$ the \emph{braided Heisenberg double}
associated to the Yetter\dash Drinfeld module $V$. 
If $H$ is the trivial Hopf algebra, $H=\field$, then $\RHD_V$ is the
Heisenberg\dash Weyl algebra $S(V)\tensor S(V^*)$. 

Similarly to the
free doubles, the Hopf algebra $H$ does not need to be in the 
picture: to any braided space $(V, \Psi)$, of finite dimension and
with biinvertible braiding, is associated the minimal braided
Heisenberg double $\RHD_{(V,\Psi)}$, defined as in the Proposition with
$H=H_\Psi$. Observe that the ideals $I(V,\delta)$, $I(V^*,\delta)$ depend only on
the braiding $\Psi$ on $V$, not on the Hopf algebra $H$; and that it
automatically follows that $\ker\Wor(\Psi)$ is a two\dash sided ideal
in $T(V)$.

An important feature of the braided Heisenberg double $\RHD_V$ is that its
Harish\dash Chandra pairing is non\dash degenerate. In fact, a
stronger property holds, and $\RHD_V$ satisfies the conditions in
Proposition~\ref{good_ideals}: 

\begin{lemma}
The scalar\dash valued pairing $\epsilon((\cdot,\cdot)_H)$ between
$\mathcal B(V^*,\Psi^*)$ and $\mathcal B(V,\Psi)$ in $\RHD_V$ is
non\dash degenerate.
\end{lemma}
\begin{proof}
Let $\phi\in V^{*\tensorpow n}$ and $b\in V^{\tensorpow n}$. 
One deduces from Proposition \ref{prop:HCformula} and Lemma
\ref{lem:qb_braided} that 
\begin{equation*}
        \epsilon((\phi,b)_H) = \langle \phi, [n]!_\Psi
        b \rangle. 
\end{equation*}
This is a non\dash degenerate pairing between $\mathcal B(V^*,\Psi^*)$
and  $\mathcal B(V,\Psi)$, because the kernels of the braided
factorials are quotiented out.
\end{proof}

The non\dash degenerate pairing between $\mathcal B(V^*,\Psi^*)$ and
$\mathcal B(V,\Psi)$ induces on each of these algebras a coassociative
coproduct. Via the associativity of multiplication in $\RHD_V$, or
otherwise, it can be shown that $\mathcal B(V^*,\Psi^*)$ and
$\mathcal B(V,\Psi)$ become dually paired \emph{braided Hopf
  algebras}. Braided Hopf algebras are a relatively
recent branch of the Hopf algebra theory and quantum algebra. We will
not give details here and refer the reader to \cite{Mcompanion}.
The braided coproduct on $v\in V\subset \mathcal B(V)$ is
$\underline\Delta v = v\tensor 1 + 1\tensor v$; the braided coproduct
is not multiplicative, but rather braided\dash multiplicative, and
this allows one to extend $\underline\Delta$ to the whole of $\mathcal
B(V,\Psi)$.

\begin{remark}
In fact, one can construct a braided Heisenberg double of any pair of
dually paired braided Hopf algebras. Such a construction should be
viewed as a quantum analogue of the algebra  $\mathcal H(A)$ referred 
to as the \emph{Heisenberg double} of $A$ \cite{STS,Lu}, where $A$ is
a finite\dash dimensional Hopf algebra.  
One has $\mathcal H(A) \cong A\tensor A^*$ with defining relation 
$\phi a = \langle \phi_{(1)}, a_{(2)} \rangle a_{(1)}
          \phi_{(2)}$ between $\phi\in A^*$ and $a\in A$.
The Heisenberg double produces canonical solutions to the \emph{pentagon equation} in the
same way as the Drinfeld double $D(A)$ works for the quantum Yang\dash Baxter
equation, see  \cite{BS} (based on earlier work by Woronowicz) and a
more algebraic exposition in \cite{Mi}. However, $\mathcal H(A)$ is
not a Hopf algebra; it is a simple (matrix) algebra \cite[9.4.3]{Mon}.
\end{remark}

\subsection{The Hopf algebra structure on $\mathcal B(V,\Psi)\lcprod H$ and
  on $H\rcprod \mathcal B(V^*,\Psi^*)$}

It follows from the theory of braided Hopf algebras that, while the
Nichols\dash Woronowicz algebras $\mathcal B(V,\Psi)$ and $\mathcal B(V^*,\Psi^*)$
are braided Hopf algebras, the algebras $\mathcal B(V)\lcprod H$ and
$H\rcprod \mathcal B(V^*)$ have the structure of ordinary Hopf
algebras. This structure is called \emph{biproduct bosonisation}, and
is due to Majid. We give the following proposition without proof;
it can be deduced from \cite{Menveloping}. 
\begin{proposition}
\label{prop:boson}
Let $V$ be a Yetter\dash Drinfeld module over a Hopf algebra $H$, with
braiding $\Psi$. The algebra  
$\mathcal B(V,\Psi)\lcprod H$ has the structure of an ordinary Hopf
algebra, which contains $\mathcal B(V,\Psi)$ as a subalgebra and $H$ as a
sub\dash Hopf algebra. 
Write a typical element of $\mathcal B(V,\Psi)\lcprod H$ as $\phi\cdot h$
where $\phi\in \mathcal B(V,\Psi)$ and $h\in H$. 
The coproduct on $v\cdot 1$, where $v\in V$, is defined by
\begin{equation*}
   \Delta (v\cdot 1) = (v\cdot 1)\tensor (1\cdot 1) + (1\cdot
   v^{(-1)})\tensor (v^{(0)}\cdot 1),
\end{equation*}
and extends to $\mathcal B(V,\Psi)\lcprod H$ by multiplicativity.
\qed   
\end{proposition}

Biproduct bosonisations, like the one given by the Proposition, are a
powerful tool in the structural theory of Hopf algebras. We will now
use a biproduct bosonisation to obtain the following

\begin{example}[Counterexample to the third Kaplansky's conjecture]
\label{ex:3rd_conj}
As we mentioned in Remark~\ref{rem:3rd_conj},
a counterexample to the third Kaplansky's conjecture is a
Hopf algebra over a field of characteristic $0$, which has a nonzero
central nilpotent element. 

It is easy to obtain a counterexample which is a braided Hopf algebra
(namely, a Nichols\dash Woronowicz algebra). Let $V$ be a finite\dash
dimensional vector space with braiding $\Psi(v\tensor
w)=-\tau(v\tensor w) = -w\tensor v$. The Woronowicz symmetriser
associated to $-\tau$ is the standard antisymmetriser: 
\begin{equation*}
[n]!_{-\tau} = \sum_{\sigma\in\Symm_n}(-1)^{\ell(\sigma)} \sigma
\quad\text{acting  on }V^{\tensorpow n}. 
\end{equation*}
It follows that the Nichols\dash Woronowicz
algebra of $(V,-\tau)$ is the exterior algebra $\wedge V$. 

Let $n=\dim V>0$. Consider the element $\omega$ spanning the top degree,
$\wedge^n V$, of the exterior algebra. Clearly, $\omega$ is a nonzero
nilpotent in $\wedge V$ as $\omega\wedge \omega=0$. 
If $\dim V$ is even, $\omega$ is central in $\wedge V$. 

To obtain an ordinary Hopf algebra, let us consider the biproduct
bosonisation of $\wedge V$ over the Hopf algebra $H_{-\tau}$, which is
the minimal Hopf algebra realising the braiding $-\tau$. According
to~\ref{HPsi}, $H_{-\tau}=\field \mathbb{Z}_2$. One can check that the
action of $\mathbb{Z}_2=\{1,s\}$ on $V$ is given by $s(v)=-v$, and the
coaction is $\delta(v)=s\tensor v$ for any $v\in V$.
If $n=\dim V$ is even, $s(\omega)=\omega$, hence $\omega$ commutes
with $s$ and is central nilpotent in the biproduct bosonisation
$\wedge V \lcprod \field\mathbb{Z}_2$ given by
Proposition~\ref{prop:boson}.  

The dimension of the Hopf algebra $\wedge V \lcprod
\field\mathbb{Z}_2$ is  $2^{n+1}$. The minimum is $8$ for $n=2$.
\end{example}

\begin{remark}
Note that in general, there is no canonical Hopf algebra structure on
$B\tensor H \tensor B'$ where $B$, $B'$ are dually paired braided Hopf
algebras. 
The difficulty in some of existing approaches to
``doubling'' a braided  Hopf algebra is that the double is expected to
be a Hopf algebra. It is in fact possible to ``double'' a
braided Hopf algebra in $\YD{H}$ if $H$ is a
self\dash dual Hopf algebra, such as $\field G$ where $G$ is a
finitely generated Abelian group. A principal example of such a double
which is a Hopf algebra
is the construction of the quantised universal enveloping algebra
$U_q(\mathfrak g)$ as a braided double of two Nichols\dash Woronowicz
algebras \cite{L, Menveloping}.
\end{remark}

\subsection{Mixed Yetter-Drinfeld structures and compatible braidings}

In the rest of this Section, we will consider minimal doubles
corresponding to  quasi\dash Yetter\dash Drinfeld modules $V$ of
the following special structure. 

Let $V$ be a finite\dash dimensional module over a Hopf algebra $H$,
with several coactions $\delta_1,\dots\delta_N$ of $H$ on $V$, 
each of them Yetter\dash Drinfeld compatible with the action. 
Let $t_1$, $t_2, \dots, t_N$ be scalar parameters. Put
\begin{equation*}
  \delta = t_1 \delta_1 + \dots + t_N \delta_N \colon V^*\tensor V \to
  H,
\end{equation*}
so that $\delta$ is a Yetter\dash Drinfeld quasicoaction on $V$
(in general, not a coaction).
We will refer to this as a \emph{mixed Yetter\dash Drinfeld
  structure}.
If $\Psi_k$ is the braiding on $V$ induced by the Yetter\dash Drinfeld
coaction $\delta_k$, the mixed Yetter\dash Drinfeld structure realises
the endomorphism 
\begin{equation*} 
    t_1 \Psi_1 + \dots + t_N \Psi_N \in \End(V\tensor V).
\end{equation*}
Note that the braidings $\Psi_k$ satisfy the compatibility
equation in Definition~\ref{def:compat}.

We will now give a description of the minimal {braided} double $\RD(V,\delta)$,
associated to a mixed Yetter\dash Drinfeld structure on $V$, in the
case when the coefficients $t_1, \dots, t_N$ are \emph{generic}.
This applies, for example, if $t_1, \dots, t_N$ are independent formal
parameters. If the $t_k$ are elements of $\field$, `generic' will mean
that $(t_1,\dots,t_N)$ are outside of a union of countably many
hyperplanes in $\field^N$; obviously, generic tuples are guaranteed to
exist only if $\field$ is uncountable. When we regard a braided
integer $[m]_\Psi$ as an operator on $V^{\tensorpow n}$, $n>m$, we
imply that it acts on the first $m$ tensor components in $V^{\tensorpow n}$.

\begin{proposition}[Minimal doubles for compatible braidings]
\label{prop:mixed}
For ge\-ne\-ric coefficients $t_1,\dots,${}$t_N$, the 
graded components $I_n=I_n(V,\delta)$ of the defining ideal $I(V,\delta)$ in
the minimal double $\RD(V,\delta)$ are given by
$I_0=I_1=0$, 
\begin{equation*}
   I_n = \bigcap \, \ker \bigl( 
 [2]_{\Psi_{k_2}} [3]_{\Psi_{k_3}} \dots [n]_{\Psi_{k_n}} \bigr),
\end{equation*}
for $n\ge 2$, where the intersection  on the right is over all sequences 
$\mathbf k=(k_2,\ldots,k_n)$ in $\{1,\dots$,$N\}^{n-1}$.
\end{proposition}
\begin{remark}
If the parameters $t_k$ are not generic, the ideals $I(V,\delta)$, $I(V^*,\delta)$
may only be bigger than in the generic case.
\end{remark}
\begin{proof}[Proof of Proposition \ref{prop:mixed}]
It is easy to see that the quasibraided integer
$\widetilde{[n]}_\delta$ is given by 
$\sum_k t_k (\id_V^{\tensorpow n-1}\tensor \delta_k)[n]_{\Psi_k}$. 
By Theorem \ref{ker_qbfact} and its proof, $I_n$ 
consists of all $\in V^{\tensorpow n}$ such that 
$\widetilde{[n]}_\delta b$ lies in $I_{n-1}\tensor H \tensor V$;
for generic $t_k$, this means $[n]_{\Psi_k}b \in I_{n-1}\tensor V$
for every $k$. Trivially, $I_0=0$ and $[1]_{\Psi_k}=\id_V$.
The Proposition now follows by induction.
\end{proof}

\subsection{Deformation of the Nichols-Woronowicz algebra}

We will now consider the example of a pair of compatible braidings, 
provided by Lemma~\ref{lem:cocomm}, 
and obtain an interesting 
(non-flat) deformation the Nichols\dash Woronowicz algebra. 

Let $(V,\Psi)$ be a finite\dash dimensional braided space with biinvertible
braiding. Assume that $\Psi$ is compatible with the trivial braiding
$\tau$ (this is a homogeneous quadratic constraint on $\Psi$). 
Then $(V, \Psi)$ can 
be realised as a Yetter\dash Drinfeld module over a cocommutative Hopf
algebra $H_\Psi$, see~\ref{cocomm}. Let $v\mapsto v^{(-1)} \tensor v^{(0)}$ be the
coaction of $H_\Psi$ on $V$. 

By Lemma~\ref{lem:cocomm}, there is a Yetter\dash Drinfeld
quasicoaction on $V$ given by $\delta_{\Psi,t\tau}(v)=v^{(-1)}
\tensor v^{(0)} + t \cdot 1\tensor v$. 
The free double $\D(V,\delta_{\Psi,t\tau})$ has commutation relation 
$[f,v]=v^{(-1)}\,\langle f,v^{(0)}\rangle + t\,\langle f,v\rangle\cdot 1$ between $f\in V^*$ and
$v\in V$. The introduction of the extra term $t\,\langle f,v\rangle\cdot 1$ makes the
maximal triangular ideal smaller:  

\begin{proposition}[Deformed braided Heisenberg double]
\label{prop:deformed}
For $t$ generic,
\begin{equation*}
      \RD(V,\delta_{\Psi,t\tau}) \cong 
{\mathcal B}_\tau(V,\Psi) \lcprod H_\Psi
      \lcprod {\mathcal B}_\tau(V^*,\Psi^*),
\end{equation*}
where the graded algebra ${\mathcal B}_\tau(V,\Psi)$ does not depend on $t$,
and its homogeneous components are
${\mathcal B}_\tau(V,\Psi)_n = V^{\tensorpow n} /  \ker
[n]!_{\Psi,\tau}$.  
The ``deformed braided factorial'' $[n]!_{\Psi,\tau}$ can be written as
\begin{equation*}
 [n]!_{\Psi,\tau} = ([2]_\Psi+u_2 [2]_\tau)\dots  ([n-1]_\Psi+u_{n-1}
 [n-1]_\tau) ([n]_\Psi+u_n [n]_\tau)
\end{equation*}
with independent formal parameters $u_k$. 
\end{proposition}
We will call ${\mathcal B}_\tau(V,\Psi)$ the \emph{deformed Nichols\dash
Woronowicz algebra}. Note that this construction works only for $\Psi$
satisfying the quadratic equation of compatibility with the trivial
braiding. 
Clearly there is a surjective map ${\mathcal B}_\tau(V,\Psi) \twoheadrightarrow \mathcal B(V,\Psi)$
onto the Nichols\dash Woronowicz algebra of $(V,\Psi)$. 
\begin{proof}[Proof of Proposition \ref{prop:deformed}]
The result follows immediately from Proposition \ref{prop:mixed},
because obviously 
\begin{equation*}
  \ker[n]!_{\Psi,\tau} = \bigcap\,
  [2]_{\Psi_{k_2}}\dots [n-1]_{\Psi_{k_{n-1}}}[n]_{\Psi_{k_n}},
\end{equation*}
where the intersection on the right is over all $k_i\in \{1,2\}$, 
$\Psi_1:=\Psi$ and $\Psi_2:=\tau$.  
\end{proof}
\subsection{Minimal quadratic doubles associated to compatible braidings}

We will finish the Section with ``quadratic versions'' of 
all {braided}
double constructions presented here. Recall the definition of the
minimal quadratic double $\D_{\mathit{quad}}(V,\delta)$ from Section~\ref{sect:brdoubles}. 

\begin{lemma}
\label{lem:quadcovers}
Let $V$ be a finite\dash dimensional vector space.
 
(1) If $\Psi$ is a biinvertible braiding on $V$,  the minimal
    quadratic double associated to $(V,\Psi)$ is 
    ${\mathcal B}_{\mathit{quad}}(V,\Psi) \tensor H_\Psi \tensor
    {\mathcal B}_{\mathit{quad}}(V^*,\Psi^*)$. 
Here
\begin{equation*}
{\mathcal B}_{\mathit{quad}}(V,\Psi) = T(V)/\lgen \ \ker
    (\id_{V^{\tensorpow 2}}+\Psi) \ \rgen
\end{equation*}
is the quadratic cover of the Nichols\dash Woronowicz algebra
$\mathcal B(V,\Psi)$. 

(2) If $\Psi_1, \dots, \Psi_N$ are braidings on $V$ arising from
    Yetter\dash Drinfeld coactions over the same Hopf algebra $H$, 
    the corresponding minimal quadratic double  
\begin{equation*}
T(V)/I_{\mathit{quad}}(V,\delta)\tensor H \tensor
    T(V^*)/I^*_{\mathit{quad}}(V,\delta)
\end{equation*}
has
\begin{equation*}
   I_{\mathit{quad}}(V,\delta) = \lgen \ \cap_{k=1}^N \ker(\id_{V^{\tensorpow
   2}}+\Psi_k) \ \rgen. 
\end{equation*}

(3) If $\Psi$ is a braiding on $V$ compatible with $\tau$, 
the above construction with $\Psi_1=\Psi$, $\Psi_2=\tau$
leads to the deformation 
\begin{equation*}
     {\mathcal B_{\mathit{quad}}}_\tau(V,\Psi) = T(V) / \lgen\  \ker(\id_{V^{\tensorpow
   2}}+\Psi) \cap \wedge^2 V \ \rgen
\end{equation*}
of the algebra ${\mathcal B}_{\mathit{quad}}(V,\Psi)$. 
Here $\wedge^2 V$ is the space of skew\dash symmetric tensors in
$V^{\tensorpow 2}$ which is the kernel of $\id_{V^{\tensorpow 2}}+\tau$.
\end{lemma}
\begin{proof}
All statements are obtained by leaving only quadratic relations in
$T(V)$, $T(V^*)$ in minimal doubles constructed in this Section.
\end{proof}


\section{The category of braided doubles over $H$. Perfect subquotients}
\label{sect:perfect}

\subsection{Morphisms between braided doubles over $H$}

We will now describe the category $\mathcal{D}_H$, whose objects are
braided doubles over a Hopf algebra $H$. 
As in \ref{hierarchy},
for a finite\dash dimensional
quasi\dash Yetter\dash Drinfeld 
module $(V,\delta)$ over $H$ denote  the set of $(V,\delta)$\dash braided doubles by
$\mathcal{D}(V,\delta)$. We have 
\begin{equation*}
     \mathrm{Ob}\ \mathcal{D}_H = \bigcup_{(V,\delta)\in \Ob\,\QYD{H}}  \mathcal{D}(V,\delta).
\end{equation*}

Before we define morphisms in $\mathcal{D}_H$, let us introduce a
small bit of notation. 
If $\mu \colon V\to W$ is a map of vector spaces, denote by 
$T(\mu)\colon T(V)\to T(W)$ the linear map which coincides with
$\mu^{\tensor n}$ on $V^{\tensorpow n}$. This is an algebra
homomorphism. 
If $I\subset T(V)$ and $J\subset T(W)$ are two\dash sided ideals such
that $\mu(I)\subseteq J$, there is an algebra homomorphism
$\overline{T(\mu)}\colon T(V)/I \to T(W)/J$.  
We say that this map is \emph{induced by $\mu$}. 

We define morphisms in 
$\mathcal{D}_H$ in a natural way: they must be triangular maps between
algebras with 
triangular decomposition over $H$ (see \ref{triang_simple}), and
should come from the maps between the generating quasi\dash YD modules. 
\begin{definition}
Let $(V,\delta_V)$, $(W,\delta_W)$ be two finite\dash dimensional quasi\dash Yetter\dash
Drinfeld modules over a Hopf algebra $H$, and let 
\begin{equation*}
 A = T(V)/I^-\lcprod H \rcprod T(V^*)/I^+, 
\qquad
B = T(W)/J^-\lcprod H \rcprod T(W^*)/J^+
\end{equation*}
be a $(V,\delta_V)$- and a $(W,\delta_W)$\dash braided double, respectively. 
Morphisms between $A$ and $B$ are algebra maps
\begin{equation*}
\overline{T(\mu)} \tensor \id_H \tensor \overline{T(\nu^*)}
\colon A \to B,
\end{equation*}
induced by a pair of $H$\dash module maps $\mu \colon V \to W$, $\nu \colon W
\to V$. 
\end{definition}
\begin{lemma} 
\label{lem:subquot}
Let $(V,\delta_V)$, $(W,\delta_W)$ be finite\dash dimensional quasi\dash YD
  modules over $H$. 

1. If a pair $V \xrightarrow{\mu} W \xrightarrow{\nu} V$ of $H$\dash module
   maps induces a morphism between a $(V,\delta_V)$\dash braided double and a
   $(W, \delta_W)$\dash braided double, then 
$\delta_V = (\id_H\tensor \nu)\circ \delta_W \circ \mu$. 

2. Any $H$\dash module maps $V \xrightarrow{\mu} W \xrightarrow{\nu}
   V$ satisfying 1.\ induce a morphism $\D(V,\delta_V)\to \D(W,\delta_W)$ between free
   braided doubles. 
\end{lemma}
\begin{proof}
  1. Take $v\in V$ and $f\in V^*$. The commutator $[f,v]$ in any
     $(V,\delta_V)$\dash braided double is equal to 
$v^{[-1]}\langle f,v^{[0]}\rangle$. As $\mu$, $\nu$ induce
     a morphism of braided doubles, the same commutator (an element of
     $H$) must be equal to 
\begin{equation*}
       \mu(v)^{[-1]}\, \langle \nu^*(f),\mu(v)^{[0]} \rangle_W
     = \mu(v)^{[-1]}\, \langle f,\nu(\mu(v)^{[0]}) \rangle_V. 
\end{equation*}
Here, $\mu(v)^{[-1]}\tensor \mu(v)^{[0]}$ is $\delta_W(\mu(v))$;
the pairing between $W^*$ and $W$ is denoted by $\langle
\cdot,\cdot\rangle_W$, the  pairing between  
$V^*$ and $V$ is denoted by $\langle
\cdot,\cdot\rangle_V$. 
Since $f\in V^*$ is arbitrary, it follows that 
$\delta_V = (\id_H\tensor \nu)\circ \delta_W \circ \mu$. 

2. Let us show that
\begin{equation*}
      j=T(\mu)\tensor \id_H \tensor T(\nu^*)\colon 
       \D(V,\delta_V) \to \D(W,\delta_W)
\end{equation*}
is a map of algebras. The semidirect product relations in $\D(V,\delta_V)$
are preserved by $j$ because $T(\mu)$, $T(\nu^*)$ are $H$\dash
equivariant maps. The commutator relation between $\nu^*(f)$ and
$\mu(v)$ is equivalent to $\delta_V = (\id_H\tensor \nu)\circ \delta_W
\circ \mu$, as was shown in the proof of part 1; and there are no other relations in 
$\D(V,\delta_V)$. 
\end{proof}
Thus, morphisms between free braided doubles are the same as pairs of
maps between quasi\dash YD modules, satisfying 
\begin{definition}(Subquotients)
Let $V$, $W$ be quasi\dash Yetter\dash Drinfeld modules over a Hopf
algebra $H$. We say that $V$ is a \emph{subquotient} of $W$
via the maps $V\xrightarrow{\mu} W \xrightarrow{\nu} V$, if 
$\mu$, $\nu$ are  $H$\dash module maps such that the 
quasicoaction $\delta_V$ is induced from $\delta_W$ via
$\delta_V = (\id_H\tensor \nu)\circ \delta_W
\circ \mu$. 
\end{definition}
We stress that none of the maps $\mu$, $\nu$ in this definition is
required to be injective or surjective. 

\begin{remark}
Let $\mathcal{D}^{\mathrm{free}}_H$ be the full subcategory of
$\mathcal{D}_H$  consisting of all free braided doubles over $H$. It
follows from Lemma \ref{lem:subquot} that
$\mathcal{D}^{\mathrm{free}}_H$ is equivalent to the following
category:
\begin{itemize}
\item[-- ] objects: finite\dash dimensional quasi\dash YD modules $V$ over $H$;
\item[-- ] morphisms between $V$ and $W$: 
    diagrams $V \xrightarrow{\mu} W \xrightarrow{\nu} V$
    which make $V$ a subquotient of $W$;
\item[-- ] composition: the composition of $V \xrightarrow{\mu} W
    \xrightarrow{\nu} V$ and $W \xrightarrow{\mu'} X
    \xrightarrow{\nu'} W$ is the diagram 
    $V \xrightarrow{\mu'\circ \mu} X
    \xrightarrow{\nu\circ\nu'} V$.  
\end{itemize}
\end{remark}

\subsection{Perfect subquotients}

Let $(V,\delta_V)$, $(W,\delta_W)$ be finite\dash dimensional 
quasi\dash YD modules over $H$. If $V$ is a
subquotient of $W$, we have a map $j\colon \D(V,\delta_V) \to \D(W,\delta_W)$
between free braided
doubles. Of course, we may consider a composite map 
\begin{equation*}
\overline{j}\colon \D(V,\delta_V) \xrightarrow{j} \D(W,\delta_W) \twoheadrightarrow \RD(W,\delta_W)
\end{equation*}
into the minimal double associated to $W$. Hence there is some
$(V,\delta)$\dash braided double $\D(V,\delta)/\ker \overline{j}$ which embeds
injectively in $\RD(W,\delta_W)$; but this may not be the minimal double
$\RD(V,\delta)$. This is the content of the following

\begin{proposition}
\label{prop:preimage}
If $V$ is a subquotient of $W$ via the maps 
$V \xrightarrow{\mu} W\xrightarrow{\nu} V$,
then $I(V,\delta_V)$ contains the preimage $T(\mu)^{-1} I(W,\delta_W)$.
\end{proposition}
\begin{proof}
For finite\dash dimensional $V$ and $W$ this is an immediate
consequence of the above: the kernel
$ T(\mu)^{-1} I(W,\delta_W) \cdot H T(V^*) + T(V) H \cdot T(\nu^*)^{-1}
I(W^*)$  of $\overline{j}$ 
is a triangular ideal in $\D(V,\delta)$, hence is contained in 
$I(V,\delta_V)HT(V^*)+T(V)H I(V^*,\delta_V)$. 

Here is an alternative proof which works for infinite\dash dimensional
$V$, $W$. Direct computation shows
\begin{equation*}
    (\id_H\tensor \nu)^{\tensorpow n} \circ
    \widetilde{[n]}!_{\delta_W}\circ 
    \mu^{\tensorpow n} = \widetilde{[n]}!_{\delta_V}. 
\end{equation*}
Therefore, $I(V,\delta_V)\cap V^{\tensorpow n}=\ker \widetilde{[n]}!_{\delta_V}$
is equal to $T(\mu)^{-1} \ker \bigl((\id_H\tensor \mu)^{\tensorpow
  n} \circ \widetilde{[n]}!_{\delta_W}\bigr)$. The latter contains 
 $T(\mu)^{-1} \ker\widetilde{[n]}!_{\delta_W} = ( T(\mu)^{-1} I(W,\delta_W) )
\cap V^{\tensorpow n}$.
\end{proof}
The Proposition immediately leads to the following
\begin{definition}
\label{def_perfect}
Let $(V,\delta_V)$, $(W,\delta_W)$ be quasi\dash YD modules
such that $V$ is a subquotient of $W$ 
via the maps $V \xrightarrow{\mu} W\xrightarrow{\nu} V$.
We say that $V$ is a \emph{perfect subquotient}
of $W$, if $I(V,\delta_V) = T(\mu)^{-1} I(W,\delta_W)$. 
\end{definition}
\begin{remark}
\label{rem:comp_perfect}
Observe that if $V$ is a perfect subquotient of $W$ via the maps
$V \xrightarrow{\mu} W\xrightarrow{\nu} V$, and $W$ is a
perfect subquotient of $X$ via the maps $W \xrightarrow{\mu'}
X\xrightarrow{\nu'} W$,  
then $V$ is a perfect subquotient of $X$ via the composition of these
two diagrams. 
\end{remark}

\subsection{Every quasi-YD module  is a perfect subquotient
 of a Yetter-Drinfeld module}

In the previous Section, we identified doubles of Nichols\dash
Woronowicz algebras as a distinguished class of minimal doubles. 
Now, given a quasi\dash Yetter\dash Drinfeld module $(V,\delta)$ over a Hopf
algebra $H$, we would like
to embed the algebra $U(V,\delta)$ (the ``lower part'' of the minimal
double $\RD(V,\delta)$) in  a Nichols\dash Woronowicz algebra of some
Yetter\dash Drinfeld module $Y$. 
This is achieved if $V$ is a perfect subquotient of $Y$.

We show in the next Theorem that for any quasi\dash YD module $V$,
there exists a Yetter\dash Drinfeld module $Y$ such that $V$ is a
perfect subquotient of $Y$. However, only in
some cases can we guarantee that $Y$  
can be chosen to be finite\dash dimensional. 

The Theorem will use the following Lemma:
\begin{lemma}
\label{lem:Y(V)}
Let $(V,\delta)$ be a quasi\dash Yetter\dash Drinfeld module over a Hopf
algebra $H$. Let $Y(V)$ be the space $H\tensor V$. 

$(a)$ $Y(V)$ is a module over $H$ with respect to the following
action:
\begin{equation*}
    h\act (x\tensor v)=h_{(1)}\, x\, Sh_{(3)}\tensor h_{(2)}\act v, 
\qquad x\in H,\ v\in V.
\end{equation*}

$(b)$ The Yetter\dash Drinfeld condition on $\delta$ is equivalent to
the map $\delta\colon V \to H\tensor V=Y(V)$ being a
morphism of $H$\dash modules.

$(c)$ The map 
\begin{equation*}
      Y(V) \to H\tensor Y(V), \qquad
       x\tensor v \mapsto x_{(1)}\tensor x_{(2)} \tensor v
\end{equation*}
is an $H$\dash  coaction on $Y(V)$, which makes 
$Y(V)$ a Yetter\dash Drinfeld module over $H$.
\end{lemma}
\begin{proof}
It is straightforward to check that the formula given in $(a)$ indeed
defines an action of $H$ on $Y(V)$. Part $(b)$ follows by rewriting
the definition of a Yetter\dash Drinfeld module over $H$ in an
equivalent form suitable for Hopf algebras:
\begin{equation*}
     \delta(h\act v) = h_{(1)} v^{(-1)} Sh_{(3)} \tensor h_{(2)}\act
     v^{(0)}. 
\end{equation*}
Part $(c)$ is also easy, and is left as an exercise to the reader.
\end{proof}

\begin{theorem}
\label{th:perfect}
1. Let $V$ be a quasi\dash Yetter\dash Drinfeld module over a Hopf
algebra $H$, with quasicoaction $\delta\colon V \to H\tensor V$. 
Then $V$ is a perfect subquotient of the Yetter\dash
Drinfeld module $Y(V)$, via the maps
$V\xrightarrow{\delta} Y(V) \xrightarrow{\epsilon\tensor \id_V} V$.

2. If $V$ is finite\dash dimensional and 
\begin{itemize}
\item
$\dim H<\infty$, or
\item
$H$ is commutative and cocommutative,
\end{itemize}
then there exists a finite\dash dimensional Yetter\dash Drinfeld
module $Y$ of which $V$ is a perfect subquotient.
\end{theorem}
\begin{proof}
The map $\mu = \delta\colon V \to Y(V)$ is an $H$\dash module map by 
Lemma \ref{lem:Y(V)}. It is easy to see 
that $\nu=\epsilon\tensor \id_V$ is also an $H$\dash module map. 
Let us check that $V\xrightarrow{\mu} Y(V) \xrightarrow{\nu} V$ is
indeed a subquotient: 
$\mu(v)^{(-1)}\tensor \nu(\mu(v)^{(0)})$ is equal to 
${v^{[-1]}}_{(1)}\tensor
\epsilon({v^{[-1]}}_{(2)}){v^{[0]}}=\delta(v)$ as required.

To show that $V$ is a perfect subquotient of $Y(V)$, denote the
braiding on $Y(V)$ by $\Psi$ and observe that 
\begin{equation*}
      \widetilde{[n]}!_\delta =  [n]_\Psi \circ \delta^{\tensorpow n},
\end{equation*}
where both sides are maps $V^{\tensorpow n} \to Y(V)^{\tensorpow
  n}$. This formula is straightforward to verify, and immediately 
implies that $\ker\widetilde{[n]}!_\delta = (\delta^{\tensorpow
  n})^{-1} \ker [n]_\Psi$, precisely as required by the definition
  of a perfect subquotient. 

Now assume that $V$ is finite\dash dimensional. 
If $\dim H<\infty$, we may take $Y$ to be the Yetter\dash Drinfeld
module $Y(V)$, because $\dim Y(V)<\infty$. 

If $H$ is commutative and cocommutative, but not necessarily of finite
dimension, it is enough
to choose a finite\dash dimensional 
Yetter\dash Drinfeld submodule in $Y(V)$
containing $\delta(V)$. 
Let $H'\subset H$ be a subspace, $\dim H'<\infty$, such that
$\delta(V)\subset H'\tensor V$. By the fundamental theorem on
coalgebras \cite[Corollary 2.2.2]{Sw},
$H'\subset C \subset H$ where $C$ is a finite\dash dimensional
subcoalgebra of $H$. 
Put $Y=C\tensor V \subset Y(V)$. 
Then $Y$ is a submodule of $Y(V)$, because, by commutativity and
cocommutativity of $H$, 
\begin{equation*}
     h\act (c\tensor v) = c\tensor (h\act v)
\end{equation*}
for any $c\tensor v\in C\tensor V$; clearly, $Y$ is a subcomodule of
$Y(V)$ because $C$ is a subcoalgebra of $H$. 
The quasi\dash YD module $V$ will be a perfect subquotient of $Y$ via the
maps $\mu$ and $\nu|_{C\tensor V}$. 
\end{proof}

\subsection{Subquotients in right quasi-Yetter-Drinfeld modules}
\label{lr}
It is also useful to consider right quasi\dash YD modules over $H$.
If $(V,\delta)$ is a (left)  quasi\dash YD module, 
then $V^*$ is naturally a right quasi\dash YD module over $H$, with  
right quasicoaction $\delta_r\colon V^* \to V^* \tensor H$ as defined
in \ref{rem:righthanded}.
There is a notion of \emph{subquotient} for right quasi\dash YD modules.
A left quasi\dash YD module $V$ is a subquotient of $W$ via the maps
$V\xrightarrow{\mu} W 
\xrightarrow{\nu} V$, if and only if the right quasi\dash YD module   
$V^*$ is a subquotient of $W^*$ via the maps $V^*\xrightarrow{\nu^*}
W^* \xrightarrow{\mu^*} V^*$. 

The notion of \emph{perfect subquotient}
is also defined for right modules. Theorem
\ref{th:perfect} clearly admits a version for right quasi\dash Yetter\dash
Drinfeld modules.
A word of warning: 
if $V\xrightarrow{\mu} W \xrightarrow{\nu} V$ is a perfect
subquotient, $V^*\xrightarrow{\nu^*}
W^* \xrightarrow{\mu^*} V^*$ is not necessarily a perfect subquotient. 

If $\mu$, $\nu$ is a pair of morphisms such that \emph{both} 
$V\xrightarrow{\mu} W
\xrightarrow{\nu} V$ and $V^*\xrightarrow{\nu^*}
W^* \xrightarrow{\mu^*} V^*$ are perfect subquotients, 
the minimal double $\RD(V,\delta)$ embeds as a subdouble in the minimal
double $\RD(W,\delta_W)$. 

Ideally, we would like to look for such an embedding of any minimal
double $\D(V,\delta)$ into a braided Heisenberg double, corresponding to
some Yetter\dash Drinfeld module $Y$.  
But in general, we have no tools to achieve this:
although the pair of morphisms $V\xrightarrow{\delta} Y(V)
\xrightarrow{\epsilon\tensor \id_V} V$, produced by Theorem
\ref{th:perfect} for finite\dash dimensional $H$, is a perfect
subquotient, the adjoint map $(\epsilon \tensor \id_V)^* 
\colon V^* \to V^*\tensor H^*$ does not typically give a perfect
subquotient. It is instructive to check this when $H$ is a group
algebra of a finite group. 

However, we manage to embed those rational Cherednik algebras, which are 
minimal doubles of a special kind over a group algebra, in braided
Heisenberg doubles.  This will be done in the next Section.

\section{Rational Cherednik algebras}
\label{sect:cherednik}

Up to now we have been dealing with common properties of braided
doubles attached to any finite\dash dimensional module over some
Hopf algebra $H$. In this Section, we will soon fix a $\field$\dash vector space
$V$, $\dim V<\infty$, and take the group algebra
$\field G$ of an \emph{irreducible linear group} $G\le \mathit{GL}(V)$ as $H$. 
Our task is to study braided doubles $\Uminus \lcprod
\field G \rcprod \Uplus$ where $\Uminus$ and $\Uplus$ are commutative algebras. 

\subsection{Quasi-Yetter-Drinfeld modules over a group algebra}

Initially, let $G$ be an arbitrary group. The coproduct, counit and
antipode in $H=\field G$ are defined on $g\in G$ by
\begin{equation*}
\Delta(g) = g\tensor g, 
\quad \epsilon(g) = 1,
\quad S(g)=g^{-1},
\end{equation*}
respectively.  Unlike for general Hopf algebras, we use traditional notation $(g,v)\mapsto
g(v)$ for an action of $G$ on a space $V$. 
The definition of a quasi\dash Yetter\dash
Drinfeld module is restated in the group algebra case as follows:
\begin{lemma}
\label{lem:group_qc}
A quasi\dash Yetter\dash Drinfeld module over a group $G$ is a
representation $V$ of $G$, equipped with a map 
\begin{equation*}
\delta\colon V \to \field G \tensor V, \qquad
\delta(v) = \sum_{h\in G} h\tensor L_h(v),
\end{equation*}
where the linear maps $L_h\in \End(V)$ satisfy
\begin{equation*}
       g(L_h(v)) = L_{ghg^{-1}}(g(v)), \qquad g,h\in G, \ v\in V. 
\end{equation*}
\end{lemma}
\begin{proof}
Note that $ g(L_h(v)) = L_{ghg^{-1}}(g(v))$ is just the $G$\dash
equivariance condition for $\delta \colon V \to Y(V)$,
where the $G$\dash action on $Y(V)=\field G \tensor V$
is given, as in Lemma \ref{lem:Y(V)}, by 
$g(x\tensor v)=gxg^{-1}\tensor g(v)$. 
\end{proof}
If $V$ is a $G$\dash module,
we will, as usual, write $(V,\delta)$ to denote a particular quasi\dash
YD module structure on $V$ given by a quasicoaction $\delta$. 
\begin{remark}
We allow the group $G$ to be infinite; however, in the present paper
we do not explore continuous versions of our constructions and treat
the fields and groups as discrete objects. 
Accordingly, whenever a summation over group elements is present, 
the sum should be well\dash defined; e.g., although all maps 
$\{L_h\in \End(V): h\in G\}$ 
may be non\dash zero, for any fixed $v\in V$ all but 
a finite number of $L_h(v)$  must be zero. 
\end{remark}

Let us now give a definition of a Yetter\dash Drinfeld module in a form
more suitable for group algebras. 

\subsection{Yetter-Drinfeld modules over groups}

First, observe that a coaction of the group algebra $\field G$ on a
vector space $Y$ is the same as a $G$\dash grading on $Y$.

Indeed, write the coaction as $\delta(y)=\sum_{h\in G}h\tensor
L_h(y)$ as in Lemma \ref{lem:group_qc}. 
The comultiplicativity axiom, $(\Delta\tensor
\id_Y)\delta=(\id_{\field G}\tensor \delta)\delta$, 
means that $L_g L_h$ equals $L_h$ if
$h=g$, or $0$ otherwise. The counitality axiom, $(\epsilon \tensor
\id_Y)\delta=\id_Y$, is equivalent to 
$\sum_{h\in G}L_h(y) = y$. Thus, $\{L_h: h\in G\}$
is a complete set of pairwise orthogonal idempotents on $Y$, and 
$Y=\oplus_{h\in G}Y_h$ where $Y_h=L_h Y$. 
Therefore, the definition of a Yetter\dash
Drinfeld module (cf.\ \ref{qydmodcomod}) looks in the
group algebra case as follows:

\begin{lemma}
\label{lem:YD_G}
A Yetter\dash Drinfeld module over a group $G$ is a vector space $Y$
such that 
\begin{enumerate}
\item
$G$ acts on $Y$: $(g,y)\in G\times Y \mapsto g(y)\in Y$;
\item
$Y$ is a $G$\dash graded space: $Y=\mathop{\oplus}\limits_{h\in G}Y_h$;
\item
the grading is compatible with the action:
$g(Y_h) \subseteq Y_{ghg^{-1}}$, $g,h\in G$.
\end{enumerate} 
The Yetter\dash Drinfeld structure induces a braiding on $Y$ by the
formula
\begin{equation*}
     \Psi(y\tensor z)= h(z) \tensor y, \qquad
     y\in Y_h, \ z\in Y. 
\qed
\end{equation*}
\end{lemma}

\subsection{Quasi-Yetter-Drinfeld modules $(V,\delta)$ with 
commutative $U(V,\delta)$}
\begin{lemma}
\label{lem:comm_eq}
Let $V$ be a finite\dash dimensional quasi\dash YD module over
$\field G$, with quasicoaction $\delta(v)=\sum_{h\in G}h\tensor
L_h(v)$. The algebra $U(V,\delta)$ is commutative, if and only if 
\begin{equation*}
      \delta(v-h(v)) \tensor L_h(w) = 
      \delta(w-h(w)) \tensor L_h(v)  
\quad\text{for any }h\in G, \ v,w\in V.
\end{equation*}
\end{lemma}
\begin{proof}
By Theorem \ref{th:perfect}, $(V,\delta)$ is a perfect subquotient of the
Yetter\dash Drinfeld module $Y(V)$ with underlying vector space
$\field G \tensor V$. 
The map $\delta\colon V \to Y(V)$ induces an embedding 
\begin{equation*}
U(V,\delta) \hookrightarrow \mathcal B(Y(V))
\end{equation*}
of algebras. Because the algebra $U(V,\delta)$ is generated by elements of $V$,  
it is commutative, if and only if for
any $v,w\in V$ the elements $\delta(v)$ and $\delta(w)$ commute in the
Nichols\dash Woronowicz algebra $\mathcal B(Y(V))$. 
By Definition \ref{def:NW} of the Nichols\dash
Woronowicz algebra,  this is equivalent to
\begin{equation*}
 (\id+\Psi)(\delta(v)\tensor \delta(w)-\delta(w)\tensor \delta(v)) = 0
\quad\text{for any }v,w\in V,
\end{equation*}
where $\Psi$ is the braiding on $Y(V)$. 
This is because the quadratic relations in $\mathcal B(Y(V))$
are the kernel of $[2]!_\Psi=\id+\Psi$. 
In other words, $\delta(V)$ must be an \emph{Abelian subspace} of the
braided space $(Y(V),\Psi)$ --- the term is from \cite[I.C]{AF}. 
The condition rewrites as
\begin{equation*}
\delta(v)\tensor \delta(w) - \Psi(\delta(w)\tensor \delta(v)) = 
\delta(w)\tensor \delta(v) - \Psi(\delta(v)\tensor \delta(w)). 
\end{equation*}
Note that the $G$\dash coaction on $Y(V)\cong H\tensor V$ is given
by $h\tensor v \mapsto h\tensor h\tensor v$. Substituting
$\delta(\cdot)=\sum_{h\in G}h\tensor L_h(\cdot)$ and
using the formula for $\Psi$ from Lemma \ref{lem:YD_G}, rewrite the
commutativity equation as
\begin{equation*}
     \delta(v)\tensor \delta(w) - \sum_h h(\delta(v))\tensor h \tensor
     L_h(w)
= 
     \delta(w)\tensor \delta(v) - \sum_h h(\delta(w))\tensor h \tensor
     L_h(v).
\end{equation*}
The left\dash hand side is $\sum_h (1-h)\delta(v)\tensor h \tensor
L_h(w)$, and this expression must be symmetric in $v$ and $w$. 
Equivalently, $(1-h)\delta(v)\tensor L_h(w)=(1-h)\delta(w)\tensor
L_h(v)$ for any $h\in G$. We may interchange the action of $1-h$ and
$\delta$ because $\delta$ is $G$\dash equivariant. 
\end{proof}

\subsection{The commutativity equation for an irreducible linear group $G$}

From now on we take $G$ to be an irreducible linear group, that is,
$G\le \mathit{GL}(V)$ such that $V$ is an irreducible $G$\dash
module. Elements of the set 
\begin{equation*}
   \PR = \{ s\in G : \dim(1-s)V = 1 \}
\end{equation*}
are called \emph{complex reflections} in $G$. Note that we do not
restrict the characteristic of the ground field $\field$; an
alternative term, more commonly used for linear groups in positive
characteristic, is \emph{pseudoreflection}.  

We will use the following easy observation about complex reflections. 
By $\langle\cdot,\cdot\rangle$ is denoted the pairing between $V^*$ and $V$.
\begin{lemma}
\label{lem:roots}
Let $s\in\PR$. 1. There are non\dash zero vectors $\alpha_s\in V^*$, 
$\check\alpha_s\in V$ such that 
\begin{equation*}
   s(v) = v - \langle \alpha_s,v\rangle \check\alpha_s
\quad\text{for }v\in V.
\end{equation*} 
The vectors $\alpha_s$, $\check\alpha_s$ are defined up to a
simultaneous rescaling which leaves $\check\alpha_s \tensor \alpha_s$ fixed. 

2. For any $g\in G$, the element $gsg^{-1}$ is also a complex
reflection, and 
\begin{equation*}
     \check\alpha_{gsg^{-1}}\tensor\alpha_{gsg^{-1}} =
      g(\check\alpha_s) \tensor g(\alpha_s).
\qed
\end{equation*}
\end{lemma}
We will refer to $\alpha_s$ (resp.\ $\check\alpha_s$) as the \emph{root}
(resp.\ the \emph{coroot}) of a complex reflection~$s$. 
Note that if $V\tensor V^*$ is identified with the algebra $\End(V)$, 
the tensor $\check\alpha_s \tensor \alpha_s$ is \emph{equal}
to the endomorphism $1-s$ of $V$.  
\begin{proposition}
\label{prop:c_refl}
Let $\delta\colon V \to \field G \tensor V$ be a
Yetter\dash Drinfeld quasicoaction on $V$.
The algebra $U(V,\delta)$ is commutative, 
if and only if $\delta$ is of the form 
\begin{equation*}
      \delta(v) = t\cdot 1\tensor v + \sum_{s\in\PR}
          s\tensor \langle \alpha_s, v\rangle b_s,
\end{equation*}
for some constant $t\in \field$ and vectors $b_s\in V$.  
\end{proposition}
\begin{proof}
As $\delta\colon V \to Y(V)$ is a map of $G$\dash modules (Lemma
\ref{lem:Y(V)}), $\ker\delta$ is a $G$\dash submodule of $V$.
By irreducibility of $V$, the quasicoaction $\delta$ is either zero or
injective.  In the trivial case $\delta=0$ one has $U(V,0)=\field$,
and we may put $t=0$, $b_s=0$ for all $s\in \PR$. 

We now assume that $\delta$ is injective. Then $\delta()$ may be dropped from
the commutativity equation in Lemma \ref{lem:comm_eq}. 
If $\delta(v)=\sum_{h\in G}h\tensor L_h(v)$ where $L_h\in \End(V)$, 
the commutativity equation for fixed $s\in G$ is now as follows:
\begin{equation*}
  (v-s(v)) \tensor L_s(w) = (w-s(w)) \tensor L_s(v).
\end{equation*} 
It is easy to see that this tensor equation can hold for arbitrary $v,w\in V$
only if one of the following holds:
\begin{equation*}
(a)\ s=1; \quad\text{or}\quad
(b)\ L_s=0; \quad\text{or}\quad
(c)\ \dim(1-s)V=\dim L_s(V)=1.
\end{equation*}
Condition $(a)$ and $(b)$ are sufficient for the commutativity
equation to hold. With regard to $(a)$, 
it follows from  Lemma \ref{lem:group_qc} that
the map $L_1 \in \End(V)$ satisfies $L_1(g(v))=g(L_1(v))$. 
Since $V$ is irreducible, by Schur's lemma $L_1=t\cdot \id_V$
for some $t\in \field$.

In $(c)$, the element $s\in G$ must be a complex reflection,
and the commutativity equation rewrites as
\begin{equation*}
  \langle\alpha_s, v\rangle \check\alpha_s \tensor L_s(w) 
  = \langle\alpha_s, w\rangle \check\alpha_s   \tensor L_s(v).
\end{equation*} 
This holds, if and only if $L_s(v)=\langle\alpha_s, v\rangle b_s$ for
some vector $b_s\in V$.
\end{proof}
If $\langle \alpha_s,\check\alpha_s\rangle\ne 0$, 
one can show, using the $G$\dash equivariance of $\delta$, 
that the vectors $b_s$ must be proportional to $\check\alpha_s$. 
But if $\mathit{char}\ \field>0$, it may happen that $\langle
\alpha_s,\check\alpha_s\rangle= 0$ for a pseudoreflection $s$. 
Nevertheless, the next Theorem will show that any possible pathological
solutions are eliminated if the algebra $U(V^*,\delta)$ is also assumed to be
commutative.

\subsection{Rational Cherednik algebras}

We denote by $\field(\PR)^G$ the space of $\field$\dash
valued functions $c$ on the set $\PR$ of complex reflections, 
$s\mapsto c_s$, such that 
$c_{gsg^{-1}}=c_s$ for any $g\in G$. 

\begin{theorem}
\label{th:comm}
Let $G\le \mathit{GL}(V)$ be an irreducible linear group and $\PR$ be
the set of all complex reflections in $G$. 

1. There exists a $(V,\delta)$\dash braided double 
$\Uminus \lcprod \field G \rcprod \Uplus$ with 
commutative algebras $\Uminus$ and $\Uplus$, 
if and only if the quasicoaction $\delta$ is
\begin{equation*}
  \delta (v) = \delta_{t,c} (v) := 
  1\tensor tv + \sum\nolimits_{s\in\PR} 
c_s s\tensor (v-s(v)),
\end{equation*}
where $t\in \field$ and $c\in\field(\PR)^G$. 

2. For any $t$ and $c$ as above, there exists a
   $(V,\delta_{t,c})$\dash 
   braided double 
   $H_{t,c}(G)$ of the form $S(V)\lcprod \field G \rcprod
   S(V^*)$. 
\end{theorem}
The Theorem leads to the following
\begin{definition}
For $(t,c)\in \field\times\field(\PR)^G$, the 
braided double $H_{t,c}(G)$, given by the Theorem, 
is called a \emph{rational Cherednik
  algebra} of $G$. 
\end{definition}
The defining relations in $H_{t,c}(G)$ thus are:
\begin{align*}
   &[v,v']=0, \quad [f,f']=0, \quad gv=g(v)g, \quad gf=g(f)g, 
\\
   &[f,v] = t\langle f,v\rangle\cdot 1 + 
   \sum_{s\in\PR} c_s \langle f, (1-s)v \rangle \cdot s 
    \quad \in\field G 
\end{align*}
for $v,v'\in V$, $g\in G$, $f,f'\in V^*$. 
\begin{remark}
(a) Rational Cherednik algebras were introduced by Etingof and Ginzburg 
   in \cite{EG} as symplectic reflection algebras for the symplectic
   space $V\oplus V^*$. It is proved in \cite{EG} (for finite $G$ and
   in characteristic $0$) that $H_{t,c}(G)$ are the only braided doubles of the form
   $S(V)\lcprod \field G \rcprod S(V^*)$. We obtain the same result 
   (for any $|G|$ and $\mathit{char}\ \field$) using
   a different approach via braided doubles and Nichols\dash
   Woronowicz algebras. With our construction, 
   we get ``for free'' an embedding of $H_{t,c}(G)$ in a braided
   Heisenberg double, see \ref{embed_bhd}.  

(b) It is clear that essentially, rational Cherednik algebras are
defined over groups generated by complex reflections
(pseudoreflections). There is classification of such groups, both in
charactersistic zero and in characteristic $p$ case.  
\end{remark}
Before proving the Theorem, let us give an example of rational
Cherednik algebras over an infinite group. 
\begin{example}[$U_q(\mathit{sl}_2)$ and quantum Smith algebras]
Let $V=\field x$ be a one\dash dimensional space.
Let $q\in\field$ be not a root of unity, and denote by $G_q$ the
infinite cyclic subgroup of $\mathit{GL}(\field x)$ generated by $q$. 
Write the group algebra $\field G_q$ as $\field \{z^n:
n\in \mathbb{Z}\}$; it acts on $\field x$  via $z(x)=qx$.
Let $\field y$ be the dual space to $\field x$. 

All $z^n$, $n\ne 0$, trivially are complex reflections. 
To any sequence $(c_n)_{n\in\mathbb{Z}}$, where all but a finite
number of entries are zero, there is associated a 
rational Cherednik algebra
$H_{c}(G_q)$ with relations 
\begin{equation*}
     zx = qxz, \quad zy = q^{-1}yz, \quad
     [y,x]=\sum\nolimits_{n\in \mathbb{Z}} c_n z^n
\end{equation*}
(note that the role of the parameter $t$ is played here by $c_0$). 
These may be viewed as ``quantum Smith algebras'' (recall Smith
algebras from Introduction, \ref{bdintro}). 
A particular case when $[y,x]=z-z^{-1}$ gives the quantised universal
enveloping algebra $U_q(\mathit{sl}_2)$.
 
A version of $H_c(G_q)$ over $\field=\mathbb C$ where 
the commutator $[y,x]$ is an infinite power
series in $z$ 
might be interesting from the viewpoint of physical applications.  
\end{example}

\begin{proof}[Proof of Theorem \ref{th:comm}]
1. Call a Yetter\dash Drinfeld quasicoaction $\delta$ on $V$
   ``left\dash good'' (resp.\ ``right\dash good'') if there exists a
   $(V,\delta)$\dash braided double $\Uminus \lcprod \field
   G \rcprod \Uplus$ 
   with commutative $\Uminus$ (resp.\ with commutative $\Uplus$). 
By Proposition \ref{prop:c_refl}, any left\dash good $\delta$ has the
form  
$\delta(v)=1\tensor tv + \sum_{s\in\PR} 
s\tensor \langle \alpha_s, v\rangle b_s$ for some vectors $b_s\in V$. 
The corresponding right\dash hand  quasicoaction $\delta_r\colon V^* \to
   V^*\tensor \field G$ on the dual $G$\dash module $V^*$ is given by 
$\delta_r(f) = tf\tensor 1 + \sum_s \langle f,b_s\rangle
   \alpha_s\tensor s$. 
Note that a $s\in G$ is a complex reflection on $V$, if and only if
   $s$ is a complex reflection on $V^*$; 
furthermore, the complex reflection $s|_{V^*}$ has
$\check\alpha_s$ as the root and $\alpha_s$ as the coroot.
 
We may now apply a 
straightforward analogue of Proposition \ref{prop:c_refl} for
the dual module $V^*$ and the quasicoaction $\delta_r$. 
It follows that a left\dash good $\delta$ is also right\dash
good, if and only if for any complex reflection $s\in G$ 
the vector $b_s$ is proportional to 
the root $\check\alpha_s$ of $s|_{V^*}$. That is, $b_s=c_s
\check\alpha_s$ where $c_s\in \field$ are some constants, so that 
$\langle\alpha_s,v\rangle b_s=c_s(v-s(v))$. 

We have shown that a quasicoaction $\delta$ is left\dash good and right\dash
good, if and only if $\delta = \delta_{t,c}$ where 
$t$ is a constant and $c$ is a scalar function on $\PR$.
It remains to check what maps 
$\delta_{t,c}$ are $G$\dash equivariant, which by Lemma
\ref{lem:group_qc} means $g\circ \delta_{t,c} \circ g^{-1}=\delta_{t,c}$. 
The left\dash hand side equals
$1\tensor tv+\sum_s c_s \cdot gsg^{-1} \tensor (v-gsg^{-1}(v))$. 
Hence $\delta_{t,c}$ is a Yetter\dash Drinfeld quasicoaction on
$V$, if and only if $c_{gsg^{-1}}=c_s$ for
all $g\in G$, $s\in \PR$. 

2. Denote by $\wedge^2 V$ the subspace of $V^{\tensorpow 2}$ spanned by 
$v\tensor w - w\tensor v$ for all $v,w\in V$. 
Let 
\begin{equation*}
   I^- = \lgen \wedge^2 V   \rgen \subset T(V),
\qquad
   I^+ = \lgen \wedge^2 V^* \rgen \subset T(V^*)
\end{equation*}
be the ideals of definition of the symmetric algebras $S(V)$ and
$S(V^*)$, respectively. We have to show that 
$I^-\tensor \field G \tensor T(V^*)$ and $T(V)\tensor \field G \tensor
I^+$ are triangular ideals in $\D((V,\delta_{t,c}),\field G)$. 

By part 1, the algebra $T(V)/I(V,\delta_{t,c})$ is commutative, therefore
$\wedge^2 V\subset I(V,\delta_{t,c})$. By Theorem \ref{ker_qbfact}, this is
equivalent to $\widetilde{[2]}!_{\delta_{t,c}} (\wedge^2 V)=0$. 
By definition of the quasibraided factorial,
$\widetilde{[2]}!_{\delta_{t,c}} = (\delta_{t,c}\tensor \id_{\field G 
  \tensor V}) \widetilde{[2]}_{\delta_{t,c}}$. By irreducibility of $V$, 
the quasicoaction $\delta_{t,c}$ is injective 
(unless we are in the trivial case $t=0$, $c_s=0$ for all $s$).  
Hence $\widetilde{[2]}_{\delta_{t,c}} (\wedge^2 V)=0$, so by Corollary
\ref{cor:nonminimal} the ideal $I^-\tensor \field G \tensor T(V^*)$
is triangular. The argument for $I^+$ is analogous.   
\end{proof}

\subsection{Minimality of $H_{t,c}(G)$ for $t\ne 0$}

We would like to know if the rational Cherednik
algebra  $H_{t,c}(G)\cong S(V)\lcprod \field G \rcprod S(V^*)$ 
is a minimal double (i.e., has no proper
quotient doubles). To investigate this, we are going to use the minimality
criterion from Corollary~\ref{cor:criterion}. It is not difficult to
deduce the following expression for the commutator of $\phi\in S(V^*)$ and $v\in V$ in
$H_{t,c}(G)$:
\begin{equation*}
     [\phi,v] = t\cdot 1\tensor \frac{\partial \phi}{\partial v} + \sum\nolimits_{s\in \PR}
             c_s \langle \alpha_s , v \rangle \cdot s \tensor 
             \frac{\phi-s(\phi)}{\alpha_s}. 
\end{equation*}
This looks very similar to the celebrated \emph{Dunkl operator} acting
on $\phi$ (see e.g.\ \cite{DO} for the complex reflection group case). 
However, note that the right\dash hand side lies in $\field G \tensor
S(V^*)$. Here, $\frac{\partial \phi}{\partial v}$ stands for the
derivative of the polynomial $\phi$ along the vector $v$, and 
$\frac{\phi-s(\phi)}{\alpha_s}$ is the divided difference operator,
sometimes called the BGG-Demazure operator.

We first determine minimality of $H_{t,c}(G)$ for $t\ne 0$. 
\begin{proposition} Let  $t\ne 0$. If the characteristic of $\field$ is zero, 
$H_{t,c}(G)$ is a minimal double. If $\mathrm{char}\ \field>0$,
$H_{t,c}(G)$ is not a minimal double.
\end{proposition}
\begin{proof}
Let $\mathit{char}\ \field=0$. As there is no non\dash constant 
polynomial $\phi\in S(V^*)$ such that $\frac{\partial \phi}{\partial
v}=0$ for any $v\in V$, the commutator $[\phi,v]$ cannot be zero for
all $v\in V$. The same reasoning applies to commutators $[f,b]$ where
$f\in V^*$, $b\in S(V)$. Thus, by Corollary~\ref{cor:criterion},
$H_{t,c}(G)$ is a minimal double. 

If $\mathit{char}\ \field=p>0$, take $\phi$ to be a $G$\dash invariant
in $V^{*\tensorpow pr}$ for some $r$. Then both the differential and the
difference parts of $[\phi,v]$ vanish for all $v$, so by Corollary~\ref{cor:criterion}
$H_{t,c}(G)$ is not a minimal double.
\end{proof}

\begin{remark}
In positive characteristic, one may consider standard modules
$\{\overline{M}_\rho : \rho \in \mathrm{Irr}(G) \}$ for
$H_{t,c}(G)$ as was suggested in \ref{standmod}. One has
$\overline{M}_\rho \cong U(V,\delta_{t,c})\tensor \rho$, where
$U(V,\delta_{t,c})$ is a proper quotient of the polynomial algebra
$S(V)$. One should study the dimension of these standard
$H_{t,c}(G)$\dash modules (it may be finite) and their
reducibility. In a particular rank $1$ case this was done by Latour
\cite{La}. 
\end{remark}

\subsection{Restricted Cherednik algebras}

The algebra $H_{0,c}(G)$ is never a minimal double. In fact, a more
appropriate object from the point of view of minimality is a
finite\dash dimensional quotient double of $H_{0,c}(G)$ called
\emph{restricted Cherednik algebra}, the definition of which we now
recall. We consider  only the most familiar case, where $\field = \mathbb
 C$ and $G$ is a complex reflection group.

Let $S(V)^{G}_+\subset S(V)$ be the set of $G$\dash
invariant polynomials with zero constant term. The algebra
$S(V)_G=S(V)/\lgen S(V)^{G}_+ \rgen$ is termed \emph{the coinvariant
  algebra} of $G$. The algebra
\begin{equation*}
     \overline{H}_{0,c}(G)= H_{0,c}(G)/\lgen S(V)^{G}_+, S(V^*)^{G}_+
     \rgen \cong S(V)_G \lcprod \mathbb{C}G \rcprod S(V^*)_G
\end{equation*}
is the restricted Cherednik algebra of $G$. 
``Baby Verma modules'' for $\overline{H}_{0,c}(G)$ are important for
the representation theory of the rational Cherednik algebra
$H_{0,c}(G)$ at $t=0$; they were introduced and studied by Gordon
\cite{Go}.

\begin{proposition}
\label{prop:min_restr}
Let $\field=\mathbb C$ and $G\le \mathit{GL}(V)$ be a complex
reflection group. If $c_s\ne 0$ for all $s\in \PR$, the algebra
$\overline{H}_{0,c}(G)$ is a minimal double.
\end{proposition}
\begin{proof}
For a non\dash constant $\phi\in S(V^*)_G$,  expressions
$\phi-s(\phi)$  cannot  vanish simultaneously for all $s\in \PR$
(otherwise, as $\PR$ generates $G$, 
$\phi$ would be a nontrivial $G$\dash invariant in
$S(V^*)_G$). It follows that $[\phi,v]$ cannot be zero for all $v\in
V$. Likewise for $S(V)_G$. The minimality follows by
Corollary~\ref{cor:criterion}. 
\end{proof}
\begin{remark}
Note that the assumtpion $c_s\ne0$ $\forall s\in\PR$, made in the
Proposition, can be slightly
relaxed and replaced with the following:
complex reflections $s$, for which $c_s\ne 0$, generate $G$. 
\end{remark}

\subsection{The Yetter-Drinfeld module $\Ycr(G)$}

For an irreducible linear group $G \le \mathit{GL}(V)$ 
we obtained a 
complete classification of braided doubles of the form $S(V)\lcprod
\field G \rcprod S(V)$ (Theorem \ref{th:comm}). 
This  was achieved by observing that $V$ must identify,
via the quasicoaction $\delta\colon V \to Y(V)$, 
with an Abelian subspace in the Yetter\dash Drinfeld module $Y(V)$. 
The quasi\dash Yetter\dash Drinfeld module $(V,\delta)$ then becomes a
perfect subquotient of $Y(V)$. 

But in fact, the Abelian subspace $\delta(V)\subset Y(V)$ will always
lie in a proper submodule $\Ycr(G)$ of $Y(V)$. We will now give an
abstract description of the Yetter\dash Drinfeld module  
$\Ycr(G)$.Let us start with ots ``essential part'', denoted by
$\RYcr(G)$. 

\begin{definition}
\label{def:RYcr}
Define $\RYcr(G)$ to be the following  Yetter\dash Drinfeld module over $G$:
\begin{itemize}
\item
$\RYcr(G)$ is a vector space spanned by 
symbols $[s]$, indexed by $s\in \PR$;
\item $G$\dash action on $\RYcr(G)$: for $g\in G$, 
\begin{equation*}
g([s])=\lambda(g,s)[gsg^{-1}],
\end{equation*}
where $\lambda(g,s)\in\field$ is such that $g(\check\alpha_s) = \lambda(g,s)\check\alpha_{gsg^{-1}}$;
\item
$G$\dash grading:
$\RYcr(G) =
  \mathop{\oplus}\limits_{s\in\PR}\RYcr(G)_s$, where 
$ \RYcr(G)_s = \field \cdot [s]$.
\end{itemize}
\end{definition}
\begin{remark}
The function $\lambda\colon G \times \PR\to
\field^\times$ is a $1$\dash cocycle, in the sense 
that $\lambda(gh,s)=\lambda(g,hsh^{-1})\lambda(h,s)$. 
Each of the coroots $\check\alpha_s\in V$ 
is defined only up to  a scalar factor; the $G$\dash module structure
on $\RYcr(G)$ does not depend on a choice of such factors, which
changes $\lambda(g,s)$ by a coboundary.
\end{remark}
\begin{definition}
The Yetter\dash Drinfeld module $\Ycr(G)$ is 
defined as $\Ycr(G)=V_1 \oplus\RYcr(G)$. 
Here $V_1$ is a copy of the module $V$ with the trivial
Yetter\dash Drinfeld structure, given by the coaction $v\mapsto
1\tensor v$. 
\end{definition}
In the next Lemma, a typical element of $\Ycr(G)$ is written as 
$v\oplus \oplus_s a_s [s]$ 
for some $v\in V$ and $a_s\in \field$.
\begin{lemma}
\label{lem:perf}
Let $\delta_{t,c}$ be the quasicoaction on $V$ introduced in
Theorem~\ref{th:comm}. 
The quasi\dash Yetter\dash Drinfeld module $(V,\delta_{t,c})$ is
a perfect subquotient of $\Ycr(G)$ via the maps 
\begin{align*}
&\mu\colon V \to\Ycr(G), \quad \mu(v)=t v \oplus
\mathop{\oplus}\limits_{s\in\PR} c_s \langle \alpha_s,
v\rangle [s] 
\quad\text{and}
\\
&\nu\colon \Ycr(G) \to V,
\quad \nu|_{V_1}=\id_V, 
\quad
\nu([s])=\check\alpha_s.
\end{align*}
\end{lemma}
\begin{proof}
By Theorem \ref{th:perfect},
$(V,\delta)$ is a perfect subquotient of the Yetter\dash Drinfeld
module $Y(V)\cong \field G \tensor V $ via the maps $V
\xrightarrow{\delta} Y(V)\xrightarrow{\epsilon\tensor\id_V} V$.  
Denote by $Y'$ the subspace of $Y(V)$ spanned by 
$1\tensor v$ for $v\in V$ and by $s\tensor \check\alpha_s$ for  $s\in
\PR$. It follows from Lemma \ref{lem:roots} and the
definition of the $G$\dash action on $Y(V)$ that $Y'$ is 
a submodule of $Y(V)$; obviously, $Y'$ is a subcomodule of $V$,
hence a Yetter\dash Drinfeld submodule. 
Observe that $\delta_{t,c}(V)\subset Y'$, hence
$(V,\delta_{t,c})$ is a perfect subquotient of $Y'$. 

We identify $\Ycr(G)$ with $Y'$ via a linear isomorphism $i\colon
\Ycr(G) \to Y'$, defined by 
$i(v\oplus\oplus_s a_s [s])
= 1\tensor v + \sum_s a_s s\tensor \check\alpha_s$. The $G$\dash
action on $\Ycr(G)$ is so defined that $i$ is a $G$\dash
module map; $i$ preserves  the $G$\dash grading, hence $i$ is a
Yetter\dash Drinfeld module isomorphism. It remains to note that
$\mu=i^{-1}\circ \delta$ and $\nu=(\epsilon\tensor \id_V)\circ i$,
thus by Remark \ref{rem:comp_perfect} $(V,\delta_{t,c})$ is a perfect
subquotient of $\Ycr(G)$ via the maps $\mu$, $\nu$. 
\end{proof}

\subsection{Embedding of $H_{t,c}(G)$ in a braided Heisenberg double}
\label{embed_bhd}

We assume that the set $\PR$ is finite, so
that the Yetter\dash Drinfeld module $\Ycr(G)$ is finite\dash
dimensional. 

Observe
that a complex reflection $s\in G$ acting on $V^*$ has 
$\check\alpha_s$ as the root and $\alpha_s$ as the coroot. 
The right Yetter\dash Drinfeld module $\RYcr(G)^*$ will be spanned by
by symbols $[s]^*$
which are a basis dual to
$[s]\in \RYcr(G)$. 
We have $\Ycr(G)^*=V_1^*\oplus \RYcr(G)^*$ where $V_1^*$ coincides
with $V^*$
as a right $G$\dash module  and has trivial right $G$\dash coaction.

We are ready to construct a triangular map from the rational Cherednik
algebra $H_{t,c}(G)$ into the braided Heisenberg double
$\RHD_{\Ycr(G)}$.  

\begin{proposition}
\label{prop:embed}
Let $G\le \mathit{GL}(V)$ be a finite linear  group, and 
let $(t,c)$ be a parameter from $\field \times
\field(\PR)^G$. 
The maps 
\begin{align*}
&     \mu \colon V \to \Ycr(G), 
     \quad \mu(v)=t v \oplus
\mathop{\oplus}\limits_{s\in\PR} c_s \langle \alpha_s,
v\rangle [s], 
\\
&
\nu^*\colon V^* \to \Ycr(G)^*, 
\quad
\nu^*(f) = f \oplus
\mathop{\oplus}\limits_{s\in\PR} \langle f, \check\alpha_s
\rangle 
[s]^*
\end{align*}
induce a morphism of braided doubles 
from the rational Cherednik
algebra $H_{t,c}(G)$ 
to the braided Heisenberg double 
$\mathcal B(\Ycr(G)) \lcprod \field G \rcprod \mathcal
B(\Ycr(G)^*)$. 
If $t\ne 0$ and $\mathit{char}\ \field=0$, this morphism is injective.  
\end{proposition}
\begin{proof}
One should note that the map denoted by $\nu^*$ in the Proposition is indeed 
adjoint to the map $\nu\colon \Ycr(G)\to V$ defined in
Lemma~\ref{lem:perf}. By Lemma~\ref{lem:perf}, 
the maps $V \xrightarrow{\mu} \Ycr(G)  \xrightarrow{\nu} V$ make the
quasi\dash Yetter\dash Drinfeld module $(V,\delta_{t,c})$ a perfect
subquotient of $\Ycr(G)$. 

In a fashion similar to Lemma~\ref{lem:perf} we can show that the
right quasi\dash YD module $V^*$ is a perfect subquotient of
$\Ycr(G)^*$ via the maps $V^* \xrightarrow{\nu^*} \Ycr(G)^*
\xrightarrow{\mu^*} V^*$.  

Since both pairs of maps $\mu,\nu$ and $\nu^*,\mu^*$ are perfect
subquotients, the situation is now precisely as described in 
\ref{lr}. That is, the maps $\mu$ and $\nu^*$ induce an injective
embedding of the minimal double associated to $(V,\delta_{t,c})$
(which itself is a triangular quotient of $H_{t,c}(G)$, and
coincides with $H_{t,c}(G)$  when $t\ne 0$ and $\mathit{char}\ \field=0$)
in the braided Heisenberg double $\mathcal{H}_{\Ycr(G)}$. 
\end{proof}

\subsection{Braided doubles containing rational Cherednik algebras of
  complex reflection groups}

From now on we assume that
\begin{itemize}
\item 
 the field $\field$ is $\mathbb{C}$;
\item
 $G$ is an irreducible finite  complex reflection group;
\item
 $V$ is the reflection representation of $G$.
\end{itemize}
We will now use general results up to and including Proposition
\ref{prop:embed} to construct two  
interesting realisations of rational Cherednik algebras of
complex reflection groups: one for $t=0$, the other for $t\ne 0$. 
Both realisations are in braided doubles
associated to the Yetter\dash Drinfeld module $Y_G:=\RYcr(G)$ (see 
Definition \ref{def:RYcr}), which looks ``nicer'' than $Y_\PR(G)$
because it does not contain an extra copy of the space $V$. 
Thus, $Y_G$ is a vector space 
spanned by symbols $\{[s] : s\in \PR\}$, 
with $G$\dash action $g([s])=\lambda(g,s) [gsg^{-1}]$ and braiding 
given by $\Psi([r]\tensor [s]) = r([s])\tensor [r]$ for $r,s\in \PR$.
We will illustrate the constructions by examples for $G=\Symm_n$. 

\subsection{The case $t=0$.}

Let $\mathcal H_{Y_G}$ be the braided Heisenberg double associated to
the Yetter\dash Drinfeld module $Y_G$. This is an algebra with generators 
$[s]$, $[s]^*$ ($s\in \PR$) and $g\in G$. 

To understand what are the relations between these generators in
$\mathcal H_{Y_G}$, 
denote by $I(Y_G)$ the ideal of relations between the generators $[s]$
in the Nichols\dash Woronowicz algebra $\mathcal B(Y_G)$. 
It is the kernel of the Woronowicz symmetriser $\Wor(\Psi)$.
Similarly, let $I(Y_G^*)$ be the ideal of relations between the
generators $[s]^*$ in $\mathcal B(Y_G^*)$. 
The defining relations in $ \mathcal H_{Y_G}$ are 
\begin{gather*}
      I(Y_G), \quad I(Y_G^*), \quad 
      g\cdot [s]=g([s])\cdot g, \quad  g\cdot [s]^*=g([s]^*)\cdot g,
\\
 [r]^*\cdot [s] - [s]\cdot [r]^* = 
\Bigl\{\begin{matrix}&s,\text{ if $r=s$}, 
\\ &0,\text{ if $r\ne s$}.\end{matrix}\Bigr.
\end{gather*}
We will now embed the restricted Cherednik algebra
$\overline{H}_{0,c}(G)$ 
in the braided Heisenberg double $\mathcal H_{Y_G}$.
\begin{theorem}
\label{th:emb0}
Let $G$ be an irreducible complex reflection group, $\PR\subset G$ be
the set  of complex reflections, and $c$ be a conjugation\dash
invariant function on $\PR$. 
Define a linear map $M_c$ from $V\oplus \mathbb{C}G \oplus V^*$ 
to $Y_G\oplus \mathbb{C}G \oplus Y_G^*$ by 
\begin{equation*}
M_c(v)=\sum_{s\in S}c_s \langle \alpha_s,v\rangle [s],
\quad 
M_c(f)=\sum_{s\in S} \langle f,\check\alpha_s\rangle [s]^*,
\quad
M_c(g)=g,
\end{equation*}
where $v\in V$, $f\in V^*$, $g\in G$. Then $M_c$ extends to an 
algebra homomorphism 
$M_c\colon \overline{H}_{0,c}(G) \to \mathcal H_{Y_G}$. 
If $c$ is generic, for example $c_s\ne 0$ for all $s$, 
the homomorphism $M_c$ is injective.
\end{theorem}
\begin{proof}
Proposition \ref{prop:embed} gives a triangular map
$H_{0,c}(G) \to \mathcal
H_{\Ycr(G)}$ where $\Ycr(G)=V_1\oplus Y_G$ (direct sum of
Yetter\dash Drinfeld modules), and that triangular map is defined by the same formulas
as $M_c$. Moreover, since $t=0$, the image of that triangular map
lies in the subdouble of $\mathcal H_{\Ycr(G)}$ generated by $Y_G$ and
$Y_G^*$. This subdouble is precisely $\mathcal H_{Y_G}$. The image of 
$H_{0,c}(G)$ in  the braided Heisenberg double $\mathcal H_{Y_G}\cong
\mathcal B(Y_G)\lcprod 
\mathbb{C}G \rcprod  \mathcal B(Y_G^*)$
is the minimal double which is the quotient of $H_{0,c}(G)$; for
generic $c$, this is $\overline{H}_{0,c}(G)$ by Proposition \ref{prop:min_restr}. 
\end{proof}
\begin{remark}
The Theorem implies (and provides a new proof of)
the realisation of the coinvariant algebra $S(V)_G$ in the
Nichols\dash Woronowicz algebra $\mathcal B(Y_G)$, which was the main
result of \cite{B} when $G$ is a Coxeter group, and of \cite{KM} when
$G$ is an arbitrary complex reflection group.
\end{remark}

Now observe that for any algebra $A=\Uminus \lcprod H \lcprod \Uplus$ with
triangular decomposition over a bialgebra $H$, the subalgebra
$\Uminus$ is a left $A$\dash module (it is the induced module
$M_\epsilon$ in the notation of~\ref{standmod}, where $\epsilon$ is
the counit viewed as a $1$\dash dimensional representation of $H$).
In particular, the  Nichols\dash Woronowicz algebra
$\mathcal B(Y_G)$ is a left module for the braided Heisenberg double
$\mathcal{H}_{Y_G}$. Since $\overline{H}_{0,c}$ embeds or maps into
$\mathcal{H}_{Y_G}$, we have 
\begin{corollary}
\label{prop:cher_action}
The Nichols\dash Woronowicz algebra $\mathcal B(Y_G)$ is a module over
the restricted Cherednik algebra $\overline{H}_{0,c}$, for any $c$.
\qed
\end{corollary}

\begin{example}
Let us consider the case when $G=\Symm_n$ is a
symmetric group acting on $\mathbb{C}^n$.
The action
restricts onto the subspace $V=\{(x_1,\dots,x_n) : \sum x_i=0\}$ which
is the irreducible reflection representation of $\Symm_n$. 
We have the restricted Cherednik algebra $\overline{H}_{0,c}(\Symm_n)$ for any
$c\in \mathbb{C}$; all such algebras are isomorphic for $c\ne0$.

The reflections in $\Symm_n$ are transpositions $(ij)$, $1\le i<j\le
n$. The Nichols\dash Woronowicz algebra $\mathcal B(Y_{\Symm_n})$ has
generators $[ij]$. 
The $\Symm_n$\dash action on generators is via $g([ij])=[g(i)g(j)]$
for $g\in \Symm_n$, if we agree that the symbol $[ij]$ stands for
$-[ji]$ whenever $i>j$. 
The quadratic relations in $\mathcal B(Y_{\Symm_n})$ are
\begin{equation*}
    [ij][kl]=[kl][ij], \quad [ij][jk]+[jk][ki]+[ki][ij]=0
\quad\text{for distinct $i,j,k,l$},
\end{equation*}
so that the quadratic cover $\mathcal B_{\mathit{quad}}(Y_{\Symm_n})$
of $\mathcal B(Y_{\Symm_n})$ is the \emph{Fomin\dash Kirillov
  quadratic algebra} $\mathcal{E}_n$, introduced in \cite{FK}. 
Using our method and results from \cite{FK}, one can check that 
the restricted Cherednik algebra $\overline{H}_{0,c}(\Symm_n)$ acts both
on $\mathcal B(Y_{\Symm_n})$ and on $\mathcal{E}_n$.
\end{example}

It is a conjecture (which is at least ten years old at the time 
when the present paper is being written, and is still open) that the
Nichols\dash Woronowicz algebra 
$\mathcal B(Y_{\Symm_n})$ is quadratic and coincides with  $\mathcal{E}_n$.
The algebras $\mathcal{E}_n$ are finite\dash dimensional for $n\le
5$. Open questions about $\mathcal{E}_n$ include infinite\dash
dimensionality for $n>5$ and the Hilbert series. Overall, very little is
known about the structure of $\mathcal{E}_n$.

It therefore may
prove to be helpful to study the $\overline{H}_{0,c}(\Symm_n)$\dash action on
$\mathcal{E}_n$ and to use the representation theory of the restricted
Cherednik algebra to obtain information about the Fomin\dash Kirillov
algebra. 
For example, any simple $\overline{H}_{0,c}(\Symm_n)$\dash module, 
viewed as an $\Symm_n$\dash module, is a regular representation of
$\Symm_n$ \cite{Go}. Thus the very fact that
$\overline{H}_{0,c}(\Symm_n)$ acts on $\mathcal{E}_n$ implies 
a new result:

\begin{proposition}
As an $\Symm_n$\dash module, the Fomin\dash Kirillov algebra
$\mathcal{E}_n$ is a direct sum of copies of the regular
representation of $\Symm_n$.  
\qed
\end{proposition}

\subsection{The case $t\ne 0$.} 
When $t\ne 0$, the algebra $H_{t,c}(G)$ embeds, as a subalgebra, 
in a ``deformed quadratic double'' associated to $Y_G$. 

Note that $Y_G$ is a Yetter\dash Drinfeld module over $\mathbb{C}G$ which
is a cocommutative Hopf algebra. Hence there exists a deformed
Nichols\dash Woronowicz algebra $\mathcal{B}_\tau(Y_G)$. We will only
need its quadratic cover 
\begin{equation*}
{\mathcal{B}_{\mathit{quad}}}_\tau(Y_G) = T(Y_G) / {I_{\mathit{quad}}}_\tau(Y_G)
\end{equation*}
given in Lemma \ref{lem:quadcovers}, where the ideal 
${I_{\mathit{quad}}}_\tau(Y_G)$ is generated by $\ker(\id+\Psi)\cap
\wedge^2 Y_G$. The corresponding braided double is 
$\mathcal{H}_t(Y_G)$ with relations
\begin{gather*}
      {I_{\mathit{quad}}}_\tau(Y_G), \quad 
      {I_{\mathit{quad}}}_\tau(Y_G^*), \quad 
      g\cdot [s]=g([s])\cdot g, \quad  g\cdot [s]^*=g([s]^*)\cdot g,
\\
 [r]^*\cdot [s] - [s]\cdot [r]^* = 
\Bigl\{\begin{matrix}&s+t\cdot 1,\text{ if $r=s$}, 
\\ &0,\text{ if $r\ne s$}.\end{matrix}\Bigr.
\end{gather*}
\begin{theorem}
\label{th:emb}
Let $t\ne 0$ and $c$ be generic.
In the notation of Theorem \ref{th:emb0}, the map $M_c$ extends to an
injective algebra homomorphism $H_{t,c}(G) \to \mathcal{H}_{t'}(Y_G)$
for some $t'$ depending on $t$ and $c$. 
\end{theorem}
\begin{proof}
Computing the commutator of $M_c(f)$ and $M_c(v)$ in a braided double 
$\mathcal{H}_{t'}(Y_G)$, we obtain 
\begin{equation*}
  [M_c(f),M_c(v)] = \sum\nolimits_{s\in S} c_s \langle f, (1-s)v\rangle
  (s+t'\cdot 1).
\end{equation*}
Note that $\sum_s  c_s \langle f, (1-s)v\rangle$ is a $G$\dash
invariant pairing between $V^*$ and $V$, and because $V$ is
irreducible, this pairing is proportional to $\langle\cdot, \cdot
\rangle$ (and is non\dash zero for $c$ generic).

 Choose $t'$ in such a
way that $\sum_s  c_s \langle f, (1-s)v\rangle\cdot t' = t$. Then
$M_c$ extends to a morphism of braided doubles between $T(V)\lcprod 
\mathbb{C}\rcprod T(V^*)$  (with commutator given by the quasicoaction
$\delta_{t,c}$) and $\mathcal{H}_{t'}(Y_G)$. It follows from
Lemma \ref{lem:quadcovers} that $M_c(V)$ (resp.\ $M_c(V^*)$) generates
a commutative subalgebra in ${\mathcal{B}_{\mathit{quad}}}_\tau(Y_G)$
  (resp.\ ${\mathcal{B}_{\mathit{quad}}}_\tau(Y_G^*)$). 
Therefore, $M_c$ factors through the rational Cherednik algebra
$H_{t,c}(G)$, and is injective on this algebra because 
$H_{t,c}(G)$ is a minimal double. 
\end{proof}
\begin{remark}
In fact, by modifying the map $M_c$ on the space $V^*$, one can relax
 the condition that $c$ is generic and only assume that $c$ is not
 identically zero on $\PR$.
\end{remark}

\begin{example}
In the case $G=\Symm_n$, the quadratic algebra
${\mathcal{B}_{\mathit{quad}}}_\tau(Y_G)$ coincides with the universal
  enveloping algebra $U(\mathrm{tr}_n)$ associated to the classical
  Yang\dash Baxter equation and introduced in \cite{BEER}. 
It is a Koszul algebra. We thus obtain an action of $H_{t,c}(\Symm_n)$
on  $U(\mathrm{tr}_n)$ and a generalisation of
$U(\mathrm{tr}_n)$ for an arbitrary irreducible complex reflection
group.
\end{example}


\setcounter{subsection}{0}
\setcounter{theorem}{0}

\section*{Appendix: Triangular decomposition over a bialgebra} 
\renewcommand{\thesection}{A}

This Appendix contains proofs to a number of facts about
triangular ideals in braided doubles.
We do it in somewhat more general situation, 
for algebras with triangular decomposition over a bialgebra. 
Triangular decomposition of universal enveloping algebras of Kac\dash 
Moody algebras has been generalised in several useful ways
\cite{Ginzb_prim,K}.
Most relevant for the present paper is the following
definition
(all algebras are assumed to be associative and unital):

\begin{definition}
\label{def_triang_decomp}
An algebra $A$ has \emph{triangular decomposition 
over a bialgebra $H$} if $A$ has distinguished subalgebras
$H \hookrightarrow A$, $U^- \hookrightarrow A$, $U^+ \hookrightarrow A$ such
that:
\begin{itemize} 

\item 
$H$ acts covariantly from the  left on the algebra $U^-$, and
from the right on $U^+$; 

\item 

the multiplication in $A$ induces a linear isomorphism $\Uminus \tensor H
\tensor \Uplus \to A$; it makes $\Uminus \tensor H$ a subalgebra of
$A$ isomorphic to the semidirect product $\Uminus\lcprod
H$ by the left $H$\dash action, and similarly $ H \tensor \Uplus $ a
subalgebra isomorphic to the semidirect product $H\rcprod \Uplus$;

\item 
the algebras $\Uminus $, $\Uplus $ are equipped 
with $H$\dash  equivariant characters (algebra homomorphisms) 
$\epsilonminus \colon \Uminus \to
\field$, $\epsilonplus \colon \Uplus  \to \field$.  
\end{itemize}
\end{definition}

An $H$\dash equivariant character $\epsilon^\pm\colon U^\pm\to\field$ is  
a homomorphism of $H$\dash modules, where the action of $H$ on
$\field$ is via the counit $\epsilon$ of $H$. (These characters are required
to ensure that $A$ has a maximal triangular ideal 
--- see below.) Semidirect products are also
known as smash products \cite[Definition 4.1.3]{Mon}.

\subsection{Triangular subalgebras, ideals and quotients. 
Triangular-simple algebras}

\label{triang_simple}

Let $H$ be a bialgebra.
Algebras $A\cong \Uminus \tensor H \tensor \Uplus$ with triangular
decomposition over $H$ form a 
category. Morphisms in this category are algebra maps of the form
\begin{equation*}
\mu = \mu^-\tensor \id_H \tensor \mu^+ \colon 
\Uminus_1 \tensor H \tensor \Uplus_1 \to 
\Uminus_2 \tensor H \tensor \Uplus_2,
\end{equation*}
where $\mu^-\colon U_1^- \to U_2^-$ (resp.\ $\mu^+\colon U_1^+ \to
U_2^+$)  is a left (resp.\ right) $H$\dash module algebra
homomorphism, which intertwines the characters $\epsilonminus$ (resp.\
$\epsilonplus$). 
Among morphisms  are embeddings of triangular subalgebras
\begin{equation*}
\Dbl' = {U'}^-\tensor H \tensor {U'}^+ \hookrightarrow
\Dbl = \Uminus \tensor H \tensor \Uplus,
\end{equation*}
where ${U'}^\pm$ embed in $U^\pm$ as subalgebras, and 
triangular quotient maps
\begin{equation*}
    \Dbl = \Uminus \tensor H \tensor \Uplus \twoheadrightarrow 
    \Dbl'' = {U''}^-\tensor H \tensor {U''}^+,
\end{equation*} 
induced by pairs $U^\pm \twoheadrightarrow {U''}^\pm$ 
of surjective algebra maps.

A kernel of a triangular quotient map will be called a \emph{triangular
ideal}. 
We say that
$\Dbl = \Uminus \tensor H \tensor \Uplus$ is a \emph{triangular\dash simple
algebra} over $H$, if $\Dbl$ has no non\dash
trivial triangular quotients.

The next Proposition describes triangular ideals.

\begin{proposition}
\label{triang_ideals}
Let $\Dbl = \Uminus \tensor H \tensor \Uplus$ be an algebra with
triangular decomposition.  
A linear subspace $J\subset \Dbl$ is a triangular ideal
in $\Dbl$, if and only if 
\begin{equation*}
           J = J^- \tensor H \tensor \Uplus + \Uminus \tensor H
           \tensor J^+,
\end{equation*}
where $J^-$ (resp.\ $J^+$) is a two\dash sided ideal in $\Uminus$
(resp.\ $\Uplus$), satisfying the following:
\begin{itemize}
\item[$(\ref{triang_ideals}.1)$] 
$J^-$, $J^+$ are invariant with respect to the $H$\dash action on
$\Uminus$, $\Uplus$;
\item[$(\ref{triang_ideals}.2)$]
$J^- \subset \ker \epsilonminus$, $J^+ \subset \ker \epsilonplus$;
\item[$(\ref{triang_ideals}.3)$]
$\Uplus \cdot J^-$ and $J^+\cdot \Uminus$  (the products with respect
to the multiplication in the algebra $\Dbl$) lie in 
$J^-\tensor H \tensor \Uplus+\Uminus \tensor H \tensor J^+$.
\end{itemize}
\end{proposition} 
\begin{proof}
We start with the `only if' part.
Let $J$ be a triangular ideal, i.e., the kernel of a map 
\begin{equation*}
     \mu^-\tensor \id_H \tensor \mu^+ \colon \Uminus \tensor H \tensor
\Uplus \to {U'}^- \tensor H \tensor {U'}^+,
\end{equation*}
 which is a morphism of
algebras with triangular decomposition over $H$. 
Then $J=J^-\tensor H \tensor\Uplus+\Uminus\tensor H \tensor J^+$
where $J^\pm = \ker \mu^\pm$. 
Since $\mu^-$ is an algebra morphism, $J^-$ is a two\dash sided
ideal in $U^-$. Furthermore, $J^-$ is $H$\dash invariant because 
$\mu^-$ is an $H$\dash morphism, and $J^-\subset\ker \epsilonminus$ 
because $\epsilonminus|_{\vphantom{{U'}^-}{U}^-} = \epsilonminus|_{{U'}^-}\circ \mu^-$. 
This verifies properties $(\ref{triang_ideals}.1)$,
$(\ref{triang_ideals}.2)$ for the ideal $J^-$, and they are verified
for $J^+$ in the same way. 
Property $(\ref{triang_ideals}.3)$ follows from the fact that
$J^-=J^-\tensor 1 \tensor 1$ and $J^+=1\tensor 1 \tensor J^+$ lie in
$J$ which is a two\dash sided ideal in $\Dbl$.

To prove the `if' part of the Proposition,
we must show that if two\dash sided ideals $J^\pm\subset U^\pm$
satisfy $(\ref{triang_ideals}.1)$--$(\ref{triang_ideals}.3)$, then
$J=J^- \tensor H \tensor \Uplus + \Uminus \tensor H \tensor J^+$ is a
triangular ideal in $\Dbl$.
We begin by checking that $J$ is a two\dash sided ideal in $\Dbl$.
First,
\begin{equation*}
\Uminus \cdot J \subset \Uminus \cdot J^- H \Uplus
+ \Uminus \cdot \Uminus H J^+
= J^- H \Uplus +  \Uminus H J^+ = J
\end{equation*}
because $J^-$ is a left ideal in $\Uminus$.
Similarly, $J \cdot \Uplus \subset J$. Next,
\begin{equation*}
 H \cdot J \subset H \cdot J^- H \Uplus + H \cdot \Uminus H J^+
\subset J^- H \Uplus + \Uminus H J^+ = J,
\end{equation*}
as $h\cdot J^- = (h_{(1)}\act J^-) h_{(2)} \subset J^-\cdot H$ by 
$(\ref{triang_ideals}.1)$ for any $h\in H$.
Similarly, $J\cdot H \subset J$.
Finally, using $(\ref{triang_ideals}.3)$, 
\begin{equation*}
\Uplus \cdot J \subset 
(\Uplus \cdot J^-) H \Uplus + (\Uplus \cdot \Uminus H )J^+
\subset J H \Uplus + (\Uminus H \Uplus)J^+;
\end{equation*}
we have already established that $J H \Uplus\subset J$, 
and $(\Uminus H \Uplus)J^+ = \Uminus H J^+$ because $J^+$ is a two\dash
sided ideal in $\Uplus$. Thus, $\Uplus \cdot J \subset J$. Quite
similar argument shows that $J\cdot \Uminus\subset J$.
The subalgebras $\Uminus$, $H$ and $\Uplus$ generate $\Dbl$ as
an algebra, hence $\Dbl \cdot J$, $J\cdot\Dbl\subset J$
as required.

Let $p\colon \Dbl \twoheadrightarrow \Dbl / J$ be the quotient map of
algebras. 
Our goal is to show that $\Dbl/J$ is an algebra with
triangular decomposition over $H$, and $p$ is a triangular morphism 
(then $J=\ker p$ is, by definition, a triangular ideal).

Since $J=J^- \tensor H \tensor \Uplus + \Uminus \tensor H
\tensor J^+$, we have a vector space tensor product decomposition
\begin{equation*}
       \Dbl/J = (\Uminus /J^-) \tensor H \tensor  (\Uplus /J^+),
\end{equation*}
and $p=p^-\tensor \id_H \tensor p^+$ 
where $p^\pm\colon U^\pm \twoheadrightarrow U^\pm/J^\pm$ are quotient maps. 
We observe that $U^\pm/J^\pm$ and $H$ are subalgebras in $\Dbl/J$
(these are $p$\dash images of $U^\pm$ and $H$, respectively).
Moreover,
by $(\ref{triang_ideals}.1)$, one has the induced $H$\dash action on
$U^\pm/J^\pm$, and $p^\pm$ are $H$\dash algebra homomorphisms.
The relation between $\bar b=p^-(b)\in \Uminus/J^-$ and $h\in H$,
\begin{equation*}
     h\cdot \bar b = p(h\cdot b) = 
    p((h_{(1)}\act b)\, h_{(2)}) = 
    (h_{(1)}\act \bar b )  \, h_{(2)},
\end{equation*}
is the cross\dash product relation between $\Uminus/J^-$ and $H$.
Similarly for $H$ and $\Uplus/J^+$.
And finally, by $(\ref{triang_ideals}.2)$, there are induced characters
$\epsilon^\pm \colon {U}^\pm/J^\pm \to \field$, such that $p^-$
intertwines the characters $\epsilonminus$ on $\Uminus$ and
${U}^-/J^-$ (similarly for $p^+$). 
Thus, $\Dbl/J$ is indeed an algebra with triangular decomposition over
$H$, and $p\colon \Dbl\twoheadrightarrow \Dbl / J$ is a triangular morphism.
\end{proof}
The Proposition and its proof have the following important corollary:
\begin{corollary}
\label{cor_sum}
Let $\Dbl$ be an algebra with triangular decomposition over a
bialgebra~$H$.

1. If $J$ is a triangular ideal in $\Dbl$, the quotient algebra
$\Dbl/J$ has triangular decomposition over $H$.

2. A surjective morphism of algebras with triangular decomposition
   over $H$ maps triangular ideals to triangular ideals. 

3. A sum of triangular ideals in $\Dbl$ is a triangular ideal in $\Dbl$. 

4. The algebra $\Dbl$ has a greatest triangular ideal $I_\Dbl$.

5. The algebra $\Dbl$ has a unique triangular\dash simple quotient
   over $H$, which    is $\Dbl/I_\Dbl$.
\end{corollary}
\begin{proof}
1. In the proof of the `if' part of 
   the Proposition, the quotient $\Dbl/J$ was explicitly constructed
   and shown to have triangular decomposition.

2. Let $\mu=\mu^-\tensor \id_H \tensor \mu^+$ be a triangular morphism
from $\Dbl=\Uminus \tensor H \tensor \Uplus$ onto $\Dbl'' = {U'}^-
\tensor H \tensor {U''}^+$.
It is easy to see that if $J^\pm\subset U^\pm$ are ideals satisfying
$(\ref{triang_ideals}.1)$--$(\ref{triang_ideals}.3)$, then 
$\mu^\pm(J^\pm)$ are ideals in ${U''}^\pm=\mu^\pm(U^\pm)$ which, too, satisfy 
$(\ref{triang_ideals}.1)$--$(\ref{triang_ideals}.3)$.

3. Let $\{J_\alpha\}$ be a family of triangular ideals in $\Dbl$.
Then for each index $\alpha$, $J_\alpha = J_\alpha^- \tensor H \tensor
\Uplus + \Uminus \tensor H \tensor J_\alpha^+$, where
$J_\alpha^\pm\subset U^\pm$ are a pair of two\dash sided
ideals satisfying $(\ref{triang_ideals}.1)-(\ref{triang_ideals}.3)$.
It is clear that $J^-=\sum_\alpha J^-_\alpha$,  $J^+=\sum_\alpha
J^+_\alpha$ is a pair of two\dash sided 
ideals in $\Uminus$, $\Uplus$ satisfying
$(\ref{triang_ideals}.1)-(\ref{triang_ideals}.3)$ (in particular,
$J^\pm\ne U^\pm$ because $J^\pm\subset \ker\epsilon^\pm$).
Thus, $J=\sum_\alpha J_\alpha = J^- H \Uplus + \Uminus H J^+$ is a
triangular ideal in $\Dbl$.

4. The ideal $I_\Dbl$ is the sum of all triangular ideals in $\Dbl$.

5. It follows from 2.\ that if a triangular ideal $J$ is not greatest
   in $\Dbl$, then $\Dbl/J$ has a non\dash trivial triangular ideal
   $I_\Dbl/J$, hence is not triangular\dash
   simple.
   It remains to observe that $\Dbl/I_\Dbl$ is triangular\dash simple,
   because a non\dash trivial triangular quotient map
   $\Dbl/I_\Dbl\twoheadrightarrow \Dbl''$ 
   would give rise to composite triangular quotient map
   $\Dbl\twoheadrightarrow \Dbl/I_\Dbl \twoheadrightarrow \Dbl''$
   whose kernel is a triangular ideal strictly larger than $I_\Dbl$.
\end{proof}

A key question about an algebra $\Dbl$ with triangular decomposition
over $H$ is to find its unique maximal triangular ideal $I_\Dbl$. 
We will now show that there is a natural ``upper
bound'' for $I_\Dbl$, given by kernels of the \emph{Harish\dash
  Chandra pairing} in $\Dbl$.

\subsection{The Harish-Chandra pairing}
\label{HCpairing}

Let $\Dbl  = \Uminus  \tensor H \tensor \Uplus $ be an algebra with
triangular decomposition over a bialgebra $H$. 
The \emph{Harish\dash Chandra projection map} is a 
linear map from $\Dbl $ onto $H$ defined as
\begin{equation*}
    \pr = \epsilonminus  \tensor \id_H \tensor \epsilonplus  
          \colon \Uminus  \tensor H \tensor \Uplus  
          \twoheadrightarrow H. 
\end{equation*}
The \emph{Harish\dash Chandra pairing} is an $H$\dash valued
bilinear pairing
 between $\Uplus $ and $\Uminus $:
\begin{equation*}
   (\cdot,\cdot)_H\colon \Uplus \times \Uminus \to H, 
\qquad 
  (\phi,b)_H = \pr(\phi b),
\end{equation*}
where the product of $\phi\in \Uplus $ and $b\in \Uminus $ is taken in
$\Dbl $.    

The Harish\dash Chandra projection map $\pr$ will in general not be an
algebra homomorphism. 
However, it is an $H$--$H$ bimodule map. 

\subsection{The kernels of the Harish-Chandra pairing}
\label{subsect:reduction}

Let $\Dbl = \Uminus  \tensor H \tensor \Uplus $ be any algebra
with triangular decomposition over $H$. Let 
\begin{equation*}
   \Iminus  = \{ b \in \Uminus  \mid (\phi,b)_H=0\ 
                 \forall  \phi\in \Uplus \},
\quad
   \Iplus   = \{ \phi \in \Uplus  \mid (\phi,b)_H=0\ 
                 \forall b \in \Uminus \}
\end{equation*}
be the kernels of the Harish\dash Chandra pairing in $\Uminus $ and
$\Uplus $.
We have the following 
\begin{proposition}
\label{HCkernels}
All triangular ideals in $\Dbl$ lie in $\Iminus \tensor H \tensor
\Uplus + \Uminus \tensor H \tensor \Iplus$.
\end{proposition}
\begin{proof}
A triangular ideal in $\Dbl$ is of the form $J^-\tensor H \tensor
\Uplus + \Uminus \tensor H \tensor J^+$, where $J^\pm$ are ideals in
$U^\pm$ satisfying
$(\ref{triang_ideals}.1)$--$(\ref{triang_ideals}.3)$. 
In particular, $(\ref{triang_ideals}.2)$ says that 
$J^\pm \subset \ker\epsilon^\pm$, therefore $J$ lies in the kernel of
the Harish\dash Chandra projection $\pr=\epsilonminus \tensor \id_H
\tensor \epsilonplus$; and $(\ref{triang_ideals}.3)$ says that 
$\Uplus \cdot J^- \subset J$, hence 
$(\Uplus,J^-)_H= \pr(\Uplus \cdot J^-) = 0$ and $J^-\subset \Iminus$. 
Similarly, $J^+\subset \Iplus$. 
\end{proof}

Thus, if the Harish\dash Chandra pairing in $\Dbl$ is non\dash
degenerate, the algebra $\Dbl$ is automatically triangular\dash
simple. The converse is not true (and has explicit counterexamples). 
Note that $\Iminus H \Uplus+\Uminus H \Iplus$ 
is not even guaranteed to be an ideal in~$\Dbl$.


\end{document}